\documentclass{article}
\usepackage[margin=1in,footskip=0.25in]{geometry}
\usepackage{amsthm,amsfonts,amsmath, amssymb,dsfont,color, multirow}
\usepackage{commath,enumerate, mathtools,float,graphicx}

\usepackage{bm,thmtools}
\usepackage{hyperref}
\usepackage[numbers]{natbib}
\usepackage{mathrsfs,soul}

\newtheorem{thm}{{Theorem}}[]

\newtheorem{lem}{Lemma}
\newtheorem{prop}{Proposition}

\newtheorem{rem}{Remark}
\newtheorem{defn}{Definition}
\newtheorem{fact}{Fact}

\usepackage{balance}
\usepackage[T1]{fontenc}

\usepackage{subcaption, multirow} 
\usepackage{makecell}

\newcommand{\e}[2][{}]{\mathbb{E}_{#1} \sbr{#2}}

\newcommand{\define}{\triangleq}

\newcommand{\UH}{\mathsf{H}}   %% upright H
\newcommand{\UT}{\mathsf{T}}   %% upright T

\newcommand{\E}{\mathbb{E}}

\renewcommand{\P}{\mathbb{P}}
\newcommand{\C}{\mathbb{C}}
\newcommand{\N}{\mathbb{N}} %% Use \gauss{mean}{covariance}
\newcommand{\R}{\mathbb{R}}
\renewcommand{\O}{\mathbb{O}}
\newcommand{\W}{\mathbb{N}_0}
\newcommand{\Tr}{\mathsf{Tr}}

\newcommand{\gauss}[2]{\mathcal{N}\left( #1,#2 \right)}

\newcommand{\explain}[2]{\overset{\text{\tiny{#1}}}{#2}}
\newcommand{\diag}[1]{\text{Diag}\left(#1\right)}

\newcommand{\Indicator}[1]{\mathbb{I}_ {#1}}
\newcommand{\gpdf}[2]{\psi(#1;#2)}
\newcommand{\ip}[2]{\langle {#1}, {#2} \rangle}

\renewcommand{\i}{\mathrm{i}}
\newcommand{\unif}[1]{\text{Unif} \left( #1 \right)}
\newcommand{\barB}{\overline{\bm B}}
\newcommand{\bern}{\mathsf{Bern}}
\newcommand{\altprod}{\mathcal{\bm A}}
\renewcommand{\part}[2]{\mathcal{P}_{#1}(#2)}
\newcommand{\blocks}{\mathcal{V}}
\newcommand{\cset}[1]{\mathcal{C}(#1)}
\newcommand{\op}{\mathsf{op}}
\newcommand{\trim}{\mathcal{T}}
\newcommand{\weightedG}[1]{\mathcal{G}(#1)}
\newcommand{\weightedGnum}[2]{\mathcal{G}_{\mathsf{#2}}(#1)}
\newcommand{\degree}{\mathsf{d}}
\newcommand{\matmom}[4]{\mathcal{M}(#1,#2,#3,#4)}
\newcommand{\limmom}[2]{{\mu}(#1,#2)}
\DeclareMathOperator*{\plim}{p-lim}
\newcommand{\bnomdistr}[2]{\mathsf{Binom}(#1,#2)}
\newcommand{\fr}{\mathsf{Fr}}
\DeclareMathOperator*{\var}{Var}
\DeclareMathOperator*{\polylog}{polylog}
\newcommand{\singleblks}[1]{\mathscr{S}(#1)}
\newcommand{\firstblk}[1]{\mathscr{F}(#1)}

\newcommand{\fblknum}[2]{\mathscr{F}_{#1}(#2)}
\newcommand{\lblknum}[2]{\mathscr{L}_{#1}(#2)}
\newcommand{\firstfnc}{\mathscr{F}}
\newcommand{\ffncnum}[1]{{\mathscr{F}_{#1}}}
\newcommand{\lfncnum}[1]{{\mathscr{L}_{#1}}}
\newcommand{\lastfnc}{\mathscr{L}}

\newcommand{\lastblk}[1]{\mathscr{L}(#1)}
\newcommand{\tbd}{\frac{1}{4}}
\newcommand{\labelling}[1]{\mathcal{L}_{\mathsf{#1}}}
\newcommand{\dlgraph}{\bm{\ell}_{k+1}^{\otimes 2}}
\newcommand{\coeff}{g}
\newcommand{\Coeff}{G}
\renewcommand{\bar}{\overline}

\newcommand{\barb}{\overline{\bm b}}
\newcommand{\hatb}{\hat{\bm b}}
\newcommand{\milad}[1]{{#1}}
\newcommand{\rishabh}[1]{{#1}}

\title{Universality of Linearized Message Passing for Phase Retrieval with Structured Sensing Matrices}
\author{Rishabh Dudeja and Milad Bakhshizadeh \\ Department of Statistics, Columbia University}

\begin{document}
\maketitle

\begin{abstract}In the phase retrieval problem one seeks to recover an unknown $n$ dimensional signal vector $\mathbf{x}$ from $m$ measurements of the form $y_i = |(\mathbf{A} \mathbf{x})_i|$, where $\mathbf{A}$ denotes the sensing matrix. Many algorithms for this problem are based on approximate message passing. For these algorithms, it is known that if the sensing matrix $\mathbf{A}$ is generated by sub-sampling $n$ columns of a uniformly random (i.e., Haar distributed) orthogonal matrix, in the high dimensional asymptotic regime ($m,n \rightarrow \infty, n/m \rightarrow \kappa$), the dynamics of the algorithm are given by a deterministic recursion known as the state evolution. For a special class of linearized message-passing algorithms, we show that the state evolution is universal: it continues to hold even when $\mathbf{A}$ is generated by randomly sub-sampling columns of the Hadamard-Walsh matrix, provided the signal is drawn from a Gaussian prior.  
\end{abstract}
\newpage
\tableofcontents
\newpage

%%%%%%%%%%%%%%%%%%%%%%%%%%%%%%%%%%%%%%%%%%%%%%%%%%%%%%%%%%%%%%
%%%%%%%%%%%%%%%%%% INTRODUCTION%%%%%%%%%%%%%%%%%%%%%%%%%%%%%%%
%%%%%%%%%%%%%%%%%%%%%%%%%%%%%%%%%%%%%%%%%%%%%%%%%%%%%%%%%%%%%%
\section{Introduction}
In the phase retrieval one observes magnitudes of $m$ linear measurements (denoted by $y_{1:m}$) of an unknown $n$ dimensional signal vector $\bm x$:
\begin{align*}
    y_i & = |(\bm A \bm x)_{i}|,
\end{align*}
where $\bm A$ is a $m \times n$ sensing matrix. The phase retrieval problem is a mathematical model of imaging systems which are unable to measure the phase of the measurements. Such imaging systems arise in a variety of applications such as electron microscopy, crystallography, astronomy and optical imaging \citep{shechtman2015phase}.

Theoretical analyses of the phase retrieval problem seek to design algorithms to recover $\bm x$ (up to a global phase) with the minimum number of measurements. The earliest theoretical analysis modelled the sensing as a random matrix with i.i.d. Gaussian entries and designed computationally efficient estimators which recover $\bm x$ with information theoretically rate-optimal $O(n)$ (or nearly optimal $m=O(n\polylog(n))$) measurements. A representative, but necessarily incomplete, list of such works includes the analysis of convex relaxations like PhaseLift due to \citet{candes2013phaselift,candes2014solving}, PhaseMax due to \citet{bahmani2017phase,goldstein2018phasemax}, and analysis of non-convex optimization based methods due to \citet{netrapalli2013phase}, \citet{candes2015wirtinger}, and \citet{sun2018geometric}. The number of measurements required if the underlying signal has a low dimensional structure has also been investigated \citep{cai2016optimal,bakhshizadeh2017using,hand2018phase}.

Unfortunately, i.i.d. Gaussian measurements are not realizable  in practice;  instead, the sensing matrix is usually a variant of the Discrete Fourier Transform (DFT) matrix \citep{bendory2017fourier}.  Hence,  there have been efforts to extend the theory to structured sensing matrices \citep{alexeev2014phase,bandeira2014phase,candes2015phasematrix,candes2015phase,gross2015partial,gross2017improved}. A popular structured sensing ensemble is the Coded Diffraction Pattern (CDP) ensemble introduced by \citet{candes2015phasematrix} which is intended to model applications where it is possible to randomize the image acquisition by introducing random masks in front of the object. In this setup, the sensing matrix is given by:
\begin{align*}
    \bm A_{\mathsf{CDP}} & = \begin{bmatrix} \bm F_n \bm D_{1} \\ \bm F_{n} \bm D_{2} \\\vdots \\ \bm F_{n} \bm D_{L} \end{bmatrix},
\end{align*}
where $\bm F_n$ denotes the $n \times n$ DFT matrix and $\bm D_{1:L}$ are random diagonal matrices representing masks:
\begin{align*}
    \bm D_{\ell} & = \diag{e^{\i \theta_{1,\ell}}, e^{\i \theta_{2,\ell}}, \cdots, e^{\i \theta_{n,\ell}}},
\end{align*}
and $e^{\i \theta_{j,\ell}}$ are random phases. For the CDP ensemble convex relaxation methods like PhaseLift \citep{candes2015phase} and non-convex optimization based methods \citep{candes2015wirtinger} are known to recover the signal $\bm x$ with the near optimal $m=O(n \polylog(n))$ measurements. Another common structured sensing model is the sub-sampled Fourier sensing model where the sensing matrix is generated as:
\begin{align*}
    \bm A_{\mathsf{DFT}} & = \bm F_{m} \bm P \bm S,
\end{align*}
where $\bm F$ is the $m \times m$ Fourier matrix,  $\bm P$ is a uniformly random $m \times m $ permutation matrix and $\bm S$ the matrix that selects the first $n$ columns of an $m \times m$ matrix:
    \begin{align} \label{eq: S_def}
        \bm S & = \begin{bmatrix} \bm I_n \\ \bm 0_{m-n,n} \end{bmatrix}.
    \end{align}
This models a common oversampling strategy to ensure injectivity \citep{fannjiang2020numerics}. We also refer the reader to the recent review articles \citep{krahmer2014structured,bendory2017fourier,elser2018benchmark,fannjiang2020numerics} for more discussion regarding good models of practical sensing matrices. 

The aforementioned finite sample analyses show that a variety of different methods succeed in solving the phase retrieval problem with the optimal or nearly optimal order of magnitude of measurements. However, in practice, these methods can have a vast difference in performance, which is not captured by the non-asymptotic analyses. Consequently, efforts have been made to complement these results with sharp high dimensional asymptotic analyses which shed light on the performance of different estimators and information theoretic lower bounds in the high dimensional limit $m,n \rightarrow \infty, \; n/m \rightarrow \kappa$. This provides a high resolution framework to compare different estimators based on the critical value of $\kappa$ at which they achieve non-trivial performance ( i.e. better than a random guess) or exact recovery of $\bm x$. Comparing this to the critical value of $\kappa$ required information theoretically allows us to reason about the optimality of known estimators.  This research program has been executed, to varying extents, for the following unstructured sensing ensembles:
\begin{enumerate}
    \item Gaussian Ensemble: In this ensemble the entries of the sensing matrix are assumed to be i.i.d. Gaussian (real or complex). This is the most well studied ensemble in the high dimensional asymptotic limit. For this ensemble, precise performance curves for spectral methods \citep{yuelu_spectral,mondelli2019fundamental,yuelu_optimal}, convex relaxation methods like PhaseLift \citep{abbasi2019universality} and PhaseMax \citep{dhifallah2018phase}, and a class of iterative algorithms called Approximate Message Passing \citep{bayati2011dynamics}  are now well understood. The precise asymptotic limit of the Bayes risk \citep{barbier2019optimal} for Bayesian phase retrieval is also known.
    %These results provide the following sharp picture for the real phase retrieval problem: When $\kappa > 2$, no estimator works better than a random guess \citep{mondelli2019fundamental,barbier2019optimal}. When $\kappa < 2$ \citep{yuelu_spectral,yuelu_optimal,mondelli2019fundamental} optimally designed spectral methods have a non-trivial performance. When $\kappa > 1$ exact recovery is information theoretically impossible \citep{barbier2019optimal} and when $\kappa < 0.87$ the state of the art algorithm based on Approximate Message Passing achieves exact recovery \citep{barbier2019optimal}. 
    \item Sub-sampled Haar Ensemble: \milad{Let $\mathbb U(m)$ and $\mathbb O(m)$ denote the group of unitary and orthogonal matrices of size $m$, respectively.} In the sub-sampled Haar sensing model, the sensing matrix is generated by picking $n$ columns of a uniformly random orthogonal (or unitary) matrix at random:
    \begin{align*}
        \bm A_{\mathsf{Haar}} & = \bm O \bm P \bm S, 
    \end{align*}
    where $\bm O \sim \unif{\mathbb U(m)}$ (or $\bm O \sim \unif{\mathbb O(m)}$ in the real case) and $\bm P$ is a uniformly random $m \times m$ permutation matrix and $\bm S$ is the matrix defined in \eqref{eq: S_def}.
    The sub-sampled Haar model captures a crucial aspect of sensing matrices that arise in practice: namely they have orthogonal columns (note that for both the CDP and the sub-sampled Fourier ensembles we have $\bm A_{\mathsf{DFT}}^\UH \bm A_{\mathsf{DFT}} = \bm A_{\mathsf{CDP}}^\UH \bm A_{\mathsf{CDP}} = \bm I_n$). For the complex-valued sub-sampled Haar sensing model it has been shown that when $\kappa > 0.5$ no estimator performs better than a random guess \citep{dudeja2019information}. 
    Moreover, it is known that spectral estimators can achieve non-trivial performance when $\kappa < 0.5$ \citep{ma2019spectral,dudeja2020analysis}.
    \item Rotationally Invariant Ensemble: This is a broad class of unstructured sensing ensembles that include the Gaussian Ensemble and the sub-sampled Haar ensemble as special cases. Here, it is assumed that the SVD of the sensing matrix is given by:
    \begin{align*}
        \bm A & = \bm U \bm S \bm V^\UT,
    \end{align*}
    where $\bm U, \bm V$ are independent and uniformly random orthogonal matrices (or unitary in the complex case): $\bm U \sim \unif{\mathbb{O}(m)}, \; \bm V \sim \unif{\mathbb{O}(n)}$ and $\bm S$ is a deterministic matrix such that the empirical spectral distribution of $\bm S^\UT \bm S$ converges to a limiting measure $\mu_{S}$. The analysis of Approximate Message Passing algorithms has been extended to this ensemble \citep{schniter2016vector,rangan2019vector}.  For this ensemble, the non-rigorous replica method from statistical physics can be used to derive conjectures regarding the Bayes risk and performance of convex relaxations as well as spectral methods \citep{takeda2006analysis,takeda2007statistical,kabashima2008inference}. Some of these conjectures have been proven rigorously in some special cases \citep{barbier2018mutual,maillard2020phase}. 
\end{enumerate}

% \red{It is difficult to extend the techniques used for obtaining the above results to structured sensing matrices.}
The techniques used to prove the above results rely heavily on the rotational invariance of the underlying matrix ensembles. This makes it difficult to extend these results to structured sensing matrices.
% \milad{Rishabh, I think this does note convey any message unless we point out what difficulties are!  We can either remove this sentence or support it by providing some details.}

However, numerical simulations reveal an intriguing universality phenomenon: It has been observed that the performance curves derived theoretically for sub-sampled Haar sensing provide a nearly perfect fit to the empirical performance on practical sensing ensembles like $\bm A_{\mathsf{CDP}}, \bm A_{\mathsf{DFT}}$. This has been observed by a number of authors in the context of various signal processing problems. It was first pointed out by \citet{donoho2009observed} in the context of $\ell_1$ norm minimization for noiseless compressed sensing and then again by \citet{monajemi2013deterministic} for the same setup but for many more structured sensing ensembles. For noiseless compressed sensing both the Gaussian ensemble and the Sub-sampled Haar ensemble lead to identical predictions (and hence the simulations with structured sensing matrices match both of them). However, in noisy compressed sensing, the predictions from the sub-sampled Haar model and the Gaussian model are different. \citet{oymak2014case} pointed out that structured ensembles generated by sub-sampling deterministic orthogonal matrices empirically behave like Sub-sampled Haar sensing matrices. More recently, \citet{abbara2019universality} have observed this universality phenomenon in the context of approximate message passing algorithms for noiseless compressed sensing. \rishabh{In the context of phase retrieval, this phenomenon was reported by \citet{ma2019spectral} for the performance of the spectral method and by \citet{maillard2020phase} for the performance of the Approximate Message Passing algorithm of \citet{schniter2016vector}.}

\textbf{Our Contribution: } In this paper we study the real phase retrieval problem where the sensing matrix is generated by sub-sampling $n$ columns of the $m \times m$ Hadamard-Walsh matrix. Under an average case assumption on the signal vector, our main result (Theorem \ref{thm: main_result}) shows that the dynamics of a class of linearized Approximate message passing schemes for this structured ensemble are asymptotically identical to the dynamics of the same algorithm in the sub-sampled Haar sensing model in the high dimensional limit where $m,n$ diverge to infinity such that ratio $\kappa = n/m \in (0,1)$ is held fixed. This provides a theoretical justification for the observed empirical universality in this particular setup. In the following section we define the setup we study in more detail. 

% \milad{Note that partial Orthogonal matrices with $n$ orthonormal columns in $\mathbb{R}^m$ are, by definition, tall matrices, i.e. $n \leq m$, hence $\kappa = n / m \leq 1$.  Therefore, $\kappa \in (0, 1)$ covers all the proportional settings in which the problem is well-defined.  
% }

%Plan:
% \begin{enumerate}
%     \item Introduce the phase retrieval problem. 
%     \item describe theoretical work in high dimensional asymptotics for random sensing
%     \item say practical sensing is structured.
%     \item describe universality observations.
%     \item Introduce message passing algorithms. 
%     \item introduce linearized message passing
%     \item describe connection with spectral method.
%     \item say universality can't hold for abritary signals
% \end{enumerate}

%%%%%%%%%%%%%%%%%%%%%%%%%%%%%%%%%%%%%%%%%%%%%%%%%%%%%%%%%%%%%%
%%%%%%%%%%%%%%%%%%%%%%%%%%%%%%%%%%%%%%%%%%%%%%%%%%%%%%%%%%%%%%
%%%%%%%%%%%%%%%%%%%%%%%%%%%%%%%%%%%%%%%%%%%%%%%%%%%%%%%%%%%%%%

%%%%%%%%%%%%%%%%%%%%%%%%%%%%%%%%%%%%%%%%%%%%%%%%%%%%%%%%%%%%%%
%%%%%%%%%%%%%%%%%%%% SETUP %%%%%%%%%%%%%%%%%%%%%%%%%%%%%%%%%%%
%%%%%%%%%%%%%%%%%%%%%%%%%%%%%%%%%%%%%%%%%%%%%%%%%%%%%%%%%%%%%%
\subsection{Setup}

\subsubsection{Sensing Model}
 \rishabh{As mentioned in the introduction, we study the phase retrieval problem where the measurements $y_1,y_2, \dots y_m$ are given by:
\begin{align*}
    y_i & = (|\bm A \bm x|)_{i}.
\end{align*}
The matrix $\bm A$ is called the sensing matrix. We also define $\bm z \explain{def}{=} \bm A \bm x$ which we refer to as the signed measurements (which are not observed). The following 3 models for the sensing matrix $\bm A$ play a key role in this paper. In each of these models,  $\bm P$ is a uniformly random $m \times m$ permutation matrix and $\bm S$ is the selection matrix as defined in \eqref{eq: S_def}.}

\paragraph{Sub-sampled Hadamard Sensing Model} Assume that $m = 2^{\ell}$ for some $\ell \in \N$. In the sub-sampled Hadamard sensing model the sensing matrix is generated by sub-sampling $n$ columns of a $m \times m$ Hadamard-Walsh matrix $\bm H$ uniformly at random:
\begin{align}
    \label{hadamard_sensing}
    \bm A & = \bm H \bm P \bm S,
\end{align}
Recall that the Hadamard-Walsh matrix has a closed form formula: For any $i,j \in [m]$, let $\bm i, \bm j$ denote the binary representations of $i-1,j-1$. Hence, $\bm i, \bm j \in \{0,1\}^{\ell}$. Then the $(i,j)$-th entry of $\bm H$ is given by:
\begin{align} \label{eq: hadamard_formula}
    H_{ij} & = \frac{(-1)^{\ip{\bm i}{\bm j}}}{\sqrt{m}},
\end{align}
where $\ip{\bm i}{\bm j} = \sum_{k=1}^\ell i_k j_k$. \rishabh{It is well known that $\bm H$ is orthogonal, i.e. $\bm H^\UT \bm H = \bm I_m$. This sensing model can be thought of as a real-valued analog of the sub-sampled Fourier sensing model. It is an example of a structured sensing model for which is not covered by existing results and our primary goal will be to understand the dynamics of linearized approximate message passing algorithms (introduced below) for this sensing model. While our primary focus is the sub-sampled Hadamard sensing model, we believe our techniques should extend to structured sensing matrices with orthogonal columns, particularly those constructed by randomly sub-sampling other orthogonal matrices like the Discrete Fourier Transform (DFT) matrix and the Discrete Cosine Transform (DCT) matrix. A more detailed discussion regarding these extensions appears in the conclusion section (Section \ref{seq:conclusion}).}
\begin{rem} Some authors refer to any orthogonal matrix with $\pm 1$ entries as a Hadamard matrix. We emphasize that we claim results only about the Hadamard-Walsh construction given in \eqref{eq: hadamard_formula} and not arbitrary Hadamard matrices.
\end{rem}
\paragraph{Sub-sampled Haar Sensing Model} In this model the sensing matrix is generated by sub-sampling $n$ columns, chosen uniformly at random, of a $m \times m$ uniformly random orthogonal matrix:
\begin{align}
    \label{haar_sensing}
    \bm A & = \bm O \bm P \bm S,
\end{align}
where $\bm O \sim \unif{\O(m)}$. Existing theory applies to this sensing model and our goal will be to transfer these results to the sub-sampled Hadamard model.

\paragraph{Sub-sampled Orthogonal Model} This model includes both sub-sampled Hadamard and Haar models as special cases.  In this model the sensing matrix is generated by sub-sampling $n$ columns chosen uniformly at random of a $m \times m$  orthogonal matrix $\bm U$:
\begin{align}
    \label{orthogonal_sensing}
    \bm A & = \bm U \bm P \bm S,
\end{align}
where $\bm U$ is a fixed or random orthogonal matrix. Setting $\bm U = \bm O$ gives the sub-sampled Haar model and setting $\bm U = \bm H$ gives the sub-sampled Hadamard model. Our primary purpose for introducing this general model is that it allows us to handle both the sub-sampled Haar and Hadamard models in a unified way. Additionally, some of our intermediate results hold for any orthogonal matrix $\bm U$ whose entries are delocalized, and we wish to record that when possible.

In addition, we introduce the following matrices which will play an important role in our analysis:
\rishabh{\begin{enumerate}
    \item We define $\bm B \explain{def}{=} \bm P \bm S \bm S^\UT \bm P^\UT$. Observe that $\bm B$ is a random diagonal matrix with $\{0,1\}$ entries. It is easy to check that the distribution of $\bm B$ is described as follows: pick a uniformly random subset $S \subset [m]$ with $|S| = n$ and set:
\begin{subequations} \label{eq:B-barB-def}
    \begin{align}
        B_{ii} & = \begin{cases} 1 : & i \in S \\ 0 : & i \notin S \end{cases}.
    \end{align}
    \item Note that $\E \bm B = \kappa \bm I_m$. We define the zero mean random diagonal matrix $\barB \explain{def}{=} \bm B - \kappa \bm I_m$. Hence,
    \begin{align}
        \overline{B}_{ii} & = \begin{cases} 1 -\kappa : & i \in S \\ -\kappa : & i \notin S \end{cases}.
    \end{align}
\end{subequations}
    \item We define the matrix $\bm \Psi \explain{def}{=} \bm U \barB \bm U^\UT = \bm A \bm A^\UT - \kappa \bm I_m$.
\end{enumerate}
}

\begin{rem}
\rishabh{All the sensing ensembles introduced in this section have orthogonal columns, and hence,  make sense only when $n \leq m$ or equivalently $\kappa \in [0,1]$. We will additionally assume that $\kappa$ lies in the open interval $(0,1)$. The setting when the number of measurements $m$ is more than the dimension of the signal $n$ corresponds to the over-sampled regime, which is the natural regime to study unstructured phase retrieval problems, where the unknown signal is not assumed to have any low-dimensional structure (like sparsity). When the signal has some low-dimensional structure, like sparsity, it is interesting to study compressive phase retrieval where the number of measurements $m$ is less than the signal dimension $n$. In this situation, the interesting sensing ensembles would be those constructed by randomly sub-sampling rows of a deterministic or random orthogonal matrix. However, this paper focuses entirely on the over-sampled regime and unstructured signals.} 
\end{rem}

\subsubsection{Algorithm} We study a class of linearized message passing algorithms. This is a class of iterative schemes which execute the following updates:
\begin{subequations}\label{eq: LAMP_iteration_prelim}
\begin{align} 
        \hat{\bm z}^{(t+1)} &:= \left( \frac{1}{\kappa} \bm A \bm A^\UT - \bm I \right) \cdot \left( \eta_t(\bm Y) - \frac{\E\Tr(\eta_t(\bm Y))}{m} \bm I\right) \cdot \hat{\bm z}^{(t)}, \\
        \hat{\bm x}^{(t+1)} & := {\bm A^\UT \hat{\bm z}^{(t+1)}},
\end{align}
\end{subequations}
where
\begin{align*}
    \bm Y = \diag{y_1,y_2 \dots y_m},
\end{align*}
and $\eta_t: \R \rightarrow \R$ are  bounded Lipchitz functions that act entry-wise on the diagonal matrix $\bm Y$. \rishabh{The expectation in \eqref{eq: LAMP_iteration_prelim} is with respect to the randomness in $\bm y$. This randomness arises from two sources: (possible) randomness in the signal $\bm x$ and the randomness in the sensing matrix $\bm A$.} The iterates $(\hat{\bm z}^{(t)})_{t \geq 0}$ should be thought as estimates of the signed measurements $\bm z = \bm A \bm x$.  We now provide further context and motivation regarding the iteration in \eqref{eq: LAMP_iteration_prelim}.

\paragraph{Interpretation as Linearized AMP} \rishabh{Our primary motivation for studying the iteration \eqref{eq: LAMP_iteration_prelim} is that it is the simplest iterative scheme of interest to investigate the empirically observed universality phenomenon.} \rishabh{The iteration \eqref{eq: LAMP_iteration_prelim} can be thought of as a linearization of a broad class of non-linear approximate message passing algorithms introduced by \citet{schniter2016vector}.} These algorithms execute the iteration:
\begin{subequations}  \label{eq: NL_AMP}
\begin{align}
    \hat{\bm z}^{(t+1)}&:= \left( \frac{1}{\kappa} \bm A \bm A^\UT - \bm I \right) \cdot H_t(\bm y, \hat{\bm z}^{(t)}), \\
    \hat{\bm x}^{(t+1)} & := {\bm A^\UT \hat{\bm z}^{(t+1)}}.
\end{align}
\end{subequations}
where $H_t : \R^2 \rightarrow \R $ is a bounded Lipschitz function which satisfies the divergence-free property:
\begin{align}
    \frac{1}{m} \sum_{i=1}^m \E \partial_z H_t(y_i, \hat{z}^{(t)}_i) & = 0.
\end{align}
Indeed, if $H_t$ was linear in the second ($z$) argument (or was approximated by its linearization), one obtains the iteration in \eqref{eq: LAMP_iteration_prelim}. By appropriately choosing the function $H_t$ in the iteration, one can obtain the state-of-the-art performance for phase retrieval with sub-sampled Haar sensing. This algorithm achieves non-trivial (better than random) performance when $\kappa < 2/3$, and exact recovery when $\kappa < 0.63$ \citep{maillard2020phase}. \rishabh{Empirically, the universality phenomenon appears to be very general and also seems to hold for the non-linear iteration \ref{eq: NL_AMP} (see \citep[Figure 2]{maillard2020phase}). While our analysis currently does not cover the non-linear iteration \eqref{eq: NL_AMP}, we hope our techniques can be extended to analyze \eqref{eq: NL_AMP} in the future.}

\paragraph{Connection to Spectral Methods} Given that the algorithm we analyze \eqref{eq: LAMP_iteration_prelim} does not cover the state-of-the-art algorithm, one can reasonably ask what performance can one achieve with the linearized iteration \eqref{eq: LAMP_iteration_prelim}. It turns out that the iteration in \eqref{eq: LAMP_iteration_prelim} can implement a popular class of spectral methods which estimates the signal vector $\bm x$ as proportional to the leading eigenvector of the matrix:
\begin{align*}
    \bm M & = \frac{1}{m} \sum_{i=1}^m  \trim(y_i) \bm a_i \bm a_i^\UT,
\end{align*}
where $\bm a_{1:m}$ denote the rows of $\bm A$ and $\trim: \R_{\geq 0} \rightarrow (-\infty,1)$ is a trimming function. \rishabh{Spectral estimators are often used as an initialization for more sophisticated iterative recovery algorithms \citep{netrapalli2013phase,candes2015wirtinger,ChenCandes17, montanari2021estimation,mondelli2021approximate,mondelli2021pca} such as the non-linear approximate message passing algorithm in \eqref{eq: NL_AMP}, which requires an informative initialization in order to have a non-trivial performance.}  \rishabh{The performance of these spectral estimators have been analyzed in the high dimensional limit \citep{dudeja2020analysis} for the sub-sampled Haar model. While simulations show that the same result holds for sub-sampled Hadamard sensing, the proof approach of \citep{dudeja2020analysis} does not extend to this sensing model since it crucially relies on the rotational invariance of the sub-sampled Haar model.} In this situation, the iterative algorithm in \eqref{eq: LAMP_iteration_prelim} provides a theoretical tractable alternative that is closely connected to spectral estimators. This connection was established by \citet{ma2019spectral}, who proposed setting the functions $\eta_t$ in the following way:
\begin{align} \label{eq: LAMP_spectral_config}
    \eta_t(y) & = \left( \frac{1}{\mu} - \trim(y) \right)^{-1},
\end{align}
where $\mu \in (0,1)$ is a tuning parameter. \citeauthor{ma2019spectral} show that with this choice of $\eta_t$, every fixed point of the iteration \eqref{eq: LAMP_iteration_prelim} denoted by $\bm z^\infty$, $\bm A^\UT \bm z^\infty$ is an eigenvector of the matrix $\bm M$. Furthermore, suppose $\mu$ is set to be the solution to the equation:
\begin{align} \label{eq: mu_parameter_setting}
    \psi_1(\mu) = \frac{1}{1-\kappa}, \; \psi_1(\mu) \explain{def}{=} \frac{\E |Z|^2 G}{\E G},
\end{align}
where the joint distribution of $(Z,G)$ is given by:
\begin{align*}
    Z \sim \gauss{0}{1}, \; G = \left( \frac{1}{\mu} - \trim(|Z|) \right)^{-1}.
\end{align*}
Then, \citeauthor{ma2019spectral} have shown that the linearized message passing iterations \eqref{eq: LAMP_iteration_prelim} achieve the same performance as the spectral method for the sub-sampled Haar model as $t \rightarrow \infty$. 

\rishabh{Finally, we remark that when the sensing matrix is rotationally invariant, even though spectral estimators can be analyzed directly using random matrix theory, the characterization of the dynamics of linearized message passing algorithm in \eqref{eq: LAMP_iteration_prelim} along with its connection to spectral estimators has still proved to be useful as a proof technique to address questions beyond those that can be answered by direct analysis of the spectral estimator using random matrix theory alone. Examples include (i) work by \citet{montanari2021estimation,mondelli2021approximate,mondelli2021pca} who use this proof technique to study the dynamics of non-linear approximate message passing algorithms initialized with spectral estimators for inference problems involving rotationally invariant matrices and (ii) work by \citet{mondelli2021optimal} who rely on this technique to characterize the joint distribution of the spectral estimator and the ordinary least squares (OLS) estimator and use this characterization to design the optimal strategy to combine these estimators. Hence, the analysis of the dynamics of linearized message-passing algorithms \eqref{eq: LAMP_iteration_prelim}  is likely to be useful for deriving similar results for the sub-sampled Hadamard sensing model studied in this paper. This serves as additional motivation for studying this particular family of iterative algorithms.}  

\paragraph{The State Evolution Formalism} An important property of the AMP algorithms of \eqref{eq: LAMP_iteration_prelim} and \eqref{eq: NL_AMP} is that for the sub-sampled Haar model, the dynamics of the algorithm can be tracked by a deterministic scalar recursion known as the state evolution. \rishabh{This was first shown for Gaussian sensing matrices by \citet{bayati2011dynamics} and subsequently for rotationally invariant ensembles by \citet{rangan2019vector} and \citet{takeuchi2019rigorous}. More recently, significant generalizations of these results have obtained in the work of \citet{fan2020approximate} and subsequent works by \citet{zhong2021approximate,venkataramanan2021estimation}. By instantiating \citet[Theorem 1]{venkataramanan2021estimation} to our setup, we obtain the following state evolution for Linearized AMP algorithms (additional details regarding this derivation are provided in Appendix \ref{sec:appendix-SE-derivation}).}

\begin{prop}[State Evolution \citep{venkataramanan2021estimation}] \label{prop: SE_rotationally_invariant} Suppose that the sensing matrix is generated from the sub-sampled Haar model and the signal vector is normalized such that $\|\bm x\|_2^2/m \explain{P}{\rightarrow} 1$ and the iteration \eqref{eq: LAMP_iteration_prelim} is initialized as:
\begin{align*}
    \hat{\bm z}^{(0)} & = \alpha_0 \bm z + \sigma_0 \bm w,
\end{align*}
where $\alpha_0 \in \R, \sigma_0 \in \R_{+}$ are fixed and $\bm w \sim \gauss{\bm 0}{\bm I_m}$. Then for any fixed $t \in \N$, as $m,n \rightarrow \infty$, $n/m \rightarrow \kappa$, we have,
\begin{align*}
    \frac{\ip{\hat{\bm z}^{(t)}}{\bm z}}{m} &\explain{P}{\rightarrow} \alpha_t, \; \frac{\|{\hat{\bm z}^{(t)}}\|_2^2}{m} \explain{P}{\rightarrow} \alpha_t^2 + \sigma_{t}^2,  \\
     \frac{\ip{\hat{\bm x}^{(t)}}{\bm x}}{m} &\explain{P}{\rightarrow} \alpha_t, \; \frac{\|{\hat{\bm x}^{(t)}}\|_2^2}{m} \explain{P}{\rightarrow} \alpha_t^2 + (1-\kappa) \sigma_t^2,
\end{align*}
where $(\alpha_t,\sigma_t^2)$ are given by the recursion:
\begin{subequations}
\label{eq: SE_recursion}
\begin{align}
    \alpha_{t+1} & = \intoo{{\frac{1}{\kappa}}-1} \cdot \alpha_t \cdot \E Z^2 \overline{\eta}_t(|Z|), \\
    \sigma_{t+1}^2 & = \left( \frac{1}{\kappa} - 1 \right) \cdot \left( \alpha_t^2 \cdot \left\{\E Z^2 \overline{\eta}^2_t(|Z|) - (\E Z^2 \overline{\eta}_t(|Z|))^2 \right\} + \sigma_t^2 \E \overline{\eta}_t^2(|Z|) \right).
\end{align}
\end{subequations}
In the above display, $Z \sim \gauss{0}{1}$ and $\overline{\eta}_t(z) = \eta_t(z) - \E \eta_t(|Z|)$.
\end{prop}
The above proposition lets us track the evolution of some performance metrics like the mean squared error (MSE) and the cosine similarity of the iterates. The proof of Proposition \ref{prop: SE_rotationally_invariant} crucially relies on the rotational invariance of the sub-sampled Haar ensemble via Bolthausen's conditioning technique \citep{bolthausen2009high} and does not extend to structured sensing ensembles. 

\rishabh{
\begin{rem} A limitation of Proposition \ref{prop: SE_rotationally_invariant} is that it characterizes the dynamics of linearized AMP algorithms only in the regime when the number of iterations $t = O(1)$ as $m,n \rightarrow \infty$. In this regime, these algorithms need to be initialized informatively (that is, $|\alpha_0| > 0$) to have a non-trivial performance in $O(1)$ iterations. Such an initialization may not always be available in practice. Despite this, the state evolution results, such as the one in Proposition \ref{prop: SE_rotationally_invariant}, can provide theoretical insights into the performance of practical algorithms like spectral estimators. As discussed previously, when the sensing matrix is rotationally invariant, even though spectral estimators can be analyzed directly using random matrix theory, the characterization of the dynamics of linearized AMP algorithms along with their connection to spectral estimators has still proved to be useful as a proof technique to address questions beyond those that can be answered by direct analysis of the spectral estimator using random matrix theory alone \citep{montanari2021estimation,mondelli2021approximate,mondelli2021optimal,mondelli2021pca}.
\end{rem}
}

\paragraph{A Demonstration of the Universality phenomenon} For the sake of completeness, we provide a self contained demonstration of the universality phenomenon that we seek to study in Figure \ref{fig: empirical} and Figure \ref{fig: empirical2}.
\begin{figure}[h] 
\centering
\includegraphics[width=\textwidth]{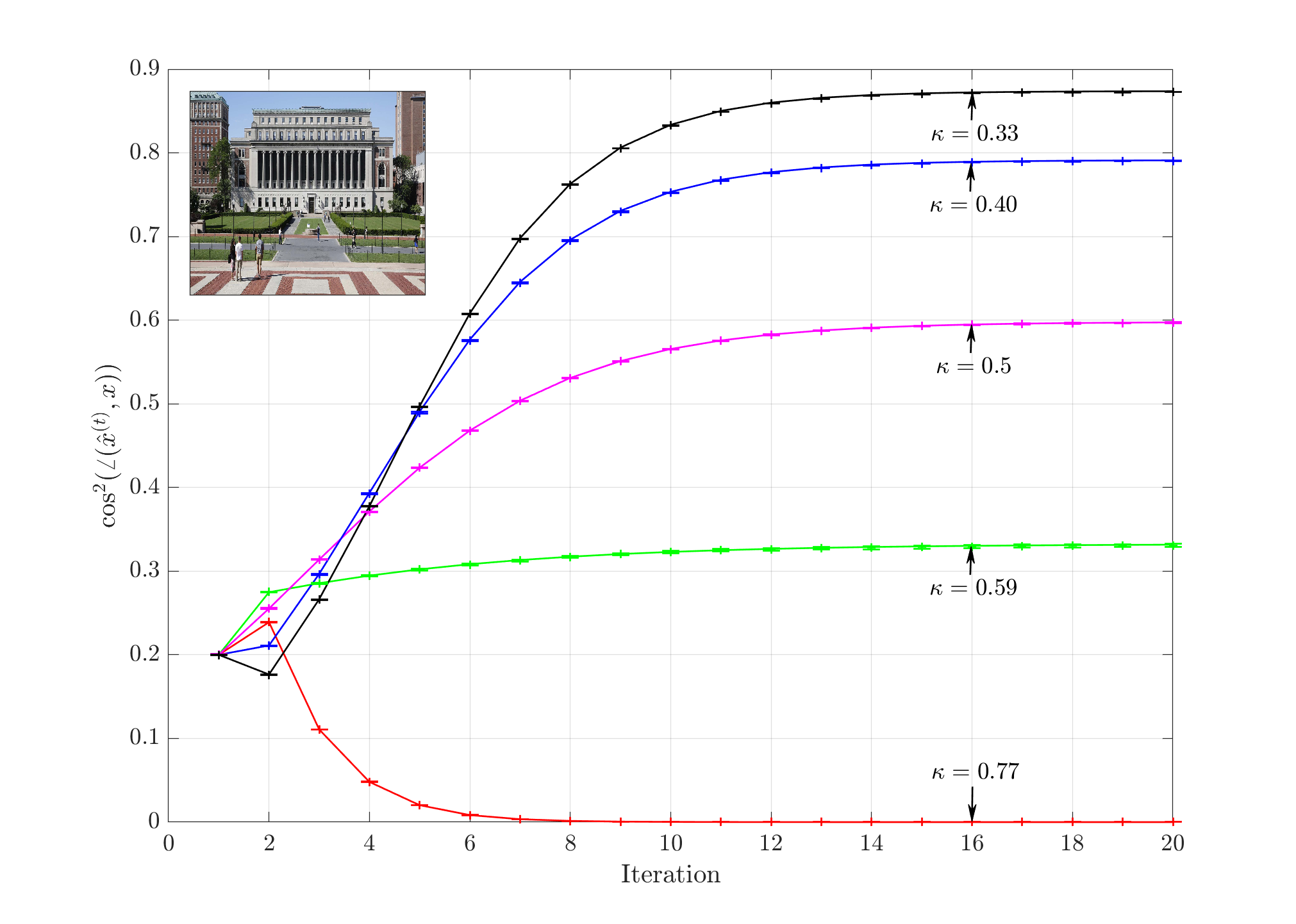}
\caption{Solid Lines: Predicted Dynamics derived using the State Evolution for sub-sampled Haar sensing (Proposition \ref{prop: SE_rotationally_invariant}), + markers: Dynamics of Linearized Message Passing averaged over ten repetitions with sub-sampled Hadamard sensing when the signal is an actual image (shown in inset). The error bars represent the standard error across repetitions.  }
\label{fig: empirical}
\end{figure}
\begin{figure}[h] 
\centering
\includegraphics[width=0.8\textwidth]{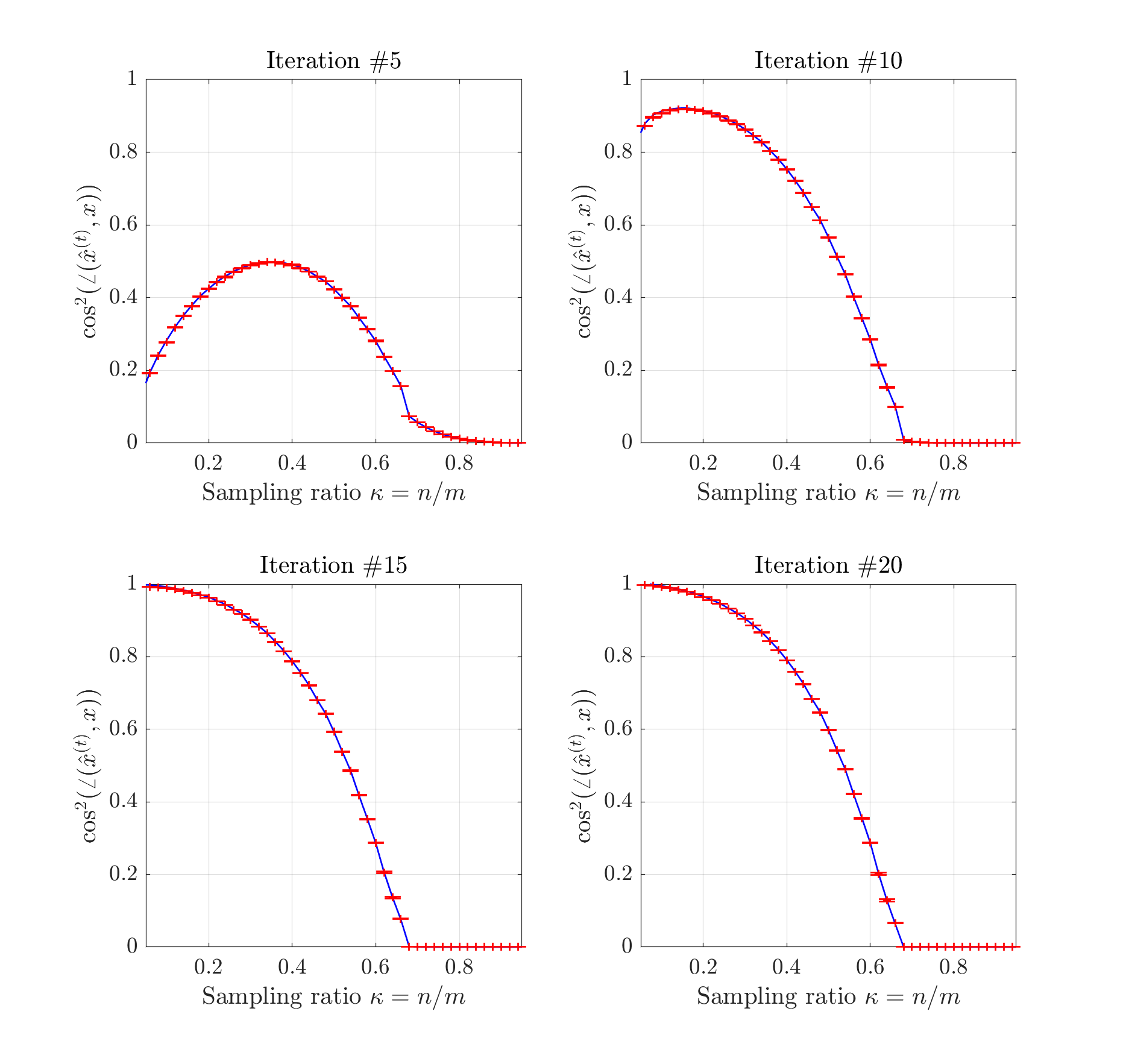}
\caption{Blue Solid Lines: Predicted Dynamics derived using the State Evolution for sub-sampled Haar sensing (Proposition \ref{prop: SE_rotationally_invariant}), Red + markers: Dynamics of Linearized Message Passing averaged over ten repetitions with sub-sampled Hadamard sensing when the signal is an actual image. The error bars represent the standard error across repetitions. }
\label{fig: empirical2}
\end{figure}

To generate these figures:
\begin{enumerate}
    \item We used a $1024\times 256$ image (after vectorization, shown as inset in Figure \ref{fig: empirical}) as the signal vector. Each of the red, blue, green channels were centered so that that their mean was zero and standard deviation was $1$. 
    \item We set $m = 1024 \times 256$.
    \item In order to generate problems with different $\kappa$ we down-sampled the original image to obtain a new signal with $n \approx m \kappa$ (up to rounding errors) for a fine grid of $\kappa$ values in the interval $[0.05,0.95]$.
    \item We used a randomly sub-sampled Hadamard matrix for sensing. This was used to construct a phase retrieval problem for each of the red, blue and green channels. 
    \item We used the linearized message passing configured to implement the spectral estimator (c.f. \eqref{eq: LAMP_spectral_config} and \eqref{eq: mu_parameter_setting}) with the optimal trimming function \citep{yuelu_optimal,ma2019spectral}:
    \begin{align*}
    \trim_\star(y) & = 1 - \frac{1}{y^2}.
    \end{align*}
    We ran the algorithm for 20 iterations and tracked the squared cosine similarity:
    \begin{align*}
        \cos^2(\angle(\hat{\bm x}^{(t)}, \bm x)) \explain{def}{=} \frac{|\ip{\hat{\bm x}^{(t)}}{\bm x}|^2}{\|\hat{\bm x}^{(t)} \|_2^2 \|\bm x\|_2^2}.
    \end{align*}
    We averaged the squared cosine similarity across the RGB channels.
    \item We repeated this for 10 different random sensing matrices. The average cosine similarity is represented by $+$ markers in Figure \ref{fig: empirical} and Figure \ref{fig: empirical2} and the error bars represent the standard error across 10 repetitions. The solid curves represent the predictions derived from State Evolution for sub-sampled Haar sensing (see Proposition \ref{prop: SE_rotationally_invariant}). In Figure \ref{fig: empirical}, we plotted the entire dynamics for 20 iterations for 5 representative values of $\kappa \in \{0.33, 0.40, 0.5, 0.59, 0.77\}$. In Figure \ref{fig: empirical2}, we chose 4 representative iterations $t \in \{5, 10, 15, 20\}$ and plotted the squared cosine similarity at these iterations for a fine grid of $\kappa$ values in $[0.05,0.95]$. We can observe that the State Evolution closely tracks the empirical dynamics. 
\end{enumerate}

\paragraph{Assumption on the signal} It is easy to see that, unlike in the sub-sampled Haar case, the state evolution cannot hold for arbitrary worst case signal vectors for the sub-sampled Hadamard sensing models since the orthogonal signal vectors $\sqrt{m} \bm e_1$ and $\sqrt{m} \bm e_2$ generate the same measurement vector $\bm y = (1, 1 \cdots ,1)^\UT$. This is a folklore argument for non-identifiability of the phase retrieval problem for $\pm 1$ sensing matrices \citep{krahmer2014structured}. Hence we study the universality phenomenon under the simplest average case assumption on the signal, namely $\bm x \sim \gauss{\bm 0}{\bm I_n/\kappa}$.

\subsection{Notation}
\paragraph{Important Sets} $\N, \W, \R, \C$ denote the sets of natural numbers, non-negative integers, real numbers, and complex numbers, respectively. $[k]$ denotes the set $\{1,2,\cdots ,k\}$ and $[i:j]$ denotes the set~{$\{i,i+1,i+2 \cdots ,j-1, j\}$}. $\O(m)$ refers to the set of all $m \times m$ orthogonal matrices and $\mathbb U(m)$ refers to the set of all $ m \times m$ unitary matrices. 

\paragraph{Stochastic Convergence} $\explain{P}{\rightarrow}$ denotes convergence in probability. If for a sequence of random variables we have $X_n \explain{P}{\rightarrow} c$ for a deterministic $c$, we say $\plim X_n = c$.

\paragraph{Linear Algebraic Aspects} We will use bold face letters to refer to vectors and matrices. For a matrix $\bm V \in \R^{m \times n}$, we adopt the convention of referring to the columns of $\bm V$ by $\bm V_1, \bm V_2 \cdots \bm V_n \in \R^m$ and to the rows by $\bm v_1, \bm v_2 \cdots \bm v_m \in \R^n$. For a vector $\bm v$, $\|\bm v\|_1, \|\bm v\|_2, \|\bm v\|_\infty$ denote the $\ell_1,\ell_2$, and $\ell_\infty$ norms, respectively. By default, $\|\bm v\|$ denotes the $\ell_2$ norm. For a matrix $\bm V$, $\|\bm V\|_\op, \|\bm V\|_\fr, \|\bm V\|_\infty$ denote the operator norm,  Frobenius norm, and the entry-wise $\infty$-norm, respectively. For vectors $\bm v_1, \bm v_2 \in \R^n$, $\ip{\bm v_1}{\bm v_2}$ denotes the inner product $\ip{\bm v_1}{\bm v_2} = \sum_{i=1}^n v_{1i} v_{2i}$. For matrices $\bm V_1, \bm V_2 \in \R^{m \times n}$, $\ip{\bm V_1}{\bm V_2}$ denotes the matrix inner product $\sum_{i=1}^m \sum_{j=1}^n (V_1)_{ij} (V_2)_{ij}$.

\paragraph{Important distributions} $\gauss{\mu}{\sigma^2}$ denotes the scalar Gaussian distribution with mean $\mu$ and variance $\sigma^2$. $\gauss{\bm \mu}{\bm \Sigma}$ denotes the multivariate Gaussian distribution with mean vector $\bm \mu$ and covariance matrix $\bm \Sigma$. $\bern{(p)}$ denotes Bernoulli distribution with bias $p$. $\bnomdistr{n}{p}$ denotes the Binomial distribution with $n$ trials and bias $p$. For an arbitrary set $S$, $\unif{S}$ denotes the uniform distribution on the elements of $S$. For example, $\unif{\O(m)}$ denotes the Haar measure on the orthogonal group. 

\paragraph{Order Notation and Constants} We use the standard $O(\cdot)$ notation. $C$ will be used to refer to a universal constant independent of all parameters. When the constant $C$ depends on a parameter $k$ we will make this explicit by using the notation $C_k$ or $C(k)$. We say a sequence $a_n = O(\polylog(n))$ if there exists a fixed, finite constant $K$ such that $a_n \leq O(\log^K(n))$.

\section{Main Result}
Now, we are ready to state our main result.

\begin{thm} \label{thm: main_result}
Consider the linear message passing iterations \eqref{eq: LAMP_iteration_prelim}. Suppose that:
\begin{enumerate}
    \item The functions $\eta_t$ are bounded and Lipchitz.
    \item The signal is generated from the Gaussian prior: $\bm x \sim \gauss{\bm 0}{\frac{1}{\kappa}\bm I_n}$.
    \item The sensing matrix is generated from the sub-sampled Hadamard ensemble.
    \item The iteration \eqref{eq: LAMP_iteration_prelim} is initialized as:
\begin{align*}
    \hat{\bm z}^{(0)} & = \alpha_0 \bm z + \sigma_0 \bm w,
\end{align*}
where $\alpha_0 \in \R, \sigma_0 \in \R_{+}$ are fixed and $\bm w \sim \gauss{\bm 0}{\bm I_m}$.  
\end{enumerate}
Then for any fixed $t \in \N$, as $m,n \rightarrow \infty$, $n= \kappa m$, we have,
\begin{align*}
    \frac{\ip{\hat{\bm z}^{(t)}}{\bm z}}{m} &\explain{P}{\rightarrow} \alpha_t, \; \frac{\|{\hat{\bm z}^{(t)}}\|_2^2}{m} \explain{P}{\rightarrow} \alpha_t^2 + \sigma_{t}^2,  \\
     \frac{\ip{\hat{\bm x}^{(t)}}{\bm x}}{m} &\explain{P}{\rightarrow} \alpha_t, \; \frac{\|{\hat{\bm x}^{(t)}}\|_2^2}{m} \explain{P}{\rightarrow} \alpha_t^2 + (1-\kappa) \sigma_t^2,
\end{align*}
where $(\alpha_t,\sigma_t^2)$ are given by the recursion in \eqref{eq: SE_recursion}.
\end{thm}

Theorem \ref{thm: main_result} simply states that the dynamics of linearized message passing in the sub-sampled Hadamard model are asymptotically indistinguishable from the dynamics in the sub-sampled Haar model. This provides a theoretical justification for the universality depicted in Figure \ref{fig: empirical}.
%%%%%%%%%%%%%%%%%%%%%%%%%%%%%%%%%%%%%%%%%%%%%%%%%%%%%%%%%%%%%%
%%%%%%%%%%%%%%%%%%%%%%%%%%%%%%%%%%%%%%%%%%%%%%%%%%%%%%%%%%%%%%
%%%%%%%%%%%%%%%%%%%%%%%%%%%%%%%%%%%%%%%%%%%%%%%%%%%%%%%%%%%%%%

%%%%%%%%%%%%%%%%%%%%%%%%%%%%%%%%%%%%%%%%%%%%%%%%%%%%%%%%%%%%%%
%%%%%%%%%%%%%%%%% RELATED WORK %%%%%%%%%%%%%%%%%%%%%%%%%%%%%%%
%%%%%%%%%%%%%%%%%%%%%%%%%%%%%%%%%%%%%%%%%%%%%%%%%%%%%%%%%%%%%%
\section{Related Work}
\label{section: related_work}

\paragraph{Gaussian Universality} A number of papers have tried to explain the observations of \citet{donoho2009observed} regarding the universality in performance of $\ell_1$ minimization for noiseless linear sensing. For noiseless linear sensing, the Gaussian sensing ensemble, sub-sampled Haar sensing ensemble, and structured sensing ensembles like sub-sampled Fourier sensing ensemble behave identically. Consequently, a number of papers have tried to identify the class of sensing matrices which behave like Gaussian sensing matrices. It has been shown that sensing matrices with i.i.d. entries under mild moment assumptions behave like Gaussian sensing matrices in the context of performance of general (non-linear) Approximate Message Passing schemes \citep{bayati2011dynamics,chen2020universality}, the limiting Bayes risk \citep{barbier2018mutual}, and the performance of estimators based on convex optimization \citep{korada2011applications,oymak2018universality}. The assumption that the sensing matrix has i.i.d. entries has been relaxed to the assumption that it has i.i.d. rows (with possible dependence within a row) \citep{abbasi2019universality}. Finally, we emphasize that in the presence of noise or when the measurements are non-linear, the structured ensembles that we consider here, obtained by sub-sampling a deterministic orthogonal matrix like the Hadamard-Walsh matrix, no longer behave like Gaussian matrices, but rather like sub-sampled Haar matrices. 

\paragraph{A result for highly structured ensembles} While the results mentioned above move beyond i.i.d. Gaussian sensing, the sensing matrices they consider are still largely unstructured and highly random. In particular, they do not apply to the sub-sampled Hadamard ensemble considered here. A notable exception is the work of \citet{donoho2010counting} which considers a random undetermined system of linear equations (in $\bm x$) of the form $\bm A \bm x = \bm A \bm x_0$ for a random matrix $\bm A \in \R^{m \times n}$ and a $k$-sparse non-negative vector $\bm x_0 \in \R_{\geq 0}^n$. \citeauthor{donoho2010counting} shows that as $m,n,k \rightarrow \infty$ such that $n/m \rightarrow \kappa_1, k/m \rightarrow \kappa_2$, the probability that $\bm x_0$ is the unique non-negative solution to the system sharply transitions from $0$ to $1$ depending on the values $\kappa_1,\kappa_2$. Moreover, this transition is universal across a wide range of random $\bm A$, including Gaussian ensembles, random matrices with i.i.d. entries sampled from a symmetric distribution, and highly structured ensembles whose null space is given by  a random matrix $\bm B \in \R^{n-m \times n}$ generated by multiplying the columns of a fixed matrix $\bm B_0$ whose columns are in general position by i.i.d. random signs. The proof technique of \citeauthor{donoho2010counting} uses results from the theory of random polytopes and it is not obvious how to extend their techniques beyond the case of solving under-determined linear equations.   

\paragraph{Universality Results in Random Matrix Theory} The phenomenon that structured orthogonal matrices, such as Hadamard and Fourier matrices, behave like random Haar matrices in some aspects has been studied in the context of random matrix theory \citep{anderson2010introduction} and in particular free probability \citep{mingo2017free}. A well known result in free probability (see the book of \citet{mingo2017free} for a textbook treatment) is that if $\bm U \sim \unif{\mathbb{U}(m)}$ and $\bm D_1,\bm D_2$ are deterministic $m \times m$ diagonal matrices then $\bm U \bm D_1 \bm U^\UH$  and $\bm D_2$ are asymptotically free and consequently the limiting spectral distribution of matrix polynomials in $\bm D_2$ and $\bm U \bm D_1 \bm U^\UH$  can be described in terms of the limiting spectral distribution of $\bm D_1$ and $\bm D_2$. \citet{tulino2010capacity,farrell2011limiting} have obtained an extension of this result where a Haar unitary matrix is replaced by  $m \times m$ Fourier matrix: If $\bm D_1, \bm D_2$ are independent diagonal matrices then $\bm F_m \bm D_1 \bm F_m^\UH $ is asymptotically free from $\bm D_2$. The result of these authors has been extended to other deterministic orthogonal/unitary matrices (such as the Hadamard-Walsh matrix) conjugated by random signed permutation matrices by \citet{ANDERSON2014381}. \rishabh{In order to see how the result of \citeauthor{tulino2010capacity} connects with ours note that the linearized AMP iterations \eqref{eq: LAMP_iteration_prelim} involve 2 random matrices: $\bm A \bm A^\UT =  \bm H \bm B \bm H^\UT$ where $\bm B$ is the diagonal Bernoulli matrix defined in \eqref{eq:B-barB-def} and $\eta(\bm Y) = \diag{\eta(y_1), \dotsc, \eta(y_m)}$. Note that if $\bm B$ and the diagonal matrix $\eta(\bm Y)$ were independent, then the result of \citeauthor{tulino2010capacity} would imply that $\bm H \bm B \bm H^\UT $ and $\eta(\bm Y)$ are asymptotically free and this could potentially be used to analyze the linearized AMP algorithm. However, the key difficulty is that the measurements $\bm y$ depend on which columns of the Hadamard-Walsh matrix were selected (specified by $\bm B$). In fact, this dependence is precisely what allows the linearized AMP algorithm to recover the signal. However, we still find some of the techniques introduced by \citeauthor{tulino2010capacity} useful in our analysis. We also emphasize that asymptotic freeness of $\bm H \bm B \bm H^\UT, \; \eta(\bm Y)$ alone seems to be insufficient to characterize the behavior of Linearized AMP algorithms. Asymptotic freeness implies that the expected normalized trace of certain matrix products involving $\bm H \bm B \bm H^\UT, \; \eta(\bm Y)$ vanish in the limit $m \rightarrow \infty$. On the other hand, our proof also requires the analysis of certain quadratic forms involving $\bm H \bm B \bm H^\UT , \; \eta(\bm Y)$ (see Proposition \ref{proposition: free_probability_qf}) which do not appear to have been studied in the free probability literature.}

\paragraph{Non-rigorous Results from Statistical Physics} In the statistical physics literature Cakmak, Opper, Winther, and Fleury \citep{cakmak2017dynamical,ccakmak2019memory,cakmak2020analysis,ccakmak2020dynamical,opper2020understanding}  have developed an analysis of message passing algorithms for rotationally invariant ensembles via a non-rigorous technique called the dynamical functional theory. These works are interesting because they do not heavily rely on rotational invariance, but instead rely on results from Free probability. Since some of the free probability results have been extended to Fourier and Hadamard matrices \citep{tulino2010capacity,farrell2011limiting,ANDERSON2014381}, there is hope to generalize their analysis beyond rotationally invariant ensembles. However, currently, their results are non-rigorous due to two reasons: 1) due to the use of dynamical field theory, and 2) their application of Free probability results neglects dependence between matrices. In our work, we avoid  the use of dynamical functional theory since we analyze linearized AMP algorithms and furthermore, we properly account for dependence that is heuristically neglected in their work. 

\paragraph{The Hidden Manifold Model} Lastly, we discuss the recent works of \citet{goldt2019modelling,gerace2020generalisation,goldt2020gaussian}, where they study statistical learning problems where the feature matrix $\bm A \in \R^{m \times n}$ (the analogue of the sensing matrix in statistical learning) is generated as:
\begin{align*}
    \bm A & = \sigma(\bm Z \bm F), 
\end{align*}
where $\bm F \in \R^{d \times n}$ is a generic (possibly structured) deterministic weight matrix and $\bm Z \in \R^{m \times d} $ is an i.i.d. Gaussian matrix. The function $\sigma: \R \rightarrow \R$ acts entry-wise on the matrix $\bm Z \bm F$. For this model, the authors have analyzed the dynamics of online (one-pass) stochastic gradient descent (first non-rigorously \citep{goldt2019modelling} and then rigorously \citep{goldt2020gaussian}) and the performance of regularized empirical risk minimization with convex losses (non-rigorously) via the replica method \citep{gerace2020generalisation} in the high dimensional asymptotic $m,n,d \rightarrow \infty$, $n/m \rightarrow \kappa_1, d/m \rightarrow \kappa_2$. Their results show that in this case the feature matrix behaves like a certain correlated Gaussian feature matrix. We note that the feature matrix $\bm A$ here is quite different from the sub-sampled Hadamard ensemble since it uses $O(m^2)$ i.i.d. random variables ($\bm Z$) where as the sub-sampled Hadamard ensemble only uses $m$ i.i.d. random variables (to specify the permutation matrix $\bm P$). However, a technical result proved by the authors (Lemma A.2 of \citep{goldt2019modelling}) appears to be a special case of a classical result of \citet{mehler1866ueber,slepian1972symmetrized} which we find useful to account for the dependence between the matrices $q_t(\bm Y), \bm A$ appearing in the linearized AMP iterations \eqref{eq: LAMP_iteration_prelim}.

%%%%%%%%%%%%%%%%%%%%%%%%%%%%%%%%%%%%%%%%%%%%%%%%%%%%%%%%%%%%%%
%%%%%%%%%%%%%%%%%%%%%%%%%%%%%%%%%%%%%%%%%%%%%%%%%%%%%%%%%%%%%%
%%%%%%%%%%%%%%%%%%%%%%%%%%%%%%%%%%%%%%%%%%%%%%%%%%%%%%%%%%%%%%

%%%%%%%%%%%%%%%%%%%%%%%%%%%%%%%%%%%%%%%%%%%%%%%%%%%%%%%%%%%%%%
%%%%%%%%%%%%%%%%%%%%%% OVERVIEW %%%%%%%%%%%%%%%%%%%%%%%%%%%%%%
%%%%%%%%%%%%%%%%%%%%%%%%%%%%%%%%%%%%%%%%%%%%%%%%%%%%%%%%%%%%%%
\section{Proof Overview}
Our basic strategy to prove Theorem \ref{thm: main_result} will be as follows: Throughout the paper we will assume that Assumptions 1, 2, and 4 of Theorem \ref{thm: main_result} hold. We will seek to only show that the observables:
\begin{align} \label{eq: key_observables}
    \frac{\ip{\hat{\bm z}^{(t)}}{\bm z}}{m} , \; \frac{\|{\hat{\bm z}^{(t)}}\|_2^2}{m},
     \frac{\ip{\hat{\bm x}^{(t)}}{\bm x}}{m} , \; \frac{\|{\hat{\bm x}^{(t)}}\|_2^2}{m},
\end{align}
have the same limit in probability under both the sub-sampled Haar and the sub-sampled Hadamard sensing models. We will not need to explicitly identify their limits since Proposition \ref{prop: SE_rotationally_invariant} already identifies the limit for us, and hence, Theorem \ref{thm: main_result} will follow. 

It turns out the limits of the observables \eqref{eq: key_observables} depends only on normalized traces and quadratic forms of certain alternating products of the matrices $\bm \Psi$ and \milad{$\bm Z = \diag{z_1, ..., z_m}$}. Hence, we introduce the following definition.

\begin{defn}[Alternating Product] \label{def: alternating_product} A matrix $\altprod$ is said to be a alternating product of matrices $\bm \Psi, \bm Z$ if there exist polynomials $p_i : \R \rightarrow \R,  \; i \in {1,2 \dots ,k}$, and bounded, Lipchitz functions $q_i : \R \rightarrow \R, \; i \in \{1,2\dots k\}$ such that:
\begin{enumerate}
    \item If $B \sim \bern(\kappa)$, $\E p_i(B-\kappa) = 0$.
    \item $q_i$ are even functions i.e. $q_i(\xi) = q_i(-\xi)$ and  if $\xi \sim \gauss{0}{1}$, then, $\E q_i(\xi) =  0$,
\end{enumerate}
and, $\altprod$ is one of the following:
\begin{enumerate}
    \item Type 1:  $\altprod = p_1(\bm \Psi) q_1(\bm Z) p_{2}(\bm \Psi) \cdots  q_{k-1}(\bm Z) p_k(\bm \Psi)$
    \item Type 2: $\altprod = p_1(\bm \Psi) q_1(\bm Z) p_2(\bm \Psi) q_2(\bm Z) \cdots  p_k(\bm \Psi) q_k(\bm Z)$
    \item Type 3: $\altprod = q_1(\bm Z) p_2(\bm \Psi) q_2(\bm Z) \cdots p_k(\bm \Psi) q_k(\bm Z)$. 
    \item Type 4: $\altprod = q_1(\bm Z) p_2(\bm \Psi) q_2(\bm Z) p_3(\bm \Psi) \cdots q_{k-1}(\bm Z) p_k(\bm \Psi)$.
\end{enumerate}
In the above definitions:
\begin{enumerate}
    \item The scalar polynomial $p_i$ is evaluated at the matrix $\bm \Psi$ in the usual sense, for example if $p(\psi) = \psi^2$, then, $p(\bm \Psi) = \bm \Psi^2$.
    \item The functions $q_i$ are evaluated entry-wise on the diagonal matrix $\bm Z$, i.e. $$q_i(\bm Z) = \diag{q_i(z_1), q_i(z_2) \dots q_i(z_m)}.$$
\end{enumerate}
\end{defn}

We note that alternating products are a central notion in free probability \citep{mingo2017free}. The difference here is that we have additionally constrained the functions $p_i, q_i$ in Definition \ref{def: alternating_product}. 

Theorem \ref{thm: main_result} is a consequence of two properties of alternating products which may be of independent interest. These are stated in the following propositions.

\begin{prop} \label{proposition: free_probability_trace} Let $\altprod(\bm \Psi, \bm Z)$ be an alternating product of matrices $\bm \Psi, \bm Z$. Suppose the sensing matrix $\bm A$ is generated from the sub-sampled Haar sensing model, or the sub-sampled Hadamard sensing model, or by sub-sampling a deterministic orthogonal matrix $\bm U$ with the property:
\begin{align*}
    \|\bm U\|_{\infty} & \leq \sqrt{\frac{K_1 \log^{K_2}(m)}{m}}, \; \forall m \; \geq K_3,
\end{align*}
for some fixed constants $K_1,K_2,K_3$. Then,
\begin{align*}
    \Tr(\altprod(\bm \Psi, \bm Z))/m \explain{P}{\rightarrow} 0.
\end{align*}
\end{prop}

\begin{prop} \label{proposition: free_probability_qf} Let $\altprod(\bm \Psi, \bm Z)$ be an alternating product of matrices $\bm \Psi, \bm Z$. Then for the sub-sampled Haar sensing model and for sub-sampled Hadamard ($\bm U = \bm H$) sensing model, we have,
\begin{align*}
    \plim \;  \frac{\ip{\bm z}{\altprod\bm z}}{m}
\end{align*}
exists and is identical for the two models. 
\end{prop}

\paragraph{Outline of the Remaining Paper}The remainder of the paper is organized as follows:
\begin{enumerate}
    \item In Section \ref{section: reduction} we provide a proof of Theorem \ref{thm: main_result} assuming Propositions \ref{proposition: free_probability_trace} and \ref{proposition: free_probability_qf}. 
    \item In Section \ref{section: toolbox} we introduce some key tools required for the proof of Propositions \ref{proposition: free_probability_trace} and \ref{proposition: free_probability_qf}.
    \item The proof of Proposition \ref{proposition: free_probability_trace} can be found in Section \ref{proof_trace}.
    \item The proof of Proposition \ref{proposition: free_probability_qf} can be found in Section \ref{section: qf_proof}.
\end{enumerate}
%%%%%%%%%%%%%%%%%%%%%%%%%%%%%%%%%%%%%%%%%%%%%%%%%%%%%%%%%%%%%%
%%%%%%%%%%%%%%%%%%%%%%%%%%%%%%%%%%%%%%%%%%%%%%%%%%%%%%%%%%%%%%
%%%%%%%%%%%%%%%%%%%%%%%%%%%%%%%%%%%%%%%%%%%%%%%%%%%%%%%%%%%%%%

%%%%%%%%%%%%%%%%%%%%%%%%%%%%%%%%%%%%%%%%%%%%%%%%%%%%%%%%%%%%%%
%%%%%%%%%%%%%%%%%%%%% REDUCTION %%%%%%%%%%%%%%%%%%%%%%%%%%%%%%
%%%%%%%%%%%%%%%%%%%%%%%%%%%%%%%%%%%%%%%%%%%%%%%%%%%%%%%%%%%%%%
\section{Proof of Theorem \ref{thm: main_result}}
\label{section: reduction}
In this section we will show the analysis of the observables \eqref{eq: key_observables} reduces to the analysis of the normalized traces and quadratic forms of alternating products. In particular,  we will prove Theorem \ref{thm: main_result} using Propositions \ref{proposition: free_probability_trace} and \ref{proposition: free_probability_qf}.
\begin{proof}[Proof of Theorem \ref{thm: main_result}] For simplicity, we will assume the functions $\eta_t$ do not change with $t$, i.e. $\eta_t = \eta \; \forall \; t \geq 0$. This is just to simplify notations, and the proof of time varying $\eta_t$ is exactly the same. 
Define the function:
\begin{align*}
    q(z) & = \eta(|z|) - \E_{Z \sim \gauss{0}{1}} [\eta(|Z|)]. 
\end{align*}
{Note that the linearized message passing iterations \eqref{eq: LAMP_iteration_prelim} can be expressed as:}

% \milad{Rishabh, would you define the abbreviation PCA-EP and point out the corresponding equation here?}

\begin{align*}
    \hat{\bm z}^{(t+1)} & = \frac{1}{\kappa}  \cdot \bm \Psi \cdot q(\bm Z) \cdot \hat{\bm z}^{(t)}.
\end{align*}
Unrolling the iterations we obtain:
\begin{align*}
    \hat{\bm z}^{(t)} & = \frac{1}{\kappa^t} \cdot (\bm \Psi \cdot q(\bm Z))^t \cdot \hat{\bm z}^{(0)}.
\end{align*}
Note that the initialization is assumed to be of the form: $\hat{\bm z}^{(0)} = \alpha_0 \bm z + \sigma_0 \bm w$, where $\bm w \sim \gauss{0}{ \bm I}$. Hence:
\begin{align*}
    \hat{\bm z}^{(t)} & = \alpha_0 \frac{1}{\kappa^t}  \cdot (\bm \Psi \cdot q(\bm Z))^t \cdot \bm z + \sigma_0 \cdot \frac{1}{\kappa^t}  \cdot (\bm \Psi \cdot q(\bm Z))^t \cdot \bm w, \\
    \hat{\bm x}^{(t)} & =  \bm A^\UT \hat{\bm z}^{(t)}.
\end{align*}
We will focus on showing that the limits:
\begin{align}
    \plim \frac{\ip{\bm x}{\hat{\bm x}^{(t)}}}{m} , \;   \plim \frac{ \|\hat{\bm x}^{(t)}\|^2_2}{m}\label{to_show_reduction_norm},
\end{align}
exist and are identical for the two models. The claim for the limits corresponding to $\hat{\bm z}^{(t)}$ are exactly analogous and omitted.  Hence, the remainder of the proof is devoted to analyzing the above limits.
\begin{description}
\item [Analysis of $\ip{\bm x}{\hat{\bm x}^{(t)}}$: ]
Observe that:
\begin{align*}
    \ip{\bm x}{\hat{\bm x}^{(t)}} & = \ip{\bm A^\UT \bm z}{\bm A^\UT\hat{\bm z}^{(t)}} \\ &=  \alpha_0 \frac{1}{\kappa^t} \cdot \underbrace{\ip{\bm A^\UT \bm z}{\bm A^\UT (\bm \Psi \cdot q(\bm Z))^t \cdot \bm z}}_{(T_1)} +  \sigma_0 \cdot \frac{1}{\kappa^t} \cdot \underbrace{\ip{\bm A^\UT \bm z}{\bm A^\UT \cdot (\bm \Psi \cdot q(\bm Z))^t \cdot \bm w}}_{(T_2)}.
\end{align*}
We first analyze term $(T_1)$. Observe that:
\begin{align*}
    (T_1)  &= {\bm z}^\UT \bm A \bm A^\UT (\bm \Psi \cdot q(\bm Z))^t \bm z  \\
    & = \bm z^\UT \bm \Psi (\bm \Psi \cdot q(\bm Z))^t \bm z + \kappa \bm z^\UT (\bm \Psi \cdot q(\bm Z))^t \bm z \\
    & = \bm z^\UT \bm \Psi^2 ( q(\bm Z) \bm \Psi)^{t-1} q(\bm Z) \bm z + \kappa \bm z^\UT (\bm \Psi \cdot q(\bm Z))^t \bm z \\
    & \explain{(a)}{=} \bm z^\UT p(\bm \Psi) (q(\bm Z) \bm \Psi)^{t-1} q(\bm Z) \bm z + \kappa(1-\kappa)  \bm z^\UT (q(\bm Z) \bm \Psi)^{t-1} q(\bm Z) \bm z + \kappa \bm z^\UT (\bm \Psi \cdot q(\bm Z))^t \bm z.
\end{align*}
In the step marked (a) we defined the polynomial $p(\psi) = \psi^2 - \kappa(1-\kappa)$ which has the property $\E p(B - \kappa) = 0$ when $B \sim \bern(\kappa)$. One can check that  $Z \sim \gauss{0}{1}$, $\E q(Z)= 0$, and $q$ is a bounded, Lipchitz, even function. Hence, each of the terms appearing in step (a) are of the form $\bm z^\UT \altprod \bm z$ for some alternating product $\altprod$ (Definition \ref{def: alternating_product}) of matrices $\bm \Psi, \bm Z$. Consequently, by Proposition \ref{proposition: free_probability_qf} we obtain that term $(1)$ divided by $m$ converges to the same limit in probability under both the sub-sampled Haar sensing and the sub-sampled Hadamard sensing model. 
Next, we analyze $(T_2)$. Note that:
\begin{align*}
    \frac{\ip{\bm A^\UT \bm z}{\bm A^\UT \cdot (\bm \Psi \cdot q(\bm Z))^t \cdot \bm w}}{m} & = \bm z^\UT \bm A \bm A^\UT \bm (\bm \Psi \cdot q(\bm Z))^t \bm w/m\\
    & \explain{d}{=} \frac{\|(q(\bm Z) \bm \Psi)^t \bm A \bm A^\UT \bm z \|_2}{m} \cdot  W, \; W \sim \gauss{0}{1},
\end{align*}
where $\explain{d}{=}$ means both sides have a same distribution.
Observe that:
\begin{align*}
    \frac{\|(q(\bm Z) \bm \Psi)^t \bm A \bm A^\UT \bm z \|_2}{m} & = \frac{\|(q(\bm Z) \bm \Psi)^t \bm A \bm x \|_2}{m}\\
    &\leq \|(q(\bm Z) \bm \Psi)^t \bm A \|_{\op} \cdot  \frac{\|\bm x\|_2}{m} \\
    & \leq \|q(\bm Z)\|_{\op}^t \|\bm \Psi\|_{\op}^t \|\bm A\|_{\op} \cdot  \frac{\|\bm x\|_2}{m}. 
\end{align*}
It is easy to check that: $\|q(\bm Z)\|_{\op} \leq 2\|\eta\|_\infty< \infty$. Similarly, ${\|\bm \Psi\|_{\op} \leq 1}, \; {\|\bm A \|_\op = 1}$. Hence,
\begin{align*}
    \frac{\|(q(\bm Z) \bm \Psi)^t \bm A \bm A^\UT \bm z \|_2}{m} & \leq 2^t \|\eta\|_\infty^{t} \cdot \sqrt{\frac{\|\bm x\|^2}{m}} \cdot \frac{1}{\sqrt{m}}
\end{align*}
Observing that $\|\bm x\|^2/m \explain{P}{\rightarrow} 1$ we obtain:
\begin{align*}
     \left|\frac{\ip{\bm A^\UT \bm z}{\bm A^\UT \cdot (\bm \Psi \cdot q(\bm Z))^t \cdot \bm w}}{m} \right|& \leq 2^t \|\eta\|_\infty^{t} \cdot \sqrt{\frac{\|\bm x\|^2}{m}}  \cdot \frac{|W|}{\sqrt{m}} \explain{P}{\rightarrow} 0.
\end{align*}
Note the above result holds for both subsampled Haar sensing and subsampled Hadamard sensing. This proves that the limit
\begin{align*}
    \plim  \frac{\ip{\bm x}{\hat{\bm x}^{(t)}}}{m}
\end{align*}
exists and is identical for the two models. 
\item [Analysis of $\|\hat{\bm x}^{(t)}\|^2$: ] Recalling that:
\begin{align*}
    \hat{\bm z}^{(t)} & = \alpha_0 \frac{1}{\kappa^t}  \cdot (\bm \Psi \cdot q(\bm Z))^t \cdot \bm z + \sigma_0 \frac{1}{\kappa^t} \cdot (\bm \Psi \cdot q(\bm Z))^t \cdot \bm w, \\
    \hat{\bm x}^{(t)} &= \bm A^\UT \hat{\bm z}^{(t)},
\end{align*}
we can compute:
\begin{align*}
    \frac{1}{m} \|\hat{\bm x}^{(t)}\|_2^2 & =\frac{1}{\kappa^{2t}} \cdot \left( \alpha_0^2 \cdot (T_3) + 2 \alpha_0 \sigma_0 (T_4) + \sigma_0^2 \cdot (T_5) \right),
\end{align*}
where the terms $(T_3-T_5)$ are defined as:
\begin{align*}
    (T_3) & =  \frac{\bm z^\UT (q(\bm Z) \bm \Psi)^t \bm A \bm A^\UT (\bm \Psi \cdot q(\bm Z))^t \cdot \bm z}{m}, \\
    (T_4) & = \frac{\bm z^\UT (q(\bm Z) \bm \Psi)^t \bm A \bm A^\UT(\bm \Psi \cdot q(\bm Z))^t \cdot \bm w}{m}, \\
    (T_5) & = \frac{\bm w^\UT (q(\bm Z) \bm \Psi)^t \bm A \bm A^\UT (\bm \Psi \cdot q(\bm Z))^t \cdot \bm w}{m}.
\end{align*}
We analyze each of these terms separately. First, consider $(T_3)$. Our goal will be to decompose the matrix $(q(\bm Z) \bm \Psi)^t \bm A \bm A^\UT (\bm \Psi \cdot q(\bm Z))^t$ as:
\begin{align*}
    (q(\bm Z) \bm \Psi)^t \bm A \bm A^\UT (\bm \Psi \cdot q(\bm Z))^t &  = c_0 \bm I + \sum_{i=1}^{N_t} c_i \altprod_i,
\end{align*}
where $\altprod_i$ are alternating products of the matrices $\bm \Psi, \bm Z$ (see Definition \ref{def: alternating_product}) and $c_i$ are some scalar constants. This decomposition has the following properties: 1) It is independent of the choice of the orthogonal matrix $\bm U$ used to generate the sensing matrix. 2) The number of terms in the decomposition $N_t$ depends only on $t$ and not on $m,n$. In order to see why such a decomposition exists: first recall that $\bm A \bm A^\UT = \bm \Psi + \kappa \bm I_m$. Hence, we can write: 
\begin{align*}
     &(q(\bm Z) \bm \Psi)^t \bm A \bm A^\UT (\bm \Psi \cdot q(\bm Z))^t =  (q(\bm Z) \bm \Psi)^t \bm \Psi (\bm \Psi \cdot q(\bm Z))^t + \kappa  (q(\bm Z) \bm \Psi)^t  (\bm \Psi \cdot q(\bm Z))^t  \\
     & = (q(\bm Z) \bm \Psi)^{t-1} q(\bm Z) \Psi^3 q(\bm Z) (\bm \Psi \cdot q(\bm Z))^{t-1} + \kappa  (q(\bm Z) \bm \Psi)^{t-1} q(\bm Z) \bm \Psi^2 q(\bm Z)  (\bm \Psi \cdot q(\bm Z))^{t-1}.
\end{align*}
For any $i \in \N$, we write $\bm \Psi^i = p_i(\bm \Psi) + \mu_i \bm I$, where $\mu_i = \E (B-\kappa)^i, \; B \sim \bern(\kappa)$, and $p_i(\psi) = \psi^i - \mu_i$. This polynomial satisfies $\E p_i(B - \kappa) = 0$. This gives us:
\begin{align*}
    (q(\bm Z) \bm \Psi)^t \bm A \bm A^\UT (\bm \Psi \cdot q(\bm Z))^t &=  (q(\bm Z) \bm \Psi)^t \bm \Psi (\bm \Psi \cdot q(\bm Z))^t + \kappa (q(\bm Z) \bm \Psi)^t \bm I (\bm \Psi \cdot q(\bm Z))^t \\
    &=(q(\bm Z) \bm \Psi)^{t-1} q(\bm Z) p_3(\bm \Psi) q(\bm Z) (\bm \Psi \cdot q(\bm Z))^{t-1} 
    \\&
    \quad + \kappa   (q(\bm Z) \bm \Psi)^{t-1} q(\bm Z) \bm p_2(\bm \Psi) q(\bm Z)  (\bm \Psi \cdot q(\bm Z))^{t-1} 
    \\&
    \quad + (\mu_3+\kappa \mu_2) \cdot (q(\bm Z) \bm \Psi)^{t-1} q(\bm Z)^2 (\bm \Psi \cdot q(\bm Z))^{t-1}.
\end{align*}
In the above display, the first two terms on the RHS are in the desired alternating product form. We center the last term. For any $i \in \N$ we define $q_i(z) = q^i(z) - \nu_i$, \; $\nu_i = \E q(\xi)^i, \; \xi \sim \gauss{0}{1}$. Hence, $q^i(\bm Z) = q_i(\bm Z) + \nu_i \bm I_m$. Hence:
\begin{align*}
    (q(\bm Z) \bm \Psi)^t \bm A \bm A^\UT (\bm \Psi \cdot q(\bm Z))^t &=  (q(\bm Z) \bm \Psi)^{t-1} q(\bm Z) p_3(\bm \Psi) q(\bm Z) (\bm \Psi \cdot q(\bm Z))^{t-1} \\&
    \quad + \kappa  (q(\bm Z) \bm \Psi)^{t-1} q(\bm Z) \bm p_2(\Psi) q(\bm Z)  (\bm \Psi q(\bm Z))^{t-1} \\ &
    \quad + (\mu_3+\kappa \mu_2)  (q(\bm Z) \bm \Psi)^{t-1} q_2(\bm Z) (\bm \Psi \cdot q(\bm Z))^{t-1}\\ &
    \quad + \nu_2 \ (\mu_3+\kappa \mu_2) (q(\bm Z) \bm \Psi)^{t-1} (\bm \Psi \cdot q(\bm Z))^{t-1}. 
\end{align*}
In the above display, each of the terms in the right hand side is an alternating product except $(\mu_3+\kappa \mu_2) \cdot (q(\bm Z) \bm \Psi)^{t-1} (\bm \Psi \cdot q(\bm Z))^{t-1}$.
{Note that this term is very similar to what we have started with, but with smaller powers for $(q(\bm Z) \bm \Psi)$ and $(\bm \Psi q(\bm Z))$.  Hence, we can inductively center this term.  To make this clear, we proceed to one more step below:
}
{
\begin{align*}
    (q(\bm Z) \bm \Psi)^{t-1} (\bm \Psi \cdot q(\bm Z))^{t-1} &=
    (q(\bm Z) \bm \Psi)^{t-2} q(\bm Z) \bm \Psi^2 q(\bm Z) (\bm \Psi \cdot q(\bm Z))^{t-2}
    \\ &=
    (q(\bm Z) \bm \Psi)^{t-2} q(\bm Z)  p_2(\bm \Psi) q(\bm Z) (\bm \Psi \cdot q(\bm Z))^{t-2}
    \\ & \quad +
    \mu_2 (q(\bm Z) \bm \Psi)^{t-2} q(\bm Z)^2 (\bm \Psi \cdot q(\bm Z))^{t-2}
    \\ & = 
    (q(\bm Z) \bm \Psi)^{t-2} q(\bm Z)  p_2(\bm \Psi) q(\bm Z) (\bm \Psi \cdot q(\bm Z))^{t-2}
    \\ & \quad +
    \mu_2 (q(\bm Z) \bm \Psi)^{t-2} q_2(\bm Z) (\bm \Psi \cdot q(\bm Z))^{t-2}    
    \\ & \quad + 
    \nu_2 \mu_2 (q(\bm Z) \bm \Psi)^{t-2} (\bm \Psi \cdot q(\bm Z))^{t-2}.
\end{align*}
Hence, starting from $(q(\bm Z) \bm \Psi)^{t-1} (\bm \Psi \cdot q(\bm Z))^{t-1}$ we again end up with two alternating product terms plus $(q(\bm Z) \bm \Psi)^{t-2} (\bm \Psi \cdot q(\bm Z))^{t-2}$ (up to constant coefficients).  By continuing the same process $t - 2$ times, we can remove the last term completely and obtain finite sum of alternating products.
}

Note that this centering procedure does not depend on the choice of the orthogonal matrix $\bm U$ used to generate the sensing matrix. Furthermore, the number of terms is bounded by $N_t \leq N_{t-1} + 3$, so $N_t \leq 1 + 3t.$ Hence, we have obtained the desired decomposition:
\begin{align}
    (q(\bm Z) \bm \Psi)^t \bm A \bm A^\UT (\bm \Psi \cdot q(\bm Z))^t &  = c_0 \bm I + \sum_{i=1}^{N_t} c_i \altprod_i. \label{alt_prod_decomposition}
\end{align}
Therefore, we can write $(T_3)$ as:
\begin{align*}
    (T_3) & = c_0\frac{\|\bm z\|^2}{m} + \frac{1}{m} \sum_{i = 1}^{N_t} c_i \; \bm z^\UT \altprod_i \bm z = c_0\frac{\|\bm x\|^2}{m} + \frac{1}{m} \sum_{i = 1}^{N_t} c_i \;  \bm z^\UT \altprod_i \bm z. 
\end{align*}
Observe that $\|\bm x\|^2/m \explain{P}{\rightarrow} 1$, and Proposition \ref{proposition: free_probability_qf} guarantees $\bm z^\UT \altprod_i \bm z/m$ converges in probability to the same limit irrespective of whether $\bm U = \bm O$ or $\bm U = \bm H$. Hence, term $(T_3)$ converges in probability to the same limit for both the subsampled Haar sensing and the subsampled Hadamard sensing model. 

Next, we analyze term $(T_4)$. Repeating the arguments we made for the analysis of the term $(T_2)$ we find:
\begin{align*}
    (T_4) & = \frac{\bm z^\UT (q(\bm Z) \bm \Psi)^t \bm A \bm A^\UT(\bm \Psi \cdot q(\bm Z))^t \cdot \bm w}{m} \\&\explain{d}{=} \frac{\|(q(\bm Z) \bm \Psi)^t \bm A \bm A^\UT(\bm \Psi \cdot q(\bm Z))^t \bm z\|_2}{m} \cdot {W} \explain{P}{\rightarrow} 0,
\end{align*}
where ${W \sim \gauss{0}{1}}$.
Finally, we analyze the term $(T_5)$. Using the decomposition \eqref{alt_prod_decomposition} we have:
\begin{align*}
    (T_5) & = c_0 \frac{\|\bm w\|_2^2}{m} + \frac{1}{m} \sum_{i=1}^{N_t} c_i \;  \bm w^\UT \altprod_i \bm w.
\end{align*}
We know that $\|\bm w\|_2^2/m \explain{P}{\rightarrow} 1$. Hence, we focus on analyzing $\bm{w}^\UT \altprod_i \bm w/m$. We decompose this as:
\begin{align*}
    \frac{\bm{w}^\UT \altprod_i \bm w}{m} & =  \frac{w^\UT \altprod_i \bm w - \E[\bm w^\UT \altprod_i \bm w|\altprod_i]}{m} + \frac{\E[\bm w^\UT \altprod_i \bm w|\altprod_i]}{m}. 
\end{align*}
Observe that:
\begin{align*}
    \frac{\E[\bm{w}^\UT \altprod_i \bm w|\altprod_i]}{m} & =  \frac{ \Tr(\altprod_i) }{m} \explain{P}{\rightarrow} 0 \quad \text{(By Proposition \ref{proposition: free_probability_trace})}.
\end{align*}
On the other hand, using the Hanson-Wright Inequality (Fact \ref{fact: hanson_wright}) together with the estimates $${\|\altprod_i\|_{\op} \leq C(\altprod_i)}, \; \|\altprod_i\|_{\mathsf{Fr}} \leq \sqrt{m} \cdot C(\altprod_i),$$
for a fixed constant $C(\altprod_i)$ (independent of $m,n$) depending only on the formula for $\altprod_i$, we obtain $\forall \; \epsilon \; > 0$ :
\begin{align*}
    \P \left( \left| \bm{w}^\UT \altprod_i \bm w - \E[\bm w^\UT \altprod_i \bm w|\altprod_i] \right| > m \epsilon \; \bigg| \; \altprod_i \right) & \leq 2\exp\left(-\frac{c}{C(\altprod_i)}\cdot m\cdot \min(\epsilon,\epsilon^2)\right) \rightarrow 0
\end{align*}
Hence,
\begin{align*}
    \frac{\bm{w}^\UT \altprod_i \bm w - \E[\bm w^\UT \altprod_i \bm w|\altprod_i]}{m} \explain{P}{\rightarrow} 0.
\end{align*}
This implies $(T_5) \explain{P}{\rightarrow} c_0$ for both the models. This proves the limit : $$\plim \frac{\|\hat{\bm x}^{(t)}\|^2_2}{m}$$ exists and is identical for the two sensing models, which concludes the proof of Theorem \ref{thm: main_result}.
\end{description}
\end{proof}
%%%%%%%%%%%%%%%%%%%%%%%%%%%%%%%%%%%%%%%%%%%%%%%%%%%%%%%%%%%%%%
%%%%%%%%%%%%%%%%%%%%%%%%%%%%%%%%%%%%%%%%%%%%%%%%%%%%%%%%%%%%%%
%%%%%%%%%%%%%%%%%%%%%%%%%%%%%%%%%%%%%%%%%%%%%%%%%%%%%%%%%%%%%%

%%%%%%%%%%%%%%%%%%%%%%%%%%%%%%%%%%%%%%%%%%%%%%%%%%%%%%%%%%%%%%
%%%%%%%%%%%%%%%%%%%%%% TOOLBOX %%%%%%%%%%%%%%%%%%%%%%%%%%%%%%
%%%%%%%%%%%%%%%%%%%%%%%%%%%%%%%%%%%%%%%%%%%%%%%%%%%%%%%%%%%%%%
\section{Key Ideas for the Proof of Propositions  \ref{proposition: free_probability_trace} and \ref{proposition: free_probability_qf}}
\label{section: toolbox}
In this section, we introduce some key ideas that are important in the proof of Propositions \ref{proposition: free_probability_trace} and \ref{proposition: free_probability_qf}.   Recall that we wish to analyze the limit in probability of the normalized trace and the quadratic form. A natural candidate for this limit is the limiting value of their expectation:
\begin{align*}
    \plim \frac{1}{m}  \Tr\altprod(\bm \Psi, {\bm Z}) &\explain{?}{=} \lim_{m \rightarrow \infty} \frac{1}{m}  \E\Tr\altprod(\bm \Psi, {\bm Z}), \\
    \plim \frac{\ip{\bm z}{\altprod\bm z}}{m} & \explain{?}{=} \lim_{m \rightarrow \infty} \frac{\E\ip{\bm z}{\altprod\bm z}}{m}.
\end{align*}
In order to show this, one needs to show that the variance of the normalized trace and the normalized quadratic form converge to $0$, which involves analyzing the second moment of these quantities. However, since the analysis of the second moment uses very similar ideas as the analysis of the expectation, we focus on outlining the main ideas in the context of the analysis of expectation.

First, we observe that alternating products can be simplified significantly due to the following property of polynomials of centered Bernoulli random variables.
\begin{lem} \label{lemma: poly_psi_simple}  For any polynomial $p$ such that if $B \sim \bern(\kappa)$, $\E \; p(B-\kappa) = 0$ we have,
\begin{align*}
    p(\bm \Psi) & = (p(1-\kappa) - p(-\kappa)) \cdot \bm \Psi.
\end{align*}
\end{lem}
\begin{proof}
Observe that since $\bm \Psi = \bm U \barB \bm U^\UT$, and $\bm U$ is orthogonal, we have $p(\bm \Psi) = \bm U p(\bm \barB) \bm U^\UT$. Next, observe that:
\begin{align*}
    p(\overline{B}_{ii}) & = p(1-\kappa) B_{ii} + p(-\kappa) (1-B_{ii}) \\
    & = (p(1-\kappa) - p(-\kappa))\cdot \overline{B}_{ii} + \underbrace{\kappa p(1-\kappa) + (1-\kappa) p(-\kappa)}_{=0},
\end{align*}
where the last step follows from the assumption $\E \; p(B-\kappa) = 0$. Hence, $p(\barB) = (p(1-\kappa) - p(-\kappa)) \barB$ and $p(\bm \Psi) = (p(1-\kappa) - p(-\kappa)) \bm \Psi$.
\end{proof}

Hence, without loss of generality we can assume that each of the $p_i$ in an alternating product satisfy $p_i(\xi) = \xi$.

\subsection{Partitions} \label{section: partition_notation}
Note that the expected normalized trace and the expected quadratic form in Propositions \ref{proposition: free_probability_trace} and \ref{proposition: free_probability_qf} can be expanded as follows:
\begin{align*}
    \frac{1}{m}  \E\Tr\altprod(\bm \Psi, {\bm Z}) & = \frac{1}{m} \sum_{a_1,a_2, \dots a_k = 1}^m \E[ (\bm \Psi)_{a_1,a_2} q_1({z}_{a_2}) \cdots q_{k-1}( z_{a_{k}}) (\bm \Psi)_{a_k,a_1}], \\
        \frac{\E\ip{\bm z}{\altprod\bm z}}{m} & = \frac{1}{m} \sum_{\substack{a_{1:k+1} \in [m]}} \E[{z}_{a_1} (\bm \Psi)_{a_1,a_2} q_1({z}_{a_2}) (\bm \Psi)_{a_2, a_3} \cdots  q_{k-1}({z}_{a_k}) (\bm \Psi)_{a_k,a_{k+1}} {z}_{a_{k+1}}].
\end{align*}
\paragraph{Some Notation} Let $\part{}{[k]}$ denote the set of all partitions of a discrete set $[k]$. We use $|\pi|$ to denote the number of blocks in $\pi$. Recall that a partition $\pi \in \part{}{[k]}$  is simply a collection of disjoint subsets of $[k]$ whose union is $[k]$ i.e. $$\pi = \{\blocks_1,\blocks_2 \dots \blocks_{|\pi|}\}, \; \sqcup_{t=1}^{|\pi|} \blocks_t = [k].$$ The symbol $\sqcup$ is exclusively reserved for representing a set as a union of disjoint sets. For any element $s \in [k]$, we use the notation $\pi(s)$ to refer to the block that $s$ lies in. That is, $\pi(s) = \blocks_i$ iff $s \in \blocks_i$. For any $\pi \in \part{}{[k]}$, define the set $\cset{\pi}$ the set of all vectors $\bm a \in [m]^k$ which are constant exactly on the blocks of $\pi$:
    \begin{align*}
        \cset{\pi} & \explain{def}{=} \{ \bm a \in  [m]^k : a_s = a_t  \Leftrightarrow \pi(s) = \pi(t) \}.
    \end{align*}
Consider any $\bm a \in \cset{\pi}$. If $\blocks_i$ is a block in $\pi$, we use $a_{\blocks_i}$ to denote the unique value the vector $\bm a$ assigns to the all the elements of $\blocks_i$.

The rationale for introducing this notation is the observation that:
\begin{align*}
    [m]^k & = \bigsqcup_{\pi \in \part{}{[k]}} \cset{\pi},
\end{align*}
and hence we can write the normalized trace and quadratic forms as:
\begin{subequations}
\label{eq: trace_and_qf}
\begin{align}
    \frac{\E\Tr\altprod(\bm \Psi, {\bm Z})}{m}   & = \frac{1}{m} \sum_{\pi \in \part{}{[k]}} \sum_{\bm a \in \cset{\pi}} \E[(\bm \Psi)_{a_1,a_2} q_1({z}_{a_2}) \cdots q_{k-1}( z_{a_{k}}) (\bm \Psi)_{a_k,a_1}], \\
        \frac{\E\ip{\bm z}{\altprod\bm z}}{m} & = \frac{1}{m} \sum_{\pi \in \part{}{[k+1]}} \sum_{\bm a \in \cset{\pi}} \E[{z}_{a_1} (\bm \Psi)_{a_1,a_2} q_1({z}_{a_2})\cdots  q_{k-1}({z}_{a_k}) (\bm \Psi)_{a_k,a_{k+1}} {z}_{a_{k+1}}].
\end{align}
\end{subequations}
This idea of organizing the combinatorial calculations is due to \citet{tulino2010capacity} and the rationale for doing so will be clear in a moment.

\subsection{Concentration}
\begin{lem}\label{concentration} Let the sensing matrix $\bm A$ be generated by sub-sampling an orthogonal matrix $\bm U$. We have, for any $a,b \in [m]$:
\begin{align*}
    \P \left( \left| \Psi_{ab} \right| \geq \epsilon | \bm U \right) & \leq 4 \exp\left( - \frac{\epsilon^2}{8m \|\bm U\|_\infty^4 } \right).
\end{align*}
\end{lem}
\begin{proof}
Recall that $\bm \Psi = \bm U (\bm B - \kappa \bm I_m) \bm U^\UT$, where the distribution of the diagonal matrix $$\bm B = \diag{B_{11}, B_{22} \dots B_{mm}}$$ is described as follows: First draw a uniformly random subset $S \subset [m]$ with $|S| = n$ and set:
\begin{align*}
    B_{ii} & = \begin{cases} 0 &: i \not\in S \\ 1 &: i \in S \end{cases}.
\end{align*}
Due to the constraint that $\sum_{i=1}^m B_{ii} = n$, these random variables are not independent. In order to address this issue we couple $\bm B$ with another random diagonal matrix $\tilde{\bm B}$ generated as follows:
\begin{enumerate}
    \item First sample $N \sim \bnomdistr{m}{\kappa}$.
    \item Sample a subset $\tilde{S} \subset [m]$ with $|\tilde{S}| = N$ as follows:
    \begin{itemize}
        \item If $N\leq n$, then set $\tilde{S}$ to be a uniformly random subset of $S$ of size $N$.
        \item If $N>n$ first sample a uniformly random subset $A$ of $S^c$ of size $N-n$ and set $\tilde{S} = S \cup A $.
    \end{itemize}
    \item Set $\tilde{\bm B}$ as follows:
    \begin{align*}
        \tilde{B}_{ii} & = \begin{cases} 0 &: i \not\in \tilde{S} \\ 1 &: i \in \tilde{S}. \end{cases}.
    \end{align*}
\end{enumerate}
It is easy to check that conditional on $N$, $\tilde{S}$ is a uniformly random subset of $[m]$ with cardinality $N$. Since $N \sim \bnomdistr{m}{\kappa}$, we have $\tilde{B}_{ii} \explain{i.i.d.}{\sim} \bern(\kappa)$. Define:
\begin{align}
    T &\explain{def}{=} \Psi_{ab} = \bm u_a^\UT (\bm B- \kappa \bm I_m) \bm u_b  = \sum_{i=1}^m u_{ai} u_{bi} (B_{ii} - \E B_{ii}), \\ \tilde{T} &\explain{def}{=} \bm u_a^\UT (\tilde{\bm B}- \kappa \bm I_m) \bm u_b =  \sum_{i=1}^m u_{ai} u_{bi} ( \tilde{B}_{ii} - \E \tilde{B}_{ii}).
\end{align}
\milad{Observe that: $$|T  - \tilde{T}| = |\bm u_a^\UT (\bm B - \tilde{\bm B}) \bm u_b| = |\ip{\bm B - \tilde{\bm B}}{ \bm u_b \bm u_a^T} | \leq \|\bm B - \tilde{\bm B}\|_{1} \|\bm u_b \bm u_a^T\|_\infty \leq |N-n| \|\bm U\|_\infty^2.$$ In the above display,  the first inequality is obtained by Holder inequality, and the second one is obtained by the fact that $$\|\bm B - \tilde{\bm B}\|_{1}  = \sum\limits_{i=1}^m |B_{ii}  - \tilde{B}_{ii}| \leq |(S \backslash \tilde{S}) \cup (\tilde{S} \backslash S)| \leq \abs{N - n},$$ and $\|\bm u_b \bm u_a^T\|_\infty \leq \|\bm U\|_\infty^2$.
}
 Hence,
\begin{align*}
    \P \left( |T| \geq \epsilon \right) & \leq \P\left( |\tilde T| \geq \frac{\epsilon}{2} \right) + \P \left( |T-\tilde T| \geq \frac{\epsilon}{2} \right) \\
    & =  \P\left( |\tilde T| \geq \frac{\epsilon}{2} \right) + \P \left( |N-\E N| \geq \frac{\epsilon}{2 \|\bm U\|_\infty^2} \right) \\
    & \explain{(a)}{\leq} 4 \exp\left( - \frac{\epsilon^2}{8m\|\bm U\|_\infty^4 } \right).
\end{align*}
In the step marked (a), we used Hoeffding's Inequality.
\end{proof}
Hence the above lemma shows that,
\begin{align*}
    \|\bm \Psi\|_\infty & \leq O \left( \sqrt{m} \|\bm U\|_\infty^2 \polylog(m) \right),
\end{align*}
with high probability. Recall that in the subsampled Hadamard model $\bm U = \bm H$ and $\|\bm H\|_\infty = 1/\sqrt{m}$. Similarly, in the subsampled Haar model $\bm U = \bm O$ and $\|\bm O\|_\infty \leq O(\polylog(m)/\sqrt{m})$.  Hence, we expect:
\begin{align} \label{heuristic: concentration}
     \|\bm \Psi\|_\infty & \leq O \left( \frac{ \polylog(m) }{\sqrt{m}}\right), \; \text{ with high probability}.
\end{align}
\subsection{Mehler's Formula}
Note that in order to compute the expected normalized trace and quadratic form as given in \eqref{eq: trace_and_qf}, we need to compute:
\begin{align*}
    & \E[(\bm \Psi)_{a_1,a_2} q_1({z}_{a_2}) \cdots q_{k-1}( z_{a_{k}}) (\bm \Psi)_{a_k,a_1}],
    \\ &
    \E[{z}_{a_1} (\bm \Psi)_{a_1,a_2} q_1({z}_{a_2}) (\bm \Psi)_{a_2, a_3} \cdots  q_{k-1}({z}_{a_k}) (\bm \Psi)_{a_k,a_{k+1}} {z}_{a_{k+1}}].
\end{align*}
Note that by the tower property:
\begin{align*}
    \MoveEqLeft
    \E[(\bm \Psi)_{a_1,a_2} q_1({z}_{a_2}) \cdots q_{k-1}( z_{a_{k}}) (\bm \Psi)_{a_k,a_1}] = &
    \\ 
    & \E \left[ (\bm \Psi)_{a_1,a_2} \cdots (\bm \Psi)_{a_k,a_1}\E[ q_1({z}_{a_2}) \cdots q_{k-1}( z_{a_{k}})| \bm A ]  \right],
\end{align*}
and analogously for $\E[{z}_{a_1} (\bm \Psi)_{a_1,a_2} q_1({z}_{a_2}) (\bm \Psi)_{a_2, a_3} \cdots  q_{k-1}({z}_{a_k}) (\bm \Psi)_{a_k,a_{k+1}} {z}_{a_{k+1}}]$.
Suppose that $\bm a \in \cset{\pi}$ for some $\pi \in \part{}{[k]}$. Let $\pi = \blocks_1 \sqcup \blocks_2 \cdots \sqcup \blocks_{|\pi|}$.  Define:
\begin{align*}
    F_{\blocks_i}(\xi) & = \prod_{\substack{j \in \blocks_i\\ j \neq 1}} q_{j-1}(\xi).
\end{align*}
Then, we have:
\begin{align*}
    \E[ q_1({z}_{a_2}) \cdots q_{k-1}( z_{a_{k}})| \bm A ] & = \E\left[ \prod_{i=1}^{|\pi|} F_{\blocks_i}(z_{a_{\blocks_i}})\bigg| \bm A\right].
\end{align*}
In order to compute the conditional expectation we observe that conditionally on $\bm A$, $\bm z$ is a zero mean Gaussian vector with covariance:
\begin{align*}
    \E[\bm z \bm z^\UT | \bm A] = \frac{1}{\kappa} \bm A \bm A^\UT = \frac{1}{\kappa} \bm U \bm B \bm U^\UT = \bm I + \frac{\bm \Psi}{\kappa}.
\end{align*}
Note that since $a_{\blocks_i} \neq a_{\blocks_j}$ for $i \neq j$, we have as a consequence of \eqref{heuristic: concentration}, $\{z_{a_{\blocks_i}}\}_{i=1}^{|\pi|}$ are weakly correlated Gaussians. Hence we expect,
\begin{align*}
    \E[ q_1({z}_{a_2}) \cdots q_{k-1}( z_{a_{k}})| \bm A ] & = \prod_{i=1}^{|\pi|} \E_{Z \sim \gauss{0}{1}} F_{\blocks_i}(Z) + \text{ A small error term},
\end{align*}
where the error term is a term that goes to zero as $m \rightarrow \infty$.
Mehler's formula given in the proposition below provides an explicit formula for the error term. Observe that in \eqref{eq: trace_and_qf}:
\begin{enumerate}
    \item the sum over $\pi \in \part{}{[k]}$ cannot cause the error terms to add up since $|\part{}{[k]}|$ is a constant depending on $k$ but independent of $m$.
    \item On the other hand, the sum over $\bm a \in \cset{\pi}$ can cause the errors to add up since:
    \begin{align*}
        |\cset{\pi}| & = m \cdot (m-1) \cdots (m- |\pi| + 1).
    \end{align*}
\end{enumerate}
It is not obvious right away how accurately the error must be estimated, but it turns out that for the proof of Proposition \ref{proposition: free_probability_trace} it suffices to estimate the order of magnitude of the error term. For the proof of Proposition \ref{proposition: free_probability_qf} we need to be more accurate and the leading order term in the error needs to be tracked precisely. 

Before we state Mehler's formula we recall some preliminaries regarding Fourier analysis on the Gaussian space. Let $Z \sim \gauss{0}{1}$. Let $f: \R \rightarrow \R$ be such that $\E f^2(Z) < \infty$, i.e. $f \in L^2(\gauss{0}{1})$. The Hermite polynomials $\{H_j: j \in \N_0\}$ form an orthogonal polynomial basis for $L^2(\gauss{0}{1})$. The polynomial $H_j$ is a degree $j$ polynomial. They satisfy the orthogonality property:
\begin{align*}
    \E H_i(Z) H_j(Z) & = i! \cdot  \delta_{ij}.
\end{align*}
The first few Hermite polynomials are given by:
\begin{align*}
    H_0(z) = 1, \; H_1(z) = z, \; H_2(z) = z^2 - 1.
\end{align*}

\begin{prop} [\citet{mehler1866ueber,slepian1972symmetrized}] \label{proposition: mehler} Consider a $k$ dimensional Gaussian vector $\bm z \sim \gauss{\bm 0}{\bm \Sigma}$, such that $\Sigma_{ii} = 1$ for all $i \in [k]$. Let $f_1,f_2, \dots, f_k: \R \rightarrow \R$ be $k$ arbitrary functions whose absolute value can be upper bounded by a polynomial. Then, for any $t \in \N$ we have,
\begin{align*}
   \left| \E \left[ \prod_{i=1}^k f_i(z_i) \right] - \sum_{\substack{\bm w \in \weightedG{k}\\ \|\bm w\| \leq t }}   \left( \prod_{i=1}^k \hat{f}_i(\degree_i(\bm w)) \right) \cdot \frac{\bm \Sigma^{\bm w}}{\bm w!} \right| & \leq C \left(1 + \frac{1}{\lambda_{\min}^{4t+4}(\bm \Sigma)} \right)  \left( \max_{i \neq j} |\Sigma_{ij}| \right)^{t+1},
\end{align*}
where:
\begin{enumerate}
    \item $\weightedG{k}$ denotes the set of undirected weighted graphs with non-negative integer weights on $k$ nodes with no self loops.
    \item An element $\bm w \in \weightedG{k}$ is represented by a $k \times k$ symmetric matrix $\bm w$ with $w_{ij} = w_{ji} \in \N \cup \{0\}$, and $w_{ii} = 0$.
    \item $\degree_i(\bm w)$ denotes the degree of node $i$: $\degree_i(\bm w) = \sum_{j=1}^k w_{ij}$.
    \item $\|\bm w\|$ denotes the total weight of the graph defined as:
    \begin{align*}
        \|\bm w\| & \explain{def}{=} \sum_{i<j} w_{ij}  = \frac{1}{2} \sum_{i=1}^k \degree_i(\bm w).
    \end{align*}
    \item The coefficients $\hat{f}_i(j)$ are defined as: $\hat{f}_i(j) = \E f_i(Z) H_j(Z)$ where $Z \sim \gauss{0}{1}$.
    \item $\bm \Sigma^{\bm w}, \bm w!$ denote the entry-wise powering and factorial:
    \begin{align*}
        \bm \Sigma^{\bm w} = \prod_{i<j} \Sigma_{ij}^{w_{ij}}, \; \bm w! = \prod_{i<j} w_{ij} !
    \end{align*}
    \item $C=C_{t,k, f_{1:k}}$ is a finite constant depending only on the $t,k$, and the functions $f_{1:k}$ but is independent of $\bm \Sigma$.
\end{enumerate}
\end{prop}

This result is essentially due to \citet{mehler1866ueber} in the case $k=2$, and the result for general $k$ was obtained by \citet{slepian1972symmetrized}. Actually the results of these authors show that the probability density function of $\gauss{\bm 0}{\bm \Sigma}$ denoted by $\gpdf{\bm z}{\bm \Sigma}$ has the following Taylor expansion around $\bm \Sigma = \bm I_k$:
\begin{align*}
    \gpdf{\bm z}{\bm \Sigma} & = \gpdf{\bm z}{\bm I_k} \cdot \left( \sum_{\bm w \in \weightedG{k}} \frac{\bm \Sigma^{\bm w}}{\bm w !} \cdot \prod_{i=1}^k H_{\degree_i(\bm w)}(z_i)\right).
\end{align*}
In Appendix \ref{appendix: mehler} of the supplementary materials we check that this Taylor's expansion can be integrated, and estimate the truncation error to obtain Proposition \ref{proposition: mehler}.

At this point, we have introduced all the tools used in the proof of Proposition \ref{proposition: free_probability_trace} and we refer the reader to Section \ref{proof_trace} for the proof of Proposition \ref{proposition: free_probability_trace}.

\subsection{Central Limit Theorem} \label{section: clt}
We introduce the following definition.
\begin{defn}[Matrix Moment] \label{def: matrix moment} Let $\bm M$ be a symmetric matrix. Given:
\begin{enumerate}
    \item A partition $\pi \in \part{}{[k]}$ with blocks $\pi = \{\blocks_1, \blocks_2, \cdots ,\blocks_{|\pi|}\}$.
    \item A $k \times k$ symmetric weight matrix $\bm w \in \weightedG{k}$ with non-negative valued entries and $w_{ii} = 0 \;  \forall \; i \in [k]$.
    \item A vector $\bm a \in \cset{\pi}$.
\end{enumerate}
Define the $(\bm w, \pi, \bm a)$ - matrix moment of the matrix $\bm M$ as:
\begin{align*}
    \matmom{\bm M}{\bm w}{\pi}{\bm a} & \explain{def}{=} \prod_{i,j \in [k], i<j}  M_{a_i,a_j}^{w_{ij}}.
\end{align*}
By defining:
\begin{align*}
    W_{st}(\bm w, \pi) &\explain{def}{=} \sum_{\substack{i,j \in [k], i<j \\ \{\pi(i),\pi(j)\} = \{\blocks_s,\blocks_t\}}} w_{ij},
\end{align*}
we can write $\matmom{\bm M}{\bm w}{\pi}{\bm a}$ in the form:
\begin{align*}
    \matmom{\bm M}{\bm w}{\pi}{\bm a} & = \prod_{\substack{s,t \in [|\pi|] \\ s \leq t}} M_{a_{\blocks_s},a_{\blocks_t}}^{W_{st}(\bm w, \bm \pi)}.
\end{align*}
\end{defn}

\begin{rem}[Graph Interpretation] It is often useful to interpret the tuple $({\bm w},{\pi},{\bm a})$ in terms of graphs:
\begin{enumerate}
    \item $\bm w$ represents the adjacency matrix of an undirected weighted graph on the vertex set $[k]$ with no self-edges $(w_{ii} = 0)$. We say an edge exists between nodes $i,j \in [k]$ if $w_{ij} \geq 1$ and the weight of the edge is given by $w_{ij}$.
    \item The partition $\pi$ of the vertex set $[k]$ represents a community structure on the graph. Two vertices $i,j \in [k]$ are in the same community iff $\pi(i) = \pi(j)$.
    \item $\bm a$ represents a labelling of the vertices $[k]$ with labels in the set $[m]$ which respects the community structure. 
    \item The weights $W_{st}(\bm w, \pi)$ simply denote the total weight of edges between communities $s,t$.
\end{enumerate}
\end{rem}

The rationale for introducing this definition is as follows: When we use Mehler's formula to compute $\E[ q_1({z}_{a_2}) \cdots q_{k-1}( z_{a_{k}})| \bm A ]$ and $\E[ z_{a_1} q_1({z}_{a_2}) \cdots q_{k-1}( z_{a_{k}}) z_{a_{k+1}}| \bm A ]$, and substitute the resulting expression in \eqref{eq: trace_and_qf}, it expresses:
\begin{align*}
    \frac{\Tr\altprod(\bm \Psi, {\bm Z})}{m}, \; \frac{\E\ip{\bm z}{\altprod\bm z}}{m},
\end{align*}
in terms of the matrix moments $\matmom{\bm \Psi}{\bm w}{\pi}{\bm a}$.  

For the proof of Proposition \ref{proposition: free_probability_trace} it suffices to upper bound $|\matmom{\bm \Psi}{\bm w}{\pi}{\bm a}|$. We do so in the following lemma.

\begin{lem} \label{lemma: matrix_moment_ub} Consider an arbitrary matrix moment $\matmom{\bm \Psi}{\bm w}{\pi}{\bm a}$ of $\bm \Psi$. There exists a universal constant $C$  (independent of $m, \bm a, \pi, \bm w$) such that,
\begin{align*}
    \E |\matmom{\bm \Psi}{\bm w}{\pi}{\bm a}|  & \leq \left( \sqrt{\frac{C \|\bm w\| \log^2(m)}{m}} \right)^{\|\bm w\|},
\end{align*}
for both the sub-sampled Haar and the sub-sampled Hadamard sensing model. 
\end{lem}

The claim of the lemma is not surprising in light of \eqref{heuristic: concentration}. The complete proof follows from the concentration inequality in Lemma \ref{concentration}, which can be found in Appendix \ref{proof: matrix_moment} of the supplementary materials.

On the other hand, to prove Proposition \ref{proposition: free_probability_qf} we need a more refined analysis and we need to estimate the leading order term in $\E \matmom{\bm \Psi}{\bm w}{\pi}{\bm a}$. In order to do so, we first consider any fixed entry of $\sqrt{m}\bm \Psi$:
\begin{align*}
    \sqrt{m} \Psi_{ab} = \sqrt{m} (\bm U \barB \bm U^\UT)_{ab} = \sum_{i=1}^m \sqrt{m} \cdot u_{ai} \cdot  u_{bi} (B_{ii} - \kappa).
\end{align*}
Observe that:
\begin{enumerate}
    \item $B_{ii} - \kappa$ are centered and weakly dependent.
    \item $\sqrt{m} u_{ai} u_{bi} = O(m^{-\frac{1}{2}})$ under both the sub-sampled Haar model and the sub-sampled Hadamard model. 
\end{enumerate}
Consequently, we expect $\sqrt{m} \Psi_{ab}$ converges to a Gaussian random variable and hence, we expect that:
\begin{align*}
    \E \matmom{\sqrt{m}\bm \Psi}{\bm w}{\pi}{\bm a}
\end{align*}
converges to a suitable Gaussian moment. 
In order to show that the normalized quadratic form $\E\ip{\bm z}{\altprod\bm z}/m$ converges to the same limit under both the sensing models, we need to understand  what is the limiting value of $\E \matmom{\sqrt{m}\bm \Psi}{\bm w}{\pi}{\bm a}$ under both the models. Understanding this uses the following simple but important property of Hadamard matrices. 

\begin{lem} \label{lemma: hadamard_key_property} For any $i,j \in [m]$, we have:
\begin{align*}
    \sqrt{m}\bm h_i \odot \bm h_j & = \bm h_{i \oplus j},
\end{align*}
where $\odot$ denotes the entry-wise multiplication of vectors, and $i \oplus j \in [m]$ denotes the result of the following computation:
\begin{description}
\item [Step 1: ] Compute $\bm i, \bm j \in \{0,1\}^m$ which are the binary representations of $(i-1)$ and $(j-1)$ respectively.
\item [Step 2: ] Compute $\bm i + \bm j$ by adding $\bm i, \bm j$ bit-wise (modulo 2). 
\item [Step 3: ]  Compute the number in $[0:m-1]$ whose binary representation is given by $\bm i + \bm j$. 
\item [Step 4: ] Add one to the number obtained in Step 3 to obtain $i \oplus j \in [m]$.
\end{description}
\end{lem}
\begin{proof}
Recall by the definition of the Hadamard matrix, we have,
\begin{align*}
    h_{ik} & = \frac{1}{\sqrt{m}} (-1)^{\ip{\bm i}{\bm k}}, \; h_{jk}  = \frac{1}{\sqrt{m}} (-1)^{\ip{\bm j}{\bm k}}.
\end{align*}
Hence,
\begin{align*}
    \sqrt{m} (\bm h_i \odot \bm h_j)_k & = \frac{(-1)^{\ip{\bm i + \bm j}{\bm k}}}{\sqrt{m}} = (\bm h_{i \oplus j})_k,
\end{align*}
as claimed. 
\end{proof}

Due to the structure in Hadamard matrices, $\E \matmom{\sqrt{m}\bm \Psi}{\bm w}{\pi}{\bm a}$ might not always converge to the same limit under the subsampled Haar and the Hadamard models.  There are two kinds of exceptions:
\begin{description}
\item [Exception 1: ] Note that for the subsampled Hadamard Model,
\begin{align*}
    \sqrt{m} \Psi_{aa} = \sqrt{m} \sum_{i=1}^m \overline{B}_{ii} |h_{ai}|^2 = \frac{1}{\sqrt{m}} \sum_{i=1}^m \overline{B}_{ii} = 0.
\end{align*}
In contrast, under the subsampled Haar model, it can be shown that $\sqrt{m} \Psi_{aa}$ converges to a non-degenerate Gaussian. These exceptions are ruled out by requiring the weight matrix $\bm w$ to be disassortative with respect to $\pi$ (See definition below). 
\item [Exception 2: ] Define $\overline{\bm b} \in \R^m$ to be the vector formed by the diagonal entries of $\barB$. Observe that for the subsampled Hadamard model:
\begin{align*}
    \sqrt{m} \Psi_{ab} & = \ip{\overline{\bm b}}{\sqrt{m}\bm h_{a} \odot \bm h_{b}} = \ip{\overline{\bm b}}{\bm h_{a \oplus b}}.
\end{align*}
Consequently, if two distinct pairs $(a_1,b_1)$ and $(a_2, b_2)$ are such that $a_1 \oplus b_1 = a_2 \oplus b_2$, then $\sqrt{m} \Psi_{a_1,b_1}$ and $\sqrt{m} \Psi_{a_2,b_2}$ are perfectly correlated in the subsampled Hadamard model. In contrast, unless $(a_1,b_1) = (a_2,b_2)$, it can be shown they are asymptotically uncorrelated in the subsampled Haar model. This exception is ruled out by requiring the labelling $\bm a$ to be conflict free with respect to $(\bm w,\pi)$ (defined below).
\end{description}

\begin{defn}[Disassortative Graphs] We say the weight matrix $\bm w$ is disassortative with respect to the partition $\pi$ if: $\forall \; i,j \in [k], \; i < j$ such that $\pi(i) = \pi(j)$, we have $w_{ij} = 0$. This is equivalent to $W_{ss}(\bm w, \pi) = 0$ for all $s \in [|\pi|]$. In terms of the graph interpretation, this means that there are no intra-community edges in the graph. For any $\pi \in \part{}{[k]}$,we denote the set of all weight matrices disassortative with respect to $\pi$ by $\weightedGnum{\pi}{DA}$:
\begin{align*}
    \weightedGnum{\pi}{DA} & \explain{def}{=} \{\bm w \in \weightedG{k}: W_{ss}(\bm w, \pi) = 0 \; \forall \; s \; \in [|\pi|] \}.
\end{align*}
\end{defn}

\begin{defn}[Conflict Freeness] Let $\pi \in \part{}{[k]}$ be  a partition and let $\bm w \in \weightedGnum{\pi}{DA}$ be a weight matrix disassortative with respect to $\pi$. Let $s_1 < t_1$ and $s_2 < t_2$ be distinct pairs of communities: $s_1,s_2,t_1,t_2 \in [|\pi|]$, $(s_1,t_1) \neq (s_2,t_2)$. We say a labelling $\bm a \in \cset{\pi}$ has a conflict between distinct community pairs $(s_1,t_1)$ and $(s_2,t_2)$ if: 
\begin{enumerate}
    \item $W_{s_1,t_1}(\bm w, \pi) \geq 1, \; W_{s_2,t_2}(\bm w, \pi) \geq 1$.
    \item $a_{\blocks_{s_1}} \oplus a_{\blocks_{t_1}} = a_{\blocks_{s_2}} \oplus a_{\blocks_{t_2}}$.
\end{enumerate}
We say a labelling $\bm a$ is conflict-free if it has no conflicting community pairs.  The set of all conflict free labellings of $(\bm w, \pi)$ is denoted by $\labelling{CF}(\bm w, \pi)$.
\end{defn}

The following two propositions show that if Exception 1 and Exception 2 are ruled out, then indeed $\E \matmom{\sqrt{m}\bm \Psi}{\bm w}{\pi}{\bm a}$ converges to the same Gaussian moment under both the subsampled Haar and the Hadamard models.

\begin{prop}\label{prop: clt_random_ortho} Consider the sub-sampled Haar model $(\bm \Psi = \bm O \barB \bm O^\UT)$. Fix a partition $\pi \in \part{}{k}$ and a weight matrix $\bm w \in \weightedG{k}$. Then, there exist constants $K_1,K_2,K_3> 0$ depending only on $\|\bm w\|$ (independent of $m$), such that for any $\bm a \in \cset{\pi}$ we have:
\begin{align*}
    \left| \E \; \matmom{\sqrt{m} \bm \Psi}{\bm w}{\pi}{\bm a} - \prod_{\substack{s,t \in [|\pi|] \\ s \leq t}} \E \left[ Z_{st}^{W_{st}(\bm w, \pi)} \right] \right| & \leq \frac{K_1 \log^{K_2}(m)}{m^\tbd}, \; \forall \; m \geq K_3.
\end{align*}
In the above display, $Z_{st}, \; s \leq t, \; s,t\; \in \; [|\pi|]$ are independent Gaussian random variables with the distribution:
\begin{align*}
    Z_{st} \sim \begin{cases} s < t: &  \gauss{0}{\kappa(1-\kappa)} \\ s = t : & \gauss{0}{2 \kappa(1-\kappa)} \end{cases}.
\end{align*}
\end{prop}

\begin{prop}\label{prop: clt_hadamard} Consider the sub-sampled Hadamard model $(\bm \Psi = \bm H \barB \bm H^\UT)$. Fix a partition $\pi \in \part{}{k}$ and a weight matrix $\bm w \in \W^{k \times k}$.  Then,
\begin{enumerate}
    \item Suppose that $\bm w \not \in \weightedGnum{\pi}{DA}$, then, 
    \begin{align*}
        \matmom{\sqrt{m} \bm \Psi}{\bm w}{\pi}{\bm a} = 0.
    \end{align*}
    \item Suppose that $\bm w \in \weightedGnum{\pi}{DA}$. Then, there exist constants $K_1,K_2,K_3> 0$ depending only on $\|\bm w\|$ (independent of $m$), such that for any {conflict free labelling} $\bm a \in \labelling{CF}(\bm w, \pi)$, we have:
\begin{align*}
    \left| \E \; \matmom{\sqrt{m} \bm \Psi}{\bm w}{\pi}{\bm a} - \prod_{\substack{s,t \in [|\pi|] \\ s < t}} \E \left[ Z_{\kappa}^{W_{st}(\bm w, \pi)} \right] \right| & \leq \frac{K_1 \log^{K_2}(m)}{m^\tbd}, \; \forall \; m \geq K_3.
\end{align*}
In the above display, $Z_\kappa \sim \gauss{0}{\kappa(1-\kappa)}$.
\end{enumerate}
\end{prop}

The proof of these Propositions can be found in Appendix \ref{proof: clt_props} in the supplementary materials. The proofs use a coupling argument to replace the weakly dependent diagonal matrix $\barB$ with a i.i.d. diagonal entries (as in the proof of Lemma \ref{concentration}) along with a classical Berry-Esseen inequality due to \citet{bhattacharya1975errors}.

Finally, in order to finish the proof of Proposition \ref{proposition: free_probability_qf} regarding the universality of the normalized quadratic form we need to argue that the number of exceptional labellings under which $\E \matmom{\sqrt{m}\bm \Psi}{\bm w}{\pi}{\bm a}$ doesn't converge to the same Gaussian moment under the sub-sampled Hadamard and Haar models are an asymptotically negligible fraction of the total number of labellings.  

\begin{lem}  \label{lemma: cf_size_bound} Let $\pi \in \part{}{[k]}$ be  a partition and $\bm w \in \weightedGnum{\pi}{DA}$ be a weight matrix disassortative with respect to $\pi$. We have, $|\cset{\pi}\backslash \labelling{CF}(\bm w, \pi)| \leq |\pi|^4 \cdot m^{|\pi|-1}$, and
\begin{align*}
    \lim_{m \rightarrow \infty } \frac{\labelling{CF}(\bm w, \pi)}{m^{|\pi|}}  = 1.
\end{align*}
\end{lem}
\begin{proof}
Let $(s_1,t_1) \neq (s_2,t_2)$ be two distinct community pairs such that:
\begin{align*}
    W_{s_1,t_1}(\bm w, \pi) \geq 1, \; W_{s_2,t_2}(\bm w, \pi) \geq 1.
\end{align*}
Let $\labelling{(s_1,t_1; s_2,t_2)}(\bm w, \pi)$ denote the set of all labellings $\bm a \in \cset{\pi}$ that have a conflict between distinct community pairs $(s_1,t_1)$ and $(s_2,t_2)$:
\begin{align*}
    \labelling{(s_1,t_1; s_2,t_2)}(\bm w, \pi) &\explain{def}{=} \{\bm a \in \cset{\pi}: a_{\blocks_{s_1}} \oplus a_{\blocks_{t_1}} = a_{\blocks_{s_2}} \oplus a_{\blocks_{t_2}} \}.
\end{align*}
Then, we note that
\begin{align*}
    \cset{\pi}\backslash \labelling{CF}(\bm w, \pi) & = \bigcup_{s_1,t_1,s_2,t_2} \labelling{(s_1,t_1; s_2,t_2)}(\bm w, \pi),
\end{align*}
where the union ranges over $s_1,t_1,s_2,t_2$ such that $1 \leq s_1 < t_1  \leq |\pi|, 1 \leq s_2 < t_2  \leq |\pi|$   and $(s_1,t_1) \neq (s_2,t_2)$ and $ W_{s_1,t_1}(\bm w, \pi) \geq 1,  W_{s_2,t_2}(\bm w, \pi) \geq 1$. Next, we bound $|\labelling{(s_1,t_1; s_2,t_2)}(\bm w, \pi)|$. Since we know that $(s_1,t_1) \neq (s_2,t_2)$ and $s_1 < t_1$ and $s_2 < t_2$ out of the 4 indices $s_1,t_1,s_2,t_2$, there must be one index which is different from all the others. Let us assume that this index is $t_2$ (the remaining cases are analogous). To count $|\labelling{(s_1,t_1; s_2,t_2)}(\bm w, \pi)|$ we assign labels to all blocks of $\pi$ except $t_2$. The number of ways of doing so is at most $m^{|\pi| -1}$. After we do so, we note that $a_{\blocks_{t_2}}$ is uniquely determined by the constraint:
\begin{align*}
    a_{\blocks_{s_1}} \oplus a_{\blocks_{t_1}} = a_{\blocks_{s_2}} \oplus a_{\blocks_{t_2}}.
\end{align*}
Hence, $|\labelling{(s_1,t_1; s_2,t_2)}(\bm w, \pi)| \leq m^{|\pi| -1}$. Therefore,
\begin{align*}
     |\cset{\pi}\backslash \labelling{CF}(\bm w, \pi)| & = \sum_{s_1,t_1,s_2,t_2} |\labelling{(s_1,t_1; s_2,t_2)}(\bm w, \pi)| \leq |\pi|^4 m^{|\pi| -1}.
\end{align*}
Finally, we note that,
\begin{align*}
    |\cset{\pi}| - |\cset{\pi}\backslash \labelling{CF}(\bm w, \pi)| & = |\labelling{CF}(\bm w, \pi)| \leq |\cset{\pi}|.
\end{align*}
$|\cset{\pi}|$ is given by:
\begin{align*}
    |\cset{\pi}| & = m(m-1) \cdots (m-|\pi| + 1) = m^{|\pi|}\cdot (1+o_m(1)).
\end{align*}
Combining this with the already obtained upper bound $|\cset{\pi}\backslash \labelling{CF}(\bm w, \pi)| \leq |\pi|^4 \cdot m^{|\pi|-1}$, we obtain the second claim of the lemma. 
\end{proof}

We now have all the tools required to finish the proof of Proposition \ref{proposition: free_probability_qf} and we refer the reader to Section \ref{section: qf_proof} for the proof of this result. 

%%%%%%%%%%%%%%%%%%%%%%%%%%%%%%%%%%%%%%%%%%%%%%%%%%%%%%%%%%%%%%
%%%%%%%%%%%%%%%%%%%%%%%%%%%%%%%%%%%%%%%%%%%%%%%%%%%%%%%%%%%%%%
%%%%%%%%%%%%%%%%%%%%%%%%%%%%%%%%%%%%%%%%%%%%%%%%%%%%%%%%%%%%%%

%%%%%%%%%%%%%%%%%%%%%%%%%%%%%%%%%%%%%%%%%%%%%%%%%%%%%%%%%%%%%
%%%%%%%%%%%%%%%%%%% PROOF-TRACE %%%%%%%%%%%%%%%%%%%%%%%%%%%%%
%%%%%%%%%%%%%%%%%%%%%%%%%%%%%%%%%%%%%%%%%%%%%%%%%%%%%%%%%%%%%
%\color{OliveGreen}
\section{Proof of Proposition \ref{proposition: free_probability_trace}}
\label{proof_trace}
In this Section we prove Proposition \ref{proposition: free_probability_trace}.

Let us consider a fixed alternating product $\altprod(\bm \Psi, \bm Z)$ as given in Definition \ref{def: alternating_product}.
As a consequence of Lemma \ref{lemma: poly_psi_simple} we can assume that all the polynomials $p_i(\xi) = \xi$.
We begin by stating a few intermediate lemmas which will be used to prove Proposition \ref{proposition: free_probability_trace}. 
\begin{lem}[A high probability event] \label{lemma: trace_good_event} Let $\bm U$ denote the $m \times m$ orthogonal matrix used to generate the sensing matrix $\bm $. Define the event:
\begin{align} \label{eq: good_event_trace_lemma}
    %\mathcal{E} &= \mathcal{E}_1 \cap \mathcal{E}_2, \\
    \mathcal{E}& = \left\{ \max_{i \neq j} | (\bm A \bm A^\UT)_{ij}| \leq \sqrt{32 \cdot m \cdot  \|\bm U\|_\infty^4 \cdot  \log(m)}, \right. \nonumber\\ & \hspace{2cm} \left. \max_{i \in [m]} | (\bm A \bm A^\UT)_{ii} - \kappa | \leq \sqrt{32 \cdot m \cdot  \|\bm U\|_\infty^4 \cdot  \log(m)} \right\}.
    %\mathcal{E}_2 & = \left\{ \max_{i,j \in [m]}  | p_\ell(\bm \Psi)_{ij} |  \leq \sqrt{32 \cdot m \cdot \|\bm U\|_\infty^4 \cdot \|p_\ell\|_\infty^2 \cdot \log(m)} \; \forall \; \ell = 0,1,2 \dots k \right\}.
\end{align}
Then,
\begin{align*}
    \P(\mathcal{E} | \bm U) & \geq 1- 4/m^{2}.
\end{align*}
Furthermore, for the subsampled Haar model, when $\bm U = \bm O \sim \unif{\O(m)}$, we have: 
\begin{align*}
    \P\left( \left\{\|\bm O\|_\infty \leq \sqrt{\frac{8\log(m)}{m}} \right\} \cap \mathcal{E} \right) \geq 1 - 6/m^2.
\end{align*}
\end{lem}
The above Lemma follows from the concentration result in Lemma \ref{concentration} and a union bound. Complete details are provided in Appendix \ref{appendix: free_probability_trace} in the supplementary materials.
\begin{lem}[A Continuity Estimate] \label{lemma: continuity_trace} Let $\altprod(\bm \Psi, \bm Z)$ be an alternating product of the matrices $\bm \Psi, \bm Z$ (see Definition \ref{def: alternating_product}). Then the map $\bm Z \mapsto \Tr\altprod(\bm \Psi, \bm Z)/m$ is Lipchitz in $Z$, i.e. for any two diagonal matrices $\bm Z = \diag{z_1,z_2 \dots , z_m}, \; \bm Z^\prime = \diag{z_1^\prime,z_2^\prime \dots , z_m^\prime}$ we have:
\begin{align*}
    \left| \frac{\Tr \altprod(\bm \Psi, \bm Z)}{m} - \frac{\Tr \altprod(\bm \Psi, \bm Z^\prime)}{m} \right| & \leq \frac{C(\altprod)}{\sqrt{m}} \cdot \|\bm Z - \bm Z^\prime\|_{\fr}, 
\end{align*}
where $C(\altprod)$ denotes a constant depending only on the formula for the alternating product $\altprod$ (independent of $m,n$).
\end{lem}
This lemma follows from a straightforward computation provided in \ref{appendix: free_probability_trace} in the supplementary materals.
\begin{lem}[Analysis of Expectation] \label{lemma: trace_expectation} Let the sensing matrix $\bm A$ be drawn either from the subsampled Haar model or be generated using a deterministic orthogonal matrix $\bm U$ with the property:
\begin{align*}
    \|\bm U\|_{\infty}& \leq \sqrt{\frac{K_1 \log^{K_2}(m)}{m}},
\end{align*}
for some universal constants $K_1,K_2 \geq 0$,
then, we have:
\begin{align*}
    \frac{1}{m} \E[\Tr(\altprod(\bm \Psi, \bm Z)) | \bm A] &\explain{P}{\rightarrow} 0.
\end{align*}
\end{lem}
\begin{lem}[Analysis of Variance] \label{lemma: trace_variance} Let $\altprod(\bm \Psi, \bm Z)$ be any alternating product of the matrices $\bm\Psi,\bm Z$. Then,
\begin{align*}
    \var\left(\frac{\Tr\altprod(\bm \Psi, \bm Z)}{m}\bigg|\bm A\right) & \leq \frac{C(\altprod)}{n},
\end{align*}
where $C(\altprod)$ denotes a constant depending only on the formula for the alternating product $\altprod$ (independent of $m,n$).
\end{lem}

% Since Lemmas \ref{lemma: trace_good_event} and \ref{lemma: continuity_trace} are more intuitive, we moved their proofs to Appendix \ref{appendix: free_probability_trace}.  

Proofs of Lemmas \ref{lemma: trace_expectation} and \ref{lemma: trace_variance} can be found at Section \ref{sec: proof of trace lemmas}.  Before moving forward to the proofs of these lemmas, let us conclude the proof of Proposition \ref{proposition: free_probability_trace} assuming Lemmas \ref{lemma: trace_expectation} and \ref{lemma: trace_variance} are true.

\begin{proof}[Proof of Proposition \ref{proposition: free_probability_trace}] We write $\Tr(\altprod(\bm \Psi, \bm Z))/m$ as:
\begin{align*}
    \frac{\Tr(\altprod(\bm \Psi, \bm Z))}{m} & = \E \left[  \frac{\Tr(\altprod(\bm \Psi, \bm Z))}{m} \bigg| \bm A\right] + \left(\frac{\Tr(\altprod(\bm \Psi, \bm Z))}{m}  -\E \left[  \frac{\Tr(\altprod(\bm \Psi, \bm Z))}{m} \bigg| \bm A\right] \right). 
\end{align*}
We will show each of the two terms on the right hand side converge to zero in probability. Lemma \ref{lemma: trace_expectation} already gives:
\begin{align*}
     \E \left[  \frac{\Tr(\altprod(\bm \Psi, \bm Z))}{m} \bigg| \bm A\right] &\explain{P}{\rightarrow} 0.
\end{align*}
On the other hand, by Chebychev's Inequality and Lemma \ref{lemma: trace_variance} we have:
\begin{align*}
    \P\left[ \left| \frac{\Tr(\altprod(\bm \Psi, \bm Z))- \E[ \Tr(\altprod(\bm \Psi, \bm Z)) | \bm A]}{m} \right| > \epsilon  \bigg| \bm A\right] & \leq \frac{1}{\epsilon^2} \cdot \var\left(\frac{\Tr\altprod(\bm \Psi, \bm Z)}{m}\bigg|\bm A\right) \leq \frac{C(\altprod)}{n\epsilon^2}.
\end{align*}
Hence,
\begin{align*}
 \P\left[ \left| \frac{\Tr(\altprod(\bm \Psi, \bm Z))- \E[ \Tr(\altprod(\bm \Psi, \bm Z)) | \bm A]}{m} \right| > \epsilon  \right] \rightarrow 0.
\end{align*}
This concludes the proof of the proposition. 
\end{proof}

\subsection{Proof of Lemmas \ref{lemma: trace_expectation} and \ref{lemma: trace_variance} \label{sec: proof of trace lemmas}}

\begin{proof}[Proof of Lemma \ref{lemma: trace_expectation}]
Recall the notation regarding partitions introduced in Section \ref{section: partition_notation}.
We will organize the proof into various steps.
\begin{description}
\item [Step 1: Restricting to a Good Event.] \rishabh{We first observe that $\Tr(\altprod(\bm \Psi, \bm Z))/m$ is uniformly bounded. For example, when $\altprod(\bm \Psi, \bm Z)$ is a Type-2 alternating product:
\begin{align}
   \altprod(\bm \Psi, {\bm Z)} = (\bm \Psi) q_1({\bm Z}) (\bm \Psi) q_2({\bm Z}) \cdots  (\bm \Psi) q_k({\bm Z}),
\end{align}
we have,
\begin{align*}
    \frac{\Tr\altprod(\bm \Psi, \bm Z)}{m} & \leq \|\altprod(\bm \Psi, \bm Z)\|_{\op} \leq \|\bm \Psi \|_{\op}^k \prod_{i=1}^k \|q(\bm Z)\|_\op \leq \prod_{i=1}^k \|q_i\|_\infty \explain{def}{=} C(\altprod) < \infty,
\end{align*}
where we defined $ \|q_i\|_\infty = \sup_{\xi \in \R} |q_i(\xi)|$ and used the fact that $\|\bm \Psi\|_{\op} = \|\bm U \barB \bm U^\UT\|_{\op} = \max(\kappa, 1-\kappa) \leq 1$. In particular, note that $C(\altprod)$ is a finite constant independent of $m,n$. Analogous bounds hold for alternating forms of other types.} Recall the definition of $\mathcal{E}$ in \eqref{eq: good_event_trace_lemma}. If the sensing matrix $\bm A$ was generated by subsampling a deterministic orthogonal matrix $\bm U$ with the property
\begin{align*}
    \|\bm U\|_\infty & \leq \sqrt{\frac{K_1\log^{K_2}(m)}{m}},
\end{align*}
then Lemma \ref{lemma: trace_good_event} gives $\P(\mathcal{E}^c) \leq 4/m^2$. On the other hand, if $\bm A$ was generated by subsampling a uniformly random column orthogonal matrix $\bm O$ then we set $K_1 = 8, K_2 = 1$ and Lemma \ref{lemma: trace_good_event} gives $\P(\mathcal{E}^c) \leq 6/m^2$. Using this event, we decompose $\E[\Tr(\altprod(\bm \Psi, \bm Z) | \bm A]/m$ as:
\begin{align*}
    \frac{\E[\Tr\altprod(\bm \Psi, \bm Z)|\bm A]}{m} & = \frac{\E[\Tr\altprod(\bm \Psi, \bm Z)|\bm A]}{m} \cdot  \Indicator{\mathcal{E}} + \frac{\E[\Tr\altprod(\bm \Psi, \bm Z)|\bm A]}{m} \cdot  \Indicator{\mathcal{E}^c}.
\end{align*}
Since $\P(\mathcal{E}^c) \rightarrow 0$ and $\E[\Tr(\altprod(\bm \Psi, \bm Z) | \bm A]/m< C(\mathcal{A})<\infty$ is uniformly bounded, we immediately obtain $\E[\Tr(\altprod(\bm \Psi, \bm Z) | \bm A]\cdot\Indicator{\mathcal{E}^c}/m\explain{P}{\rightarrow} 0$. Hence, we simply need to show:
\begin{align*}
    \frac{\E[\Tr\altprod(\bm \Psi, \bm Z)|\bm A]}{m} \cdot  \Indicator{\mathcal{E}} &\explain{P}{\rightarrow} 0.
\end{align*}
\item [Step 2: Variance Normalization.]
Recall that $\bm Z = \diag{\bm z}, \; \bm z = \bm A \bm x \sim \gauss{\bm 0}{\bm A \bm A^\UT/\kappa}$. We define the normalized random vector $\tilde{\bm z}$ as:
\begin{align}
    \tilde{z}_i & = \frac{z_i}{\sigma_i}, \; \sigma_i^2 = \frac{(\bm A \bm A^\UT)_{ii}}{\kappa}. \label{eq: tilde_z_distribution}
\end{align}
Note that conditional on $\bm A$, $\tilde{\bm z}$ is a zero mean Gaussian vector with: $$\E[\tilde{z_i}^2 | \bm A] =1, \; \E[ \tilde{z}_i \tilde{z_j} | \bm A] = \frac{(\bm A \bm A^\UT)_{ij}/\kappa}{\sigma_i \sigma_j}.$$ We define the diagonal matrix $\tilde{\bm Z} = \diag{\tilde{\bm z}}$. Using the continuity estimate from Lemma \ref{lemma: continuity_trace} we have,
\begin{align*}
    \left|\frac{\Tr \altprod(\bm \Psi, \bm Z)}{m} - \frac{\Tr \altprod(\bm \Psi, \tilde{\bm Z})}{m} \right|  &\leq \frac{C(\altprod)}{\sqrt{m}} \|\bm z - \tilde{\bm z}\|_2 \\ &\leq C(\altprod) \cdot \left(\frac{1}{m} \sum_{i=1}^m z_i^2\right)^{\frac{1}{2}} \cdot \left( \max_{i \in [m]} \left| \frac{1}{\sigma_i} - 1 \right| \right) \\
    & \leq C(\altprod) \cdot \left(\frac{1}{m} \sum_{i=1}^n x_i^2\right)^{\frac{1}{2}} \cdot \left( \max_{i \in [m]} \left| \frac{1}{\sigma_i} - 1 \right| \right).
\end{align*}
We observe that $\|\bm x\|^2/m \explain{P}{\rightarrow} \kappa^{-1}$, and on the event $\mathcal{E}$,
\begin{align*}
    \max_{i \in [m]} \left| \frac{1}{\sigma_i} - 1 \right| \rightarrow 0.
\end{align*}
Hence, 
\begin{align*}
     \left|\frac{\E[\Tr \altprod(\bm \Psi, \bm Z)| \bm A]}{m} - \frac{\E[\Tr \altprod(\bm \Psi, \tilde{\bm Z})| \bm A]}{m} \right| \cdot \Indicator{\mathcal{E}} \explain{P}{\rightarrow} 0,
\end{align*}
and hence, to conclude the proof of the lemma we simply need to show:
\begin{align*}
    \frac{\E[\Tr\altprod(\bm \Psi, \tilde{\bm Z})|\bm A]}{m} \cdot  \Indicator{\mathcal{E}} &\explain{P}{\rightarrow} 0.
\end{align*}
\item [Step 3: Mehler's Formula.]
Supposing that the alternating product is of the Type 2 form (recall Definition \ref{def: alternating_product}):
\begin{align*}
   \altprod(\bm \Psi, \tilde{\bm Z)} = (\bm \Psi) q_1(\tilde{\bm Z}) (\bm \Psi) q_2(\tilde{\bm Z}) \cdots  (\bm \Psi) q_k(\tilde{\bm Z}). 
\end{align*}
The argument for the other types is very similar and we will sketch it in the end. 
We expand $ \Tr\altprod(\bm \Psi, \tilde{\bm Z})$ as follows:
\begin{align*}
    \frac{1}{m}  \Tr\altprod(\bm \Psi, \tilde{\bm Z}) & = \frac{1}{m} \sum_{a_1,a_2, \dots a_k = 1}^m (\bm \Psi)_{a_1,a_2} q_1(\tilde{\bm Z})_{a_2,a_2} \cdots (\bm \Psi)_{a_k,a_1} q_k(\tilde{\bm Z})_{a_1,a_1}.
\end{align*}
Next, we observe that:
\begin{align*}
    [m]^k & = \bigsqcup_{\pi \in \part{}{[k]}} \cset{\pi}.
\end{align*}
Hence we can decompose the above sum as:
\begin{align*}
     \frac{\E[ \Tr\altprod(\bm \Psi, \tilde{\bm Z}) \; | \bm A]}{m}  & = \sum_{\pi \in \part{}{[k]}} \frac{1}{m} \sum_{a \in \cset{\pi}} (\bm \Psi)_{a_1,a_2}  \cdots (\bm \Psi)_{a_k,a_1} \E[ \; q_1(\tilde{z}_{a_2})  \cdots q_k(\tilde{z}_{a_{k+1}}) | \bm A].
\end{align*}
By the triangle inequality,
\begin{align}
    \left| \frac{\E[ \Tr\altprod(\bm \Psi, \tilde{\bm Z}) \; | \bm A]}{m} \right| & \leq \sum_{\pi \in \part{}{[k]}} \frac{1}{m} \sum_{a \in \cset{\pi}} |(\bm \Psi)_{a_1,a_2}  \cdots (\bm \Psi)_{a_k,a_1}| |\E[ \; q_1(\tilde{z}_{a_2})  \cdots q_k(\tilde{z}_{a_1}) | \bm A]| \label{eq: trace_bound}.
\end{align}
We first bound $|\E[ \; q_1(\tilde{z}_{a_2}) q_2(\tilde{z}_{a_3}) \cdots q_k(\tilde{z}_{a_1}) | \bm A]|$. Observe that if we denote the blocks of $\pi = \{ \blocks_1, \blocks_2 \dots \blocks_{|\pi|} \}$, we can write:
\begin{align*}
\left| \E[ \; q_1(\tilde{z}_{a_2}) q_2(\tilde{z}_{a_3}) \cdots q_k(\tilde{z}_{a_1}) | \bm A]\right| & =  \left| \E\left[ \prod_{i=1}^{|\pi|} \prod_{j \in \blocks_i} q_{j-1}(\tilde{z}_{a_{\blocks_i}}) \bigg| \bm A\right]\right|.    
\end{align*}
In the above display, we have defined $q_0 \explain{def}{=} q_k$.
Define the functions $\bar{q}_1,\bar{q}_2 \dots \bar{q}_{|\pi|}$ as:
\begin{align*}
    \bar{q}_i(\xi) & = \prod_{j \in \blocks_i} q_{j-1}(\xi) - \nu_i, \; \nu_i = \E_{\xi \sim  \gauss{0}{1}} \left[ \prod_{j \in \blocks_i} q_{j-1}(\xi)\right].
\end{align*}
\rishabh{Hence, we obtain:
\begin{align}
    \left| \E[ \; q_1(\tilde{z}_{a_2}) q_2(\tilde{z}_{a_3}) \cdots q_k(\tilde{z}_{a_1}) | \bm A]\right|
    & =
    \left| \E\left[ \prod_{i=1}^{|\pi|}(\bar{q}_i(z_{a_{\blocks_i}})+\nu_i) \bigg| \bm A\right]\right| 
    \\& \explain{(a)}{=}\left| \E \left[  \sum_{V \subset [|\pi|]} \left( \prod_{i \not\in V} \nu_i \right) \cdot \left( \prod_{i \in V} \bar{q}_i(\tilde{z}_{a_{\blocks_i}})  \right)  \bigg| \bm A  \right]\right| \\
    &\leq \sum_{V \subset [|\pi|]} \left( \prod_{i \not\in V} |\nu_i| \right) \cdot \left| \E \left[ \prod_{i \in V} \bar{q}_i(\tilde{z}_{a_{\blocks_i}}) \bigg| \bm A \right] \right|. \label{eq: conditional_expectation}
\end{align}
In the above display, we expanded the product in the step marked (a) and used the triangle inequality in step (b).
}Let $\singleblks{\pi}$ denote the singleton blocks of the partition $\pi$: $\singleblks{\pi} =  \{i \in [|\pi|] :  | \blocks_i| = 1\}$. Note that for any $i \in \singleblks{\pi}$, $\nu_i = 0$ since the functions $q_i$ satisfy $\E q_i(\xi) = 0$ when $\xi \sim \gauss{0}{1}$ (Definition \ref{def: alternating_product}).  Hence,
\begin{align*}
    \left| \E[ \; q_1(\tilde{z}_{a_2}) q_2(\tilde{z}_{a_3}) \cdots q_k(\tilde{z}_{a_1}) | \bm A]\right| & \leq \sum_{V \subset [|\pi|]: \singleblks{\pi} \subset V} \left( \prod_{i \not\in V} |\nu_i| \right) \cdot \left| \E \left[ \prod_{i \in V} \bar{q}_i(\tilde{z}_{a_{\blocks_i}}) \bigg| \bm A \right] \right|.
\end{align*}
Next, we apply Mehler's Formula (Proposition \ref{proposition: mehler}) to bound:
\begin{align*}
    \left| \E \left[ \prod_{i \in V} \bar{q}_i(\tilde{z}_{a_{\blocks_i}}) \bigg| \bm A \right] \right| \Indicator{\mathcal{E}}.
\end{align*}
We make the following observations:
\begin{enumerate}
    \item Recall the distribution of $\tilde{\bm z}$ given in \eqref{eq: tilde_z_distribution} and the definition of the event $\mathcal{E}$ in \eqref{eq: good_event_trace_lemma}, we obtain:
\begin{align*}
    \max_{i \neq j} |\E [\tilde{z}_i \tilde{z}_j| \bm A]| & \leq \left( \max_{i \neq j} \frac{1}{\kappa \sigma_i \sigma_j} \sqrt{\frac{32 \cdot K_1^2 \cdot  \log^{2K_2 + 1}(m)}{m}}\right).
\end{align*}
Note that for large enough $m$, event $\mathcal{E}$ guarantees $\min_{i} \sigma_i \geq 1/2$. Hence,
\begin{align*}
      \max_{i \neq j} |\E [\tilde{z}_i \tilde{z}_j| \bm A]| & \leq  \left(  \frac{4}{\kappa} \sqrt{\frac{32 \cdot K_1^2 \cdot  \log^{2K_2 + 1}(m)}{m}}\right).
\end{align*}
For any $S \subset [m]$ with $|S| \leq k $, let $\E[\tilde{\bm z} \tilde{\bm z}^\UT | \bm A]_{S,S}$ be the principal submatrix of the covariance matrix $\E[\tilde{\bm z} \tilde{\bm z}^\UT | \bm A]$. By Gershgorin's Circle Theorem we have.
\begin{align*}
    \lambda_{\min} \left( \E[\tilde{\bm z} \tilde{\bm z}^\UT | \bm A]_{S,S}\right) & \geq 1 - k \max_{i \neq j} |\E [\tilde{z}_i \tilde{z}_j| \bm A]| \geq \frac{1}{2} \; \text{ (for $m$ large enough)}.
\end{align*}
\item We note that $\bar{q}_i$ satisfy $\E \bar{q}_i(\xi) = 0$ and $ \E \xi \bar{q}_i(\xi) = 0$ (since $\bar{q}_i$ are even functions) when $\xi \sim \gauss{0}{1}$. Hence, the first non-zero term in Mehler's expansion  corresponds to $\bm w$ such that:
\begin{align*}
    \degree_i(\bm w) \geq 2, \quad \forall \; i \; \in \; V,
\end{align*}
thus,
\begin{equation*}
    \|\bm w\| \geq |V|.
\end{equation*}
\end{enumerate}
 Hence, by Mehler's Formula (Proposition \ref{proposition: mehler}), we obtain:
\begin{align*}
    \left| \E \left[ \prod_{i \in V} \bar{q}_i(\tilde{z}_{a_{\blocks_i}}) \bigg| \bm A \right] \right| \Indicator{\mathcal{E}} & \leq C \cdot \left( \max_{i \neq j} \E [\tilde{z}_i \tilde{z}_j| \bm A]\right)^{|V|} \\&\leq C \cdot \left(  \frac{4}{\kappa} \sqrt{\frac{32 \cdot K_1^2 \cdot  \log^{2K_2 + 1}(m)}{m}}\right)^{|V|},
\end{align*}
for some finite constant $C$ depending only on $k$ and the functions $q_{1:k}$. 
Substituting this bound in \eqref{eq: conditional_expectation} we obtain:
\begin{align*}
     \left| \E[ \; q_1(\tilde{z}_{a_2}) q_2(\tilde{z}_{a_3}) \cdots q_k(\tilde{z}_{a_1}) | \bm A]\right| \cdot \Indicator{\mathcal{E}}  &\leq \sum_{V \subset [|\pi|]} \left( \prod_{i \not\in V} |\nu_i| \right) \cdot \left| \E \left[ \prod_{i \in V} \bar{q}_i(\tilde{z}_{a_{\blocks_i}}) \bigg| \bm A \right] \right| \\
     & \leq C \sum_{V \subset [|\pi|]} \left( \prod_{i \not\in V} |\nu_i| \right) \cdot \left(  \frac{4}{\kappa} \sqrt{\frac{32 \cdot K_1^2 \cdot  \log^{2K_2 + 1}(m)}{m}}\right)^{|V|}\\
     & \leq C(\altprod) \cdot \left(  \frac{4}{\kappa} \sqrt{\frac{32 \cdot K_1^2 \cdot  \log^{2K_2 + 1}(m)}{m}}\right)^{|\singleblks{\pi}|}.
\end{align*}
In the above display, $C(\altprod)$ denotes a finite constant depending only on $k$ and the functions appearing in the definition of $\altprod$. Substituting this in \eqref{eq: trace_bound}:
\begin{align*}
    &\left| \frac{\E[ \Tr\altprod(\bm \Psi, \tilde{\bm Z}) \; | \bm A]}{m} \right|\Indicator{\mathcal{E}}  \\&\hspace{0cm}\leq  \sum_{\pi \in \part{}{[k]}} \frac{C(\altprod)}{m} \sum_{a \in \cset{\pi}} |(\bm \Psi)_{a_1,a_2} \cdots (\bm \Psi)_{a_k,a_1}| \left(  \frac{4}{\kappa} \sqrt{\frac{32 \cdot K_1^2 \cdot  \log^{2K_2 + 1}(m)}{m}}\right)^{|\singleblks{\pi}|}.
\end{align*}
Again, recalling the definition of $\mathcal{E}$ in \eqref{eq: good_event_trace_lemma}, we can upper bound $|(\bm \Psi)_{a_1,a_2} \cdots (\bm \Psi)_{a_k,a_1}|$:
\begin{align}
    \left| \frac{\E[ \Tr\altprod(\bm \Psi, \tilde{\bm Z}) \; | \bm A]}{m} \right| \cdot \Indicator{\mathcal{E}} & \leq  \sum_{\pi \in \part{}{[k]}} \frac{C(\altprod)}{m} \sum_{a \in \cset{\pi}}  \cdot\left(   \sqrt{\frac{\cdot K_1^2 \cdot  \log^{2K_2 + 1}(m)}{m}}\right)^{|\singleblks{\pi}|+k} \nonumber \\
    & = \frac{C(\altprod)}{m} \sum_{\pi \in \part{}{[k]}} |\cset{\pi}| \cdot \left(   \sqrt{\frac{\cdot K_1^2 \cdot  \log^{2K_2 + 1}(m)}{m}}\right)^{|\singleblks{\pi}|+k} \label{eq: trace_bound_penultimate}. 
\end{align}
\item [Step 4: Conclusion.]
Observe that: $|\cset{\pi}| \leq m^{|\pi|}$. Recall that $\pi$ has $|\singleblks{\pi}|$ singleton blocks. All remaining blocks of $\pi$ have at least 2 elements. Hence, we can upper bound $|\pi|$ as follows:
\begin{align*}
    |\pi| & \leq \frac{k-|\singleblks{\pi}|}{2} + |\singleblks{\pi}| = \frac{k+|\singleblks{\pi}|}{2}.
\end{align*}
Substituting this in \eqref{eq: trace_bound_penultimate} along with the trivial bounds $|\singleblks{\pi}| \leq k, \; |\part{}{[k]} \leq k^k$, we obtain:
\begin{align*}
    \left| \frac{\E[ \Tr\altprod(\bm \Psi, \tilde{\bm Z}) \; | \bm A]}{m} \right| \cdot \Indicator{\mathcal{E}} & \leq \frac{C(\altprod) \cdot k^k \cdot (K_1^2 \log^{2K_2+1}(m))^k}{m} \rightarrow 0,
\end{align*}
as desired. 
\item[Step 5: Other Cases.] Recall that we had assumed that the alternating product was of Type 2:
\begin{align*}
    \altprod(\bm \Psi, \tilde{\bm Z)} = (\bm \Psi) q_1(\tilde{\bm Z}) (\bm \Psi) q_2(\tilde{\bm Z}) \cdots  (\bm \Psi) q_k(\tilde{\bm Z}).
\end{align*}
The analysis for the other types is analogous, and we briefly sketch these cases:
\begin{description}
    \item[Type 1: $\altprod(\bm \Psi, \tilde{\bm Z)} = (\bm \Psi) q_1(\tilde{\bm Z}) (\bm \Psi) q_2(\tilde{\bm Z}) \cdots  (\bm \Psi) q_k(\tilde{\bm Z}) (\bm \Psi)$.] In this case, the normalized trace is expanded as:
\begin{align*}
     &\frac{\E[ \Tr\altprod(\bm \Psi, \tilde{\bm Z}) \; | \bm A]}{m}   = \frac{1}{m}\sum_{a_0,a_1, \dots a_k = 1}^m \E[ (\bm \Psi)_{a_0,a_1} q_1(\tilde{\bm Z})_{a_1, a_1} \cdots q_{k}(\tilde{\bm Z})_{a_{k},a_k} (\bm \Psi)_{a_{k},a_0} | \bm A] \\ 
     & = \frac{1}{m} \sum_{a_0 = 1}^m \sum_{\pi \in \part{}{[k]}} \sum_{a \in \cset{\pi}} (\bm \Psi)_{a_0,a_1} (\bm \Psi)_{a_1,a_2} \cdots (\bm \Psi)_{a_k,a_0} \E[ q_1(\tilde{z}_{a_1}) \cdots q_k(\tilde{z}_{a_k}) | \bm A].
\end{align*}
As before, we can argue on the event $\mathcal{E}$, for any $a_{0:k}$:
\begin{align*}
    |\E[ q_1(\tilde{z}_{a_1}) \cdots q_k(\tilde{z}_{a_k}) | \bm A]| & \leq O \left( \left( \frac{\polylog(m)}{m} \right)^{\frac{|\singleblks{\pi}|}{2}} \right), \\
    |(\bm \Psi)_{a_0,a_1} (\bm \Psi)_{a_1,a_2} \cdots (\bm \Psi)_{a_k,a_0}| & \leq  O \left( \left( \frac{\polylog(m)}{m} \right)^{\frac{k+1}{2}} \right), \\
    |\cset{\pi}| & \leq m^{\frac{k+|\singleblks{\pi}|}{2}}, \\
    |\part{}{[k]}| & \leq k^k.
\end{align*}
This gives us:
\begin{align*}
    \left| \frac{\E[ \Tr\altprod(\bm \Psi, \tilde{\bm Z}) \; | \bm A]}{m} \right| \Indicator{\mathcal{E}} &\leq \frac{1}{m} \cdot {\overbrace{m}^{\text{choices for $a_0$}} \cdot \overbrace{|\part{}{[k]}|}^{\text{choices for $\pi$}} \cdot \overbrace{|\cset{k}|}^{\text{choices for $a_{1:k}$}}} \cdot O  \left( \frac{\polylog(m)}{m^{\frac{k+|\singleblks{\pi}|+1}{2}}} \right) \\
    & = O \left( \frac{\polylog(m)}{\sqrt{m}} \right) \rightarrow 0.
\end{align*}
\item [Type 3: $\altprod = q_0(\bm Z) (\bm \Psi) q_1(\bm Z) \cdots (\bm \Psi) q_k(\bm Z)$.] This case can be reduced to Type 1 and Type 2. Define $\tilde{q}_k(\xi) = q_0(\xi) q_k(\xi) - \nu, \; \nu = \E_{\xi \sim \gauss{0}{1}} \; q_0(\xi) q_k(\xi) $. Then:
\begin{align*}
    &\frac{\E[\Tr\altprod(\bm \Psi, \bm Z)| \bm A]}{m}  = \frac{\E[\Tr (q_0(\bm Z)(\bm \Psi) q_1(\bm Z) \cdots(\bm \Psi) q_k(\bm Z)) | \bm A]}{m} \\
    & = \frac{\E[\Tr( (\bm \Psi) q_1(\bm Z) \cdots (\bm \Psi) q_k(\bm Z) q_0(\bm Z)) | \bm A]}{m} \\
    & = \underbrace{\frac{\E[\Tr( (\bm \Psi) q_1(\bm Z) \cdots (\bm \Psi) \tilde{q}_k(\bm Z))  | \bm A]}{m}}_{\text{Type 2}} + \nu \underbrace{\frac{\E[\Tr( (\bm \Psi) q_1(\bm Z) \cdots (\bm \Psi))   | \bm A]}{m}}_{\text{Type 1}}.
\end{align*}
\item [Type 4: $\altprod(\bm \Psi, \bm Z) = q_1(\bm Z) (\bm \Psi) q_2(\bm Z) (\bm \Psi) \cdots q_k(\bm Z) (\bm \Psi)$. ] This case is exactly the same as Type 2, and exactly the same bounds hold.
\end{description}
\end{description}
This concludes the proof of Lemma \ref{lemma: trace_expectation}.
\end{proof}

\begin{proof}[Proof of Lemma \ref{lemma: trace_variance}] We observe that since $\bm \Psi = \bm A \bm A^\UT - \kappa \bm I_m$, conditioning on $\bm A$ fixes $\bm \Psi$. Hence, the only source of randomness in $\altprod(\bm \Psi, \bm Z)$ is  $\bm Z = \diag{\bm z}, \bm z = \bm A \bm x, \bm x \sim \gauss{0}{1/\kappa}$. Define the map $f(\bm x) \explain{def}{=} \Tr(\altprod(\bm \Psi, \diag{\bm A \bm x})/m$. By Lemma \ref{lemma: continuity_trace}, we have:
\begin{align*}
    |f(\bm x) - f(\bm x^\prime)| & \leq \frac{C(\altprod)}{\sqrt{m}} \cdot \|\bm A(\bm x - \bm x^\prime)\|_2 \leq \frac{C(\altprod) \|\bm A\|_{\op}}{\sqrt{m}}\cdot  \|\bm x - \bm x^\prime\|_2 =\frac{C(\altprod)}{\sqrt{m}}\cdot \|\bm x - \bm x^\prime\|_2.  
\end{align*}
Hence, $f$ is $C(\altprod)/\sqrt{n}$-Lipchitz. The claim of lemma follows from the Gaussian Poincare Inequality (see Fact \ref{fact: gaussian_poincare}).
\end{proof}

%\color{black}
%%%%%%%%%%%%%%%%%%%%%%%%%%%%%%%%%%%%%%%%%%%%%%%%%%%%%%%%%%%%%
%%%%%%%%%%%%%%%%%%%%%%%%%%%%%%%%%%%%%%%%%%%%%%%%%%%%%%%%%%%%%
%%%%%%%%%%%%%%%%%%%%%%%%%%%%%%%%%%%%%%%%%%%%%%%%%%%%%%%%%%%%%

%%%%%%%%%%%%%%%%%%%%%%%%%%%%%%%%%%%%%%%%%%%%%%%%%%%%%%%%%%%%%
%%%%%%%%%%%%%%%%%%%%%% PROOF-QF %%%%%%%%%%%%%%%%%%%%%%%%%%%%%
%%%%%%%%%%%%%%%%%%%%%%%%%%%%%%%%%%%%%%%%%%%%%%%%%%%%%%%%%%%%%
\section{Proof of Proposition \ref{proposition: free_probability_qf}}
\label{section: qf_proof}
In this section, we provide a proof of Proposition \ref{proposition: free_probability_qf}.  The proof follows from the following three results.

% Consider any fixed alternating product of the matrices $\bm \Psi, \bm Z$: $\altprod(\bm \Psi, \bm Z)$. This is of one of the four types:
% \\ \milad{Maybe we can remove below enumeration since we have it in the previous section}
% \begin{enumerate}
%     \item Type 1:  $\altprod = p_1(\bm \Psi) q_1(\bm Z) p_2(\bm \Psi) \cdots  q_{k-1}(\bm Z) p_k(\bm \Psi)$
%     \item Type 2: $\altprod = p_1(\bm \Psi) q_1(\bm Z) p_2(\bm \Psi) q_2(\bm Z) \cdots  p_k(\bm \Psi) q_k(\bm Z)$
%     \item Type 3: $\altprod = q_0(\bm Z) p_1(\bm \Psi) q_2(\bm Z) \cdots p_{k-1}(\bm \Psi) q_k(\bm Z)$. 
%     \item Type 4: $\altprod = q_0(\bm Z) p_1(\bm \Psi) q_2(\bm Z) p_2(\bm \Psi) \cdots q_k(\bm Z) p_k(\bm \Psi)$.
% \end{enumerate}

\begin{lem}[Continuity Estimates] \label{lemma: continuity_qf} For any $\bm z, \tilde{\bm z} \in \R^m$, we have,
% \begin{enumerate}
%     \item Continuity with respect to $\bm z$:
\begin{align*}
         &\left| \frac{\bm z^\UT \altprod(\bm U \barB \bm U^\UT, \diag{\bm z}) \bm z}{m}  - \frac{\widetilde{\bm z}^{\UT} \altprod(\bm U \barB \bm U^\UT, \diag{\widetilde{\bm z}}) \widetilde{\bm z}}{m}   \right|  \\ & \hspace{4cm} \leq \frac{C(\altprod)}{m} \cdot \left(  \|\bm z\|_2^2 \cdot \|\bm z - \widetilde{\bm z}\|_\infty +   \|\bm z - \widetilde{\bm z}\|_2 \cdot (\|\bm z\|_2 + \|\widetilde{\bm z}\|_2) \right),
    \end{align*}
    where $C(\altprod)$ depends only on $k$, the $\|\|_\infty$-norms, and Lipchitz constants of the functions appearing in $\altprod$.
     %Continuity with respect to $\bm U$:
    
%\begin{align*}
    % \frac{1}{m} \left| \bm z^\UT \altprod(\bm U \barB \bm U^\UT, \diag{\bm z}) \bm z -\bm z^\UT \altprod(\widetilde{\bm U} \barB \widetilde{\bm U}^\UT, \diag{\bm z}) \bm z  \right| & \leq \frac{C(\altprod) \cdot \|\bm z\|^2}{m} \cdot \|\bm U - \widetilde{\bm U}\|_\fr,
% %\end{align*}
% where $C(\altprod)$ denotes a finite constant depending only on the $\|\cdot\|_\infty$ norms and Lipchitz constants of the functions appearing in $\altprod$.
% \end{enumerate}
\end{lem}
We have relegated the proof of the above continuity estimate to Appendix \ref{proof: continuity_qf} in the supplementary materials.

\begin{prop} [Universality of the first moment of the quadratic form] \label{prop: qf_univ_mom1} For both the subsampled Haar sensing model and the subsampled Hadamard sensing model, we  have:
\begin{align*}
    \lim_{m \rightarrow \infty}\frac{\E \bm z^\UT \altprod \bm z}{m}& =  (1-\kappa)^k \cdot \left( \prod_{i} \hat{q}_i(2) \right) \cdot \left( \prod_{i} (p_i(1-\kappa) - p_i(-\kappa)) \right),
\end{align*}
where the index $i$ in the product ranges over all the $p_i,q_i$ functions appearing in $\altprod$. In the above display:
\begin{align} \label{eq: hat_q}
    \hat{q}_i(2) & = \E q_i(\xi) H_2(\xi), \; \xi \sim \gauss{0}{1},
\end{align}
where $H_2(\xi) = \xi^2 - 1$ is the degree 2 Hermite polynomial.
\end{prop}

\begin{prop} [Universality of the second moment of the quadratic form] \label{prop: qf_univ_mom2}  For both the subsampled Haar sensing model and the subsampled Hadamard sensing model we have:
\begin{align*}
    \lim_{\substack{m\rightarrow \infty}} \frac{\E (\bm z^\UT \altprod \bm z)^2}{m^2} & = (1-\kappa)^{2k} \cdot \left( \prod_{i} \hat{q}^2_i(2) \right) \cdot \left( \prod_{i} (p_i(1-\kappa) - p_i(-\kappa))^2 \right).
\end{align*}
In the above expression, $\hat{q}_i(2)$ are as defined in \eqref{eq: hat_q}.
\end{prop}

We now provide a proof of Proposition \ref{proposition: free_probability_qf} using the above results. 

\begin{proof}[Proof of Proposition \ref{proposition: free_probability_qf}]
Note that Propositions \ref{prop: qf_univ_mom1}, \ref{prop: qf_univ_mom2} together imply that,
\begin{align*}
     \var \left( \frac{ \bm z^\UT \altprod \bm z}{m}\right) & \rightarrow 0,
\end{align*}
for both the sensing models. Hence, by Chebychev's inequality and Proposition \ref{prop: qf_univ_mom1}, we have, for both the sensing models,
\begin{align*}
    \plim \frac{ \bm z^\UT \altprod \bm z}{m} & = (1-\kappa)^k \cdot \left( \prod_{i} \hat{q}_i(2) \right) \cdot \left( \prod_{i} (p_i(1-\kappa) - p_i(-\kappa)) \right).
\end{align*}
This proves the claim of Proposition \ref{proposition: free_probability_qf}.
\end{proof}
The remainder of the section is dedicated to the proof of Proposition \ref{prop: qf_univ_mom1}. The proof of Proposition \ref{prop: qf_univ_mom2} is very similar and can be found in Appendix \ref{supplement: qf_univ_mom2} in the supplementary materials.

\subsection{Proof of Proposition \ref{prop: qf_univ_mom1}}
We provide a proof of Proposition \ref{prop: qf_univ_mom1} assuming that alternating form is of Type 1. $$\altprod(\bm \Psi, \bm Z) = p_1(\bm \Psi) q_1(\bm Z) p_2(\bm \Psi) \cdots  q_{k-1}(\bm Z) p_k(\bm \Psi).$$ We will outline how to handle the other types at the end of the proof (see Remark \ref{remark: all_types_qf_mom1}). Furthermore, in light of Lemma \ref{lemma: poly_psi_simple} we can further assume that all polynomials $p_i(\psi) = \psi$. Hence, we assume that $\altprod$ is of the form: $$\altprod(\bm \Psi, \bm Z) = \bm \Psi q_1(\bm Z)\bm \Psi \cdots  q_{k-1}(\bm Z) \bm \Psi.$$

The proof of Proposition \ref{prop: qf_univ_mom1} consists of various steps which will be organized as separate lemmas. We begin by recalling that 
\begin{align*}
    \bm z & \sim \gauss{0}{\frac{\bm A \bm A^\UT}{\kappa}}.
\end{align*}
Define the event:
    \begin{align}
        \mathcal{E}  &= \left\{ \max_{i \neq j} | (\bm A \bm A^\UT|)_{ij} \leq \sqrt{\frac{2048 \cdot  \log^{3}(m)}{m}}, \; \max_{i \in [m]} | (\bm A \bm A^\UT)_{ii} - \kappa | \leq \sqrt{\frac{2048 \cdot  \log^{3}(m)}{m}} \right\} \label{eq: good_event_qf_firstmom}.
    \end{align}
    By Lemma \ref{lemma: trace_good_event}, we know that $\P(\mathcal{E}^c) \rightarrow 0$ for both the subsampled Haar sensing and the subsampled Hadamard model. 
    We define the normalized random vector $\widetilde{\bm z}$ as:
\begin{align*}
    \widetilde{z}_i & = \frac{z_i}{\sigma_i}, \; \sigma_i^2 = \frac{(\bm A \bm A^\UT)_{ii}}{\kappa}. %\label{eq: tilde_z_distribution}
\end{align*}
Note that conditional on $\bm A$, $\widetilde{\bm z}$ is a zero mean Gaussian vector with: $$\E[\widetilde{z_i}^2 | \bm A] =1, \; \E[ \widetilde{z}_i \widetilde{z_j} | \bm A] = \frac{(\bm A \bm A^\UT)_{ij}/\kappa}{\sigma_i \sigma_j}.$$ We define the diagonal matrix $\widetilde{\bm Z} = \diag{\widetilde{\bm z}}$.

\begin{lem}\label{lemma : qf_variance_normalization} We have, \begin{align*}
    \lim_{m \rightarrow \infty} \frac{\E \bm z^\UT \altprod(\bm \Psi, \bm Z) \bm z}{m}&=\lim_{m \rightarrow \infty} \frac{ \E\widetilde{\bm z}^{\UT} \altprod(\bm\Psi, \widetilde{\bm Z}) \widetilde{\bm z}}{m} \Indicator{\mathcal{E}},
\end{align*}
provided the latter limit exists.
\end{lem}
The proof of the lemma  uses the fact that $\P(\mathcal{E}^c) \rightarrow 0$, and that on the event $\mathcal{E}$ since $\sigma_i^2 \approx 1$, we have $\bm z \approx \widetilde{\bm z}$ and hence, the continuity estimates of Lemma \ref{lemma: continuity_qf} give the claim of this result. Complete details have been provided in  Appendix \ref{proof: qf_variance_normalization} in the supplementary materials. 

The advantage of Lemma \ref{lemma : qf_variance_normalization} is that $\widetilde{z}_i \sim \gauss{0}{1}$, and on the event $\mathcal{E}$ the coordinates of  $\widetilde{\bm z}$ have weak correlations. Consequently, Mehler's Formula (Proposition \ref{proposition: mehler}) can be used to analyze the leading order term in $\E[\widetilde{\bm z}^{\UT} \altprod(\bm\Psi, \widetilde{\bm Z}) \widetilde{\bm z} \;  \Indicator{\mathcal{E}}]$. Before we do so, we do one additional preprocessing step. 

\begin{lem}\label{lemma: qf_diagonal_removal} We have:
\begin{align*}
    \lim_{m \rightarrow \infty} \frac{ \E\widetilde{\bm z}^{\UT} \altprod(\bm\Psi, \widetilde{\bm Z}) \widetilde{\bm z}}{m} \Indicator{\mathcal{E}}&=\lim_{m \rightarrow \infty} \frac{\E\ip{\altprod(\bm\Psi, \widetilde{\bm Z})}{\widetilde{\bm z} \widetilde{\bm z}^\UT -  \widetilde{\bm Z}^2} \Indicator{\mathcal{E}}}{m},
\end{align*}
provided the latter limit exists. 
\end{lem}
\begin{proof}[Proof Sketch] Observe that we can write:
    \begin{align*}
        \widetilde{\bm z}^\UT \altprod \widetilde{\bm z} & = \ip{\altprod(\bm\Psi, \widetilde{\bm Z})}{\widetilde{\bm z} \widetilde{\bm z}^\UT} \\
        & \explain{(a)}{=} \ip{\altprod(\bm\Psi, \widetilde{\bm Z})}{\widetilde{\bm z} \widetilde{\bm z}^\UT -  \widetilde{\bm Z}^2} + \Tr(\altprod(\bm\Psi, \widetilde{\bm Z}) \cdot q(\widetilde{\bm Z})) + \Tr(\altprod(\bm\Psi, \widetilde{\bm Z})).
    \end{align*}
    In the step marked (a), we defined $q(\xi) = \xi^2 - 1$ which is an even function. 
    Note that we know  $|\Tr(\altprod)|/m \leq \|\altprod\|_{\op} \leq C(\altprod) < \infty$. Furthermore, by Proposition \ref{proposition: free_probability_trace}, we know $\Tr(\altprod)/m \explain{P}{\rightarrow} 0$, and hence by Dominated Convergence Theorem $\E\Tr(\altprod) \Indicator{\mathcal{E}} /m \rightarrow 0$. Additionally, note that $\Tr(\altprod q(\widetilde{\bm Z}))$ is also an alternating form except for minor issue that $q(\xi)$ is not uniformly bounded and Lipchitz. However, the combinatorial calculations in Proposition \ref{proposition: free_probability_trace} can be repeated to show that $\E \Tr(\altprod \cdot q(\widetilde{\bm Z}))/m \rightarrow 0$. Since we will see a more complicated version of these arguments in the remainder of the proof, we omit the details of this step. 
\end{proof}
Note that, so far, Lemmas \ref{lemma : qf_variance_normalization} and \ref{lemma: qf_diagonal_removal} show that:
\begin{align*}
       \lim_{m \rightarrow \infty} \frac{\E \bm z^\UT \altprod(\bm \Psi, \bm Z) \bm z}{m}&=\lim_{m \rightarrow \infty} \frac{\E\ip{\altprod(\bm\Psi, \widetilde{\bm Z})}{\widetilde{\bm z} \widetilde{\bm z}^\UT -  \widetilde{\bm Z}^2} \Indicator{\mathcal{E}}}{m},
\end{align*}
provided the latter limit exists. We now focus on analyzing the RHS. We expand 
    \begin{align*}
        \frac{\ip{\altprod(\bm\Psi, \widetilde{\bm Z})}{\widetilde{\bm z} \widetilde{\bm z}^\UT -  \widetilde{\bm Z}^2}}{m} & = \frac{1}{m} \sum_{\substack{a_{1:k+1} \in [m]\\ a_1 \neq a_{k+1}}} \widetilde{z}_{a_1} (\bm \Psi)_{a_1,a_2} q_1(\widetilde{z}_{a_2})  \cdots  q_{k-1}(\widetilde{z}_{a_k}) (\bm \Psi)_{a_k,a_{k+1}} \widetilde{z}_{a_{k+1}}.
    \end{align*}
    Recall the notation for partitions introduced in Section \ref{section: partition_notation}. Observe that:
    \begin{align*}
        \{(a_1 \dots a_{k+1}) \in [m]^{k+1} : \; a_1 \neq a_{k+1} \} & = \bigsqcup_{\substack{\pi \in \part{}{[k+1]}\\ \pi(1) \neq \pi(k+1)}}  \cset{\pi}.
    \end{align*}
    Hence,
    \begin{align*}
        &\frac{\E\ip{\altprod(\bm\Psi, \widetilde{\bm Z})}{\widetilde{\bm z} \widetilde{\bm z}^\UT -  \widetilde{\bm Z}^2} \cdot \Indicator{\mathcal{E}}}{m} = \\& \hspace{1cm} \frac{1}{m} \sum_{\substack{\pi \in \part{}{[1:k+1]}\\ \pi(1) \neq \pi(k+1)}}  \sum_{a \in \cset{\pi}} \E \; \widetilde{z}_{a_1} (\bm \Psi)_{a_1,a_2} q_1(\widetilde{z}_{a_2}) (\bm \Psi)_{a_2, a_3} \cdots  q_{k-1}(\widetilde{z}_{a_k}) (\bm \Psi)_{a_k,a_{k+1}} \widetilde{z}_{a_{k+1}} \cdot \Indicator{\mathcal{E}}.
    \end{align*}
    Fix a $\pi \in \part{}{[k+1]}$ such that $\pi(1) \neq \pi(k+1)$, and consider a labelling $\bm a \in \cset{\pi}$. By the tower property,
 \begin{align*}
        &\E \widetilde{z}_{a_1} (\bm \Psi)_{a_1,a_2} q_1(\widetilde{z}_{a_2}) (\bm \Psi)_{a_2, a_3} \cdots  q_{k-1}(\widetilde{z}_{a_k}) (\bm \Psi)_{a_k,a_{k+1}} \widetilde{z}_{a_{k+1}} \Indicator{\mathcal{E}}  = \\&\hspace{1cm}\E \left[ (\bm \Psi)_{a_1,a_2} (\bm \Psi)_{a_2, a_3} \cdots (\bm \Psi)_{a_k,a_{k+1}}\cdot \E[ \widetilde{z}_{a_1} q_1(\widetilde{z}_{a_2}) q_2(\widetilde{z}_{a_3}) \cdots q_{k-1}(\widetilde{z}_{a_k})  \widetilde{z}_{a_{k+1}} |  \bm A ] \Indicator{\mathcal{E}} \right].
    \end{align*}
    We will now use Mehler's formula (Proposition \ref{proposition: mehler}) to evaluate the conditional expectation upto leading order. Note that some of the random variables $\widetilde{z}_{a_{1:k+1}}$ are equal (as given by the partition $\pi$). Hence, we group them together and recenter the resulting functions. The blocks corresponding to $a_1, a_{k+1}$ need to be treated specially due to the presence of $\widetilde{z}_{a_1},\widetilde{z}_{a_{k+1}}$ in the above expectations. Hence, we introduce the following notations:
    \begin{align*}
        \firstblk{\pi} & = \pi(1), \; \lastblk{\pi}= \pi(k+1), \; \singleblks{\pi} = \{i \in [2:k]: |\pi(i)| = 1 \}.
    \end{align*}
    We label all the remaining blocks of $\pi$ as $\blocks_1,\blocks_2 \dots \blocks_{|\pi| - |\singleblks{\pi}| - 2}$. Hence, the partition $\pi$ is given by:
    \begin{align*}
        \pi = \firstblk{\pi} \sqcup \lastblk{\pi} \sqcup \left( \bigsqcup_{i \in \singleblks{\pi}} \{i\} \right) \sqcup \left( \bigsqcup_{t=1}^{|\pi| - |\singleblks{\pi}| - 2} \blocks_i\right).
    \end{align*}
    Note that:
    \begin{align*}
        \widetilde{z}_{a_1} \widetilde{z}_{a_{k+1}}\prod_{i=2}^k q_{i-1}(\widetilde{z}_{a_i}) & = Q_\firstfnc(\widetilde{z}_{a_1})   Q_\lastfnc(\widetilde{z}_{a_{k+1}})  \left( \prod_{i \in \singleblks{\pi}} q_{i-1}(\widetilde{z}_{a_i}) \right) \cdot \prod_{i=1}^{|\pi| - |\singleblks{\pi}| - 2} (Q_{\blocks_i}(z_{a_{\blocks_i}}) + \mu_{\blocks_i}),
    \end{align*}
    where:
    \begin{align}
        Q_\firstfnc(\xi) & = \xi \cdot \prod_{i \in \firstblk{\pi}, i \neq 1} q_{i-1}(\xi), \label{eq: q_block_first}\\
        Q_\lastfnc(\xi) & = \xi \cdot \prod_{i \in \lastblk{\pi}, i \neq k+1} q_{i-1}(\xi), \label{eq: q_block_last}\\
        \mu_{\blocks_i} & = \E_{\xi \sim \gauss{0}{1}} \left[ \prod_{j \in \blocks_i} q_{j-1}(\xi) \right], \label{eq: q_block_single}\\
        Q_{\blocks_i}(\xi) & = \prod_{j \in \blocks_i} q_{j-1}(\xi) - \mu_{\blocks_i}. \label{eq: q_block}
    \end{align}
     With this notation in place, we can apply Mehler's formula. The result is summarized in the following lemma.
     \begin{lem} \label{lemma: qf_mehler_conclusion} For any $\pi \in \part{}{[k+1]}$ such that $\pi(1) \neq \pi(k+1)$, and any labelling $\bm a \in \cset{\pi}$ we have:
\begin{subequations}
    \begin{align}
        &\Indicator{\mathcal{E}} \cdot \left| \E[ \widetilde{z}_{a_1} q_1(\widetilde{z}_{a_2}) q_2(\widetilde{z}_{a_3}) \cdots q_{k-1}(\widetilde{z}_{a_k})  \widetilde{z}_{a_{k+1}} |  \bm A ] - \sum_{\bm w \in \weightedGnum{\pi}{1}} {\coeff}(\bm w, \pi) \cdot \matmom{\bm \Psi}{\bm w}{\pi}{\bm a}  \right| \nonumber \\
             &\hspace{8cm }  \leq C(\altprod) \cdot \left( \frac{\log^3(m)}{m \kappa^2} \right)^{\frac{2 + |\singleblks{\pi}|}{2}},
    \end{align}
    where $\matmom{\bm \Psi}{\bm w}{\pi}{\bm a}$ is the matrix moment as defined in Definition \ref{def: matrix moment}. The coefficients ${\coeff}(\bm w, \pi)$ are given by:
    \begin{align}
        {\coeff}(\bm w, \pi) & = \frac{1}{{\kappa^{\|\bm w\|}} \bm w!}  \cdot \left(\hat{Q}_\firstfnc(1) \hat{Q}_\lastfnc(1) \prod_{i \in \singleblks{\pi}} \hat{q}_{i-1}(2)   \right)  \cdot \left( \prod_{i \in [|\pi| - |\singleblks{\pi}| - 2]} \mu_{\blocks_i} \right),
    \end{align}
    and, the set $ \weightedGnum{\pi}{1}$ is defined as:
    \begin{align}
            \weightedGnum{\pi}{1} & \explain{def}{=} \left\{\bm w \in \weightedG{k+1}: \degree_1(\bm w) = 1, \; \degree_{k+1}(\bm w) = 1, \; \degree_i(\bm w) = 2 \; \forall \; i \; \in \; \singleblks{\pi}, \right. \nonumber\\ &\hspace{6cm} \left.\degree_i(\bm w) = 0 \; \forall \; i \; \notin \; \{1,k+1\} \cup \singleblks{\pi} \right\}.
        \end{align}
    \label{eq: mehler_conclusion_imp}
    \end{subequations}
     \end{lem}
The proof of the lemma is obtained by instantiating Mehler's formula for this situation and identifying the leading order term. Additional details for this step are provided in  Appendix \ref{proof: qf_mehler_conclusion} in the supplementary materials.

With this, we return to our analysis of:
\begin{align*}
        &\frac{\E\ip{\altprod(\bm\Psi, \widetilde{\bm Z})}{\widetilde{\bm z} \widetilde{\bm z}^\UT -  \widetilde{\bm Z}^2} \cdot \Indicator{\mathcal{E}}}{m} = \\& \hspace{1cm} \frac{1}{m} \sum_{\substack{\pi \in \part{}{[1:k+1]}\\ \pi(1) \neq \pi(k+1)}}  \sum_{a \in \cset{\pi}} \E \; \widetilde{z}_{a_1} (\bm \Psi)_{a_1,a_2} q_1(\widetilde{z}_{a_2}) (\bm \Psi)_{a_2, a_3} \cdots  q_{k-1}(\widetilde{z}_{a_k}) (\bm \Psi)_{a_k,a_{k+1}} \widetilde{z}_{a_{k+1}} \cdot \Indicator{\mathcal{E}}.
    \end{align*}
We define the following  subsets of $\part{}{k+1}$ as:
\begin{subequations}
\label{eq: qf_firstmom_goodpartitions}
\begin{align}
    &\part{1}{[k+1]}  \explain{def}{=} \nonumber\\&\{\pi \in \part{}{k+1}: \; \pi(1) \neq \pi(k+1), \; |\pi(1)| = 1,  |\pi(k+1)| = 1,  |\pi(j)| \leq 2 \; \forall \; j \; \in \; [k+1]\}, \\
    &\part{2}{[k+1]}  \explain{def}{=} \{\pi \in \part{}{k+1}: \; \pi(1) \neq \pi(k+1)\} \backslash \part{1}{[k+1]},
\end{align}
\end{subequations}
and the error term which was controlled in Lemma \ref{lemma: qf_mehler_conclusion}:
\begin{align*}
        \epsilon(\bm \Psi, \bm a) &\explain{def}{=} \Indicator{\mathcal{E}} \cdot \left( \E[ \widetilde{z}_{a_1} q_1(\widetilde{z}_{a_2})  \cdots q_{k-1}(\widetilde{z}_{a_k})  \widetilde{z}_{a_{k+1}} |  \bm A ] - \sum_{\bm w \in \weightedGnum{\pi}{1}} {\coeff}(\bm w, \pi) \cdot \matmom{\bm \Psi}{\bm w}{\pi}{\bm a}  \right).\end{align*}

With these definitions we consider the decomposition:
\begin{align*}
    &\frac{\E\ip{\altprod(\bm\Psi, \widetilde{\bm Z})}{\widetilde{\bm z} \widetilde{\bm z}^\UT -  \widetilde{\bm Z}^2} \cdot \Indicator{\mathcal{E}}}{m} = \\& \frac{1}{m} \sum_{\pi \in \part{1}{[k+1]}} \sum_{a \in \cset{\pi}} \sum_{\bm w \in \weightedGnum{\pi}{1}} {\coeff}(\bm w, \pi)  \E\left[   (\bm \Psi)_{a_1,a_2}   \cdots   (\bm \Psi)_{a_k,a_{k+1}}\matmom{\bm \Psi}{\bm w}{\pi}{\bm a} \right] - \mathsf{I} + \mathsf{II} + \mathsf{III},
\end{align*}
where:
\begin{align*}
        \mathsf{I} &\explain{def}{=}\frac{1}{m} \sum_{\substack{\pi \in \part{}{[k+1]}\\ \pi(1) \neq \pi(k+1)}} \sum_{a \in \cset{\pi}} \sum_{\bm w \in \weightedGnum{\pi}{1}} {\coeff}(\bm w, \pi)  \E\left[   (\bm \Psi)_{a_1,a_2}   \cdots   (\bm \Psi)_{a_k,a_{k+1}}\matmom{\bm \Psi}{\bm w}{\pi}{\bm a} \Indicator{\mathcal{E}^c} \right], \\
    \mathsf{II} & \explain{def}{=} \frac{1}{m} \sum_{\substack{\pi \in \part{}{[k+1]}\\ \pi(1) \neq \pi(k+1)}} \sum_{a \in \cset{\pi}} \E\left[   (\bm \Psi)_{a_1,a_2}  \cdots   (\bm \Psi)_{a_k,a_{k+1}}\epsilon(\bm \Psi, \bm a) \Indicator{\mathcal{E}} \right], \\
    \mathsf{III} & \explain{def}{=}\frac{1}{m} \sum_{\pi \in \part{2}{[k+1]}} \sum_{a \in \cset{\pi}} \sum_{\bm w \in \weightedGnum{\pi}{1}} {\coeff}(\bm w, \pi)  \E\left[   (\bm \Psi)_{a_1,a_2}   \cdots   (\bm \Psi)_{a_k,a_{k+1}}\matmom{\bm \Psi}{\bm w}{\pi}{\bm a} \right].
\end{align*}
Define $\bm{\ell}_{k+1} \in \weightedG{k+1}$ to be the weight matrix of a simple line graph, i.e.
\begin{align*}
    (\bm \ell_{k+1})_{ij} & = \begin{cases} 1 : & |j-i|=1 \\ 0 : & \text{ otherwise} \end{cases}.
\end{align*}
This decomposition can be written compactly as: 
\begin{align*}
    \mathsf{I} &\explain{}{=}\frac{1}{m} \sum_{\substack{\pi \in \part{}{[1:k+1]}\\ \pi(1) \neq \pi(k+1)}}\sum_{a \in \cset{\pi}} \sum_{\bm w \in \weightedGnum{\pi}{1}} {\coeff}(\bm w, \pi) \cdot \E\left[   \matmom{\bm \Psi}{\bm w + \bm{\ell}_{k+1}}{\pi}{\bm a} \Indicator{\mathcal{E}^c} \right], \\
    \mathsf{II} & = \frac{1}{m} \sum_{\substack{\pi \in \part{}{[1:k+1]}\\ \pi(1) \neq \pi(k+1)}} \sum_{a \in \cset{\pi}} \E\left[   \matmom{\bm \Psi}{\bm{\ell}_{k+1}}{\pi}{\bm a}\epsilon(\bm \Psi, \bm a) \Indicator{\mathcal{E}} \right], \\
     \mathsf{III} & \explain{}{=} \frac{1}{m} \sum_{\pi \in \part{2}{[1:k+1]}} \sum_{a \in \cset{\pi}} \sum_{\bm w \in \weightedGnum{\pi}{1}} {\coeff}(\bm w, \pi) \cdot \E\left[   \matmom{\bm \Psi}{\bm w + \bm{\ell}_{k+1}}{\pi}{\bm a} \right].
\end{align*}
We will show that $\mathsf{I}, \mathsf{II}, \mathsf{III} \rightarrow 0$. Showing this involves the following components:
\begin{enumerate}
    \item Bounds on matrix moments $\E\left[   \matmom{\bm \Psi}{\bm w + \bm{\ell}_{k+1}}{\pi}{\bm a} \right]$, which have been developed in Lemma \ref{lemma: matrix_moment_ub}. 
    \item Controlling the size of the set $|\cset{\pi}|$ (since we sum over $\bm a \in \cset{\pi}$ in the above terms). Since, 
    \begin{align*}
        |\cset{\pi}| & = m(m-1) \cdots (m-|\pi| + 1) \asymp m^{|\pi|},
    \end{align*}
    we need to develop bounds on $|\pi|$. This is done in the following lemma. In contrast, the sums over $\pi \in \part{}{[k+1]}$ and $\bm w \in \weightedGnum{\pi}{1}$ are not a cause of concern since $|\part{}{[k+1]}|,|\weightedGnum{\pi}{1}|$ depend only on $k$ (which is held fixed), and not on $m$. 
\end{enumerate}
\begin{lem} \label{lemma: qf_cardinality_bounds} For any $\pi \in \part{1}{[k+1]}$, we have:
\begin{align*}
    |\pi| & = \frac{k+3+|\singleblks{\pi}| }{2} \implies |\cset{\pi}| \leq m^{\frac{k+3+|\singleblks{\pi}| }{2}}. 
\end{align*}
For any $\pi \in \part{2}{[k+1]}$, we have:
\begin{align*}
    |\pi| & \leq \frac{k+2+|\singleblks{\pi}| }{2} \implies |\cset{\pi}| \leq m^{\frac{k+2+|\singleblks{\pi}| }{2}}. 
\end{align*}
\end{lem}
\begin{proof}
Consider any $\pi \in \part{}{[k+1]}$ such that $\pi(1) \neq \pi(k+1)$.  Recall that the disjoint blocks of $|\pi|$ were given by:
\begin{align*}
        \pi = \firstblk{\pi} \sqcup \lastblk{\pi} \sqcup \left( \bigsqcup_{i \in \singleblks{\pi}} \{i\} \right) \sqcup \left( \bigsqcup_{t=1}^{|\pi| - |\singleblks{\pi}| - 2} \blocks_i\right).
    \end{align*}
Hence,
\begin{align*}
    k+1 & = |\firstblk{\pi}| + |\lastblk{\pi}| + |\singleblks{\pi}| + \sum_{t=1}^{|\pi| - |\singleblks{\pi}| - 2} | \blocks_i|.
\end{align*}
Note that:
\begin{subequations}
\label{eq: qf_firstmom_blocksizeUB_observe}
\begin{align}
    |\firstblk{\pi}| \geq 1 & \qquad \text{ (Since $1 \in \firstblk{\pi}$)}, \\
    |\lastblk{\pi}| \geq 1 & \qquad \text{ (Since $k+1 \in \lastblk{\pi}$)},\\
    |\blocks_i| \geq 2 & \qquad \text{ (Since $\blocks_i$ are not singletons)}.
\end{align}\end{subequations}
Hence,
\begin{align*}
    k+1 & \geq |\firstblk{\pi}| + |\lastblk{\pi}| + |\singleblks{\pi}| +  2 |\pi| - 2|\singleblks{\pi}| - 4,
\end{align*}
which implies:
\begin{align}
     |\pi| &\leq \frac{k+5+|\singleblks{\pi}| -|\firstblk{\pi}| - |\lastblk{\pi}| }{2}   \nonumber  \\
     & \leq \frac{k+3+|\singleblks{\pi}| }{2} \label{eq: qf_firstmom_blocksizeUB}, 
\end{align}
and hence,
\begin{align*}
    |\cset{\pi}| & \leq m^{|\pi|} \leq m^{\frac{k+3+|\singleblks{\pi}| }{2}}.
\end{align*}
Finally, observe that:
\begin{enumerate}
    \item For any $\pi \in \part{1}{[k+1]}$ each of the inequalities in \eqref{eq: qf_firstmom_blocksizeUB_observe} are exactly tight by the definition of $\part{1}{[k+1]}$ in \eqref{eq: qf_firstmom_goodpartitions}, and hence:
    \begin{align*}
        |\pi| & = \frac{k+3+|\singleblks{\pi}| }{2}.
    \end{align*}
    \item For any $\pi \in \part{2}{[k+1]}$, one of the inequalities in \eqref{eq: qf_firstmom_blocksizeUB_observe} must be strict (see \eqref{eq: qf_firstmom_goodpartitions}). Hence, when $\pi \in \part{2}{[k+1]}$, we have the improved bound:
\begin{align*}
    |\pi| & \leq \frac{k+2+|\singleblks{\pi}| }{2}.
\end{align*}
\end{enumerate}
This proves the claims of the lemma. 
\end{proof}
We will now show that $\mathsf{I}, \mathsf{II}, \mathsf{III} \rightarrow 0$.
\begin{lem} We have,
\begin{align*}
    \mathsf{I} \rightarrow 0, \; \mathsf{II} \rightarrow 0, \; \mathsf{III} \rightarrow 0 \;\text{ as $m \rightarrow \infty$},
\end{align*}
and hence:
\begin{align*}
    \lim_{m \rightarrow \infty} \frac{\E \bm z^\UT \altprod \bm z}{m}  & =  \lim_{m \rightarrow \infty} \frac{1}{m} \sum_{\pi \in \part{1}{[k+1]}} \sum_{a \in \cset{\pi}} \sum_{\bm w \in \weightedGnum{\pi}{1}} {\coeff}(\bm w, \pi) \E\left[   \matmom{\bm \Psi}{\bm w + \bm{\ell}_{k+1}}{\pi}{\bm a} \right],
\end{align*}
provided the latter limit exists. 
\end{lem}
\begin{proof}
First, note that for any $\bm w \in \weightedGnum{\pi}{1}$, we have:
\begin{align*}
    \|\bm w\| = \frac{1}{2} \sum_{i=1}^{k+1} \degree_i(\bm w) = \frac{1+1 + 2 |\singleblks{\pi}|}{2} = 1+ |\singleblks{\pi}| \; \text{ (See \eqref{eq: mehler_conclusion_imp})}.
\end{align*}
Furthermore, recalling that $\bm{\ell}_{k+1}$ is the weight matrix of a simple line graph, $\|\bm{\ell}_{k+1}\| = k$.
Now, we apply Lemma \ref{lemma: matrix_moment_ub} to obtain:
\begin{align*}
    |\E\left[   \matmom{\bm \Psi}{\bm w + \bm{\ell}_{k+1}}{\pi}{\bm a} \Indicator{\mathcal{E}^c} \right]|& \leq \sqrt{\E\left[   \matmom{\bm \Psi}{2\bm w + 2\bm{\ell}_{k+1}}{\pi}{\bm a} \right] } \sqrt{\P(\mathcal{E}^c)} \\
    &  \explain{(a)}{\leq} \left( \frac{C_k \log^2(m)}{m} \right)^{\frac{|\singleblks{\pi}| + 1 + k}{2}} \cdot \sqrt{\P(\mathcal{E}^c)} \\
    &\leq \left( \frac{C_k \log^2(m)}{m} \right)^{\frac{|\singleblks{\pi}| + 1 + k}{2}} \cdot \frac{C_k}{m}. 
\end{align*}
Analogously we can obtain:
\begin{align*}
    \E|\matmom{\bm \Psi}{\bm{\ell}_{k+1}}{\pi}{\bm a}| &  \leq  \left( \frac{C_k \log^2(m)}{m} \right)^{\frac{k}{2}},\\
    \E\left[|\matmom{\bm \Psi}{\bm w + \bm{\ell}_{k+1}}{\pi}{\bm a}| \right] & \explain{}{\leq}  \left( \frac{C_k \log^2(m)}{m} \right)^{\frac{|\singleblks{\pi}| + 1 + k}{2}}
\end{align*}
Further, recall that by Lemma \ref{lemma: qf_mehler_conclusion} we have:
\begin{align*}
    |\epsilon(\bm \Psi, \bm a)| & \leq C(\altprod) \cdot \left( \frac{\log^3(m)}{m \kappa^2} \right)^{\frac{2 + |\singleblks{\pi}|}{2}}.
\end{align*}
Using these estimates, we obtain:
\begin{align*}
    |\mathsf{I}| & \leq  \frac{C(\altprod)  }{m}  \cdot \sum_{\substack{\pi: \part{}{[k+1]} \\ \pi(0) \neq \pi(k+1)} }|\cset{\pi}| \cdot   \left( \frac{C_k \log^2(m)}{m}  \right)^{\frac{|\singleblks{\pi}| + 1 + k}{2}} \cdot\frac{C_k}{m}  \\& \explain{(a)}{\leq}   \frac{C(\altprod)  }{m}  \cdot \sum_{\substack{\pi: \part{}{[k+1]} \\ \pi(0) \neq \pi(k+1)} }m^{\frac{k+3+|\singleblks{\pi}|}{2}} \cdot   \left( \frac{C_k \log^2(m)}{m} \right)^{\frac{|\singleblks{\pi}| + 1 + k}{2}} \cdot \frac{C_k}{m}\\
    & = O \left( \frac{\polylog(m)}{m} \right).
\end{align*}
In addition:
\begin{align*}
    |\mathsf{II}| & \leq \frac{C(\altprod)}{m} \cdot \left( \frac{C_k \log^2(m)}{m} \right)^{\frac{k}{2}} \cdot  \sum_{\substack{\pi: \part{}{[k+1]} \\ \pi(0) \neq \pi(k+1)} }  |\cset{\pi}|  \cdot \left( \frac{\log^3(m)}{m \kappa^2} \right)^{\frac{2 + |\singleblks{\pi}|}{2}} \\&\explain{(a)}{\leq} \frac{C(\altprod)}{m} \cdot \left( \frac{C_k \log^2(m)}{m} \right)^{\frac{k}{2}} \cdot  \sum_{\substack{\pi: \part{}{[k+1]} \\ \pi(0) \neq \pi(k+1)} }  m^{\frac{k+3+|\singleblks{\pi}|}{2}}  \cdot \left( \frac{\log^3(m)}{m \kappa^2} \right)^{\frac{2 + |\singleblks{\pi}|}{2}} \\
     & = O \left( \frac{\polylog(m)}{\sqrt{m}} \right).
\end{align*}
Furthermore:
\begin{align*}
    |\mathsf{III}| & \leq \frac{C(\altprod)  }{m}  \cdot \sum_{\pi: \part{2}{[k+1]}  }|\cset{\pi}| \cdot   \left( \frac{C_k \log^2(m)}{m} \right)^{\frac{|\singleblks{\pi}| + 1 + k}{2}} \\
    & \explain{(a)}{\leq}\frac{C(\altprod)  }{m}  \cdot \sum_{\pi: \part{2}{[k+1]}  } m^{\frac{k+2+|\cset{\pi}|}{2}} \cdot   \left( \frac{C_k \log^2(m)}{m} \right)^{\frac{|\singleblks{\pi}| + 1 + k}{2}} \\
    &=O \left( \frac{\polylog(m)}{\sqrt{m}} \right).
\end{align*}
In each of the above displays, in the steps marked (a), we used the bounds on $|\cset{\pi}|$ from Lemma \ref{lemma: qf_cardinality_bounds}.  $C_k$ denotes a constant depending only on $k$ and $C(\altprod)$ denotes a constant depending only on $k$ and the functions appearing in $\altprod$. This concludes the proof of this lemma.
\end{proof}
So far we have shown that:
\begin{align*}
    \lim_{m \rightarrow \infty} \frac{\E \bm z^\UT \altprod \bm z}{m}  & =  \lim_{m \rightarrow \infty} \frac{1}{m} \sum_{\pi \in \part{1}{[k+1]}} \sum_{a \in \cset{\pi}} \sum_{\bm w \in \weightedGnum{\pi}{1}} {\coeff}(\bm w, \pi) \cdot \E\left[   \matmom{\bm \Psi}{\bm w + \bm{\ell}_{k+1}}{\pi}{\bm a} \right],
\end{align*}
provided the latter limit exists. Our goal is to show that the limit on the LHS exists and is universal across the subsampled Haar and Hadamard models. In order to do so, we will leverage the fact that the first order term in the expansion of $\E\left[   \matmom{\bm \Psi}{\bm w + \bm{\ell}_{k+1}}{\pi}{\bm a} \right]$ is the same for the two models if $\bm w +\bm{\ell}_{k+1}$ is disassortative with respect to $\pi$ and if $\bm a$ is a conflict-free labelling (Propositions  \ref{prop: clt_random_ortho} and \ref{prop: clt_hadamard}). Hence, we need to argue that the contribution of terms corresponding to $\bm w: \bm w + \bm{\ell}_{k+1} \not \in \weightedGnum{\pi}{DA}$ and $\bm a \not\in \labelling{CF}(\bm w + \bm{\ell}_{k+1}, \pi)$ are negligible. Towards this end, we consider the decomposition:
\begin{align*}
    &\frac{1}{m} \sum_{\pi \in \part{1}{[k+1]}} \sum_{a \in \cset{\pi}} \sum_{\bm w \in \weightedGnum{\pi}{1}} {\coeff}(\bm w, \pi) \cdot \E\left[   \matmom{\bm \Psi}{\bm w + \bm{\ell}_{k+1}}{\pi}{\bm a} \right] = \\& \hspace{0cm} \frac{1}{m} \sum_{\pi \in \part{1}{[k+1]}} \sum_{\substack{\bm w \in \weightedGnum{\pi}{1}\\ \bm w + \bm{\ell}_{k+1} \in \weightedGnum{\pi}{DA}}} \sum_{a \in \labelling{CF}(\bm w + \bm{\ell}_{k+1}, \pi)}  {\coeff}(\bm w, \pi) \cdot \E\left[   \matmom{\bm \Psi}{\bm w + \bm{\ell}_{k+1}}{\pi}{\bm a} \right] + \mathsf{IV} + \mathsf{V},
\end{align*}
where:
\begin{align*}
    \mathsf{IV} &\explain{def}{=} \frac{1}{m} \sum_{\pi \in \part{1}{[k+1]}} \sum_{a \in \cset{\pi}} \sum_{\substack{\bm w \in \weightedGnum{\pi}{1}\\ \bm w + \bm{\ell}_{k+1} \notin \weightedGnum{\pi}{DA}}} {\coeff}(\bm w, \pi) \cdot \E\left[   \matmom{\bm \Psi}{\bm w + \bm{\ell}_{k+1}}{\pi}{\bm a} \right], \\
    \mathsf{V} &\explain{def}{=} \frac{1}{m} \sum_{\pi \in \part{1}{[k+1]}} \sum_{\substack{\bm w \in \weightedGnum{\pi}{1}\\ \bm w + \bm{\ell}_{k+1} \in \weightedGnum{\pi}{DA}}} \sum_{a \in \cset{\pi} \backslash \labelling{CF}(\bm w + \bm{\ell}_{k+1}, \pi)}  {\coeff}(\bm w, \pi) \cdot \E\left[   \matmom{\bm \Psi}{\bm w + \bm{\ell}_{k+1}}{\pi}{\bm a} \right].
\end{align*}
\begin{lem}
We have $\mathsf{IV} \rightarrow 0, \mathsf{V} \rightarrow 0$, as $m \rightarrow \infty$, and hence:
\begin{align*}
    &\lim_{m \rightarrow \infty} \frac{\E \bm z^\UT \altprod \bm z}{m}   =\\  &\lim_{m \rightarrow \infty} \frac{1}{m} \sum_{\pi \in \part{1}{[k+1]}} \sum_{\substack{\bm w \in \weightedGnum{\pi}{1}\\ \bm w + \bm{\ell}_{k+1} \in \weightedGnum{\pi}{DA}}} \sum_{a \in \labelling{CF}(\bm w + \bm{\ell}_{k+1}, \pi)}  {\coeff}(\bm w, \pi) \cdot \E\left[   \matmom{\bm \Psi}{\bm w + \bm{\ell}_{k+1}}{\pi}{\bm a} \right],
\end{align*}
provided the latter limit exists. 
\end{lem}
\begin{proof}
We will prove this in two steps.
\begin{description}
\item [Step 1: $\mathsf{IV} \rightarrow 0$. ] We consider the two sensing models separately:
\begin{enumerate}
    \item Subsampled Hadamard Sensing: In this case, Proposition \ref{prop: clt_hadamard} tells us that if $\bm w + \bm{\ell}_{k+1} \not\in \weightedGnum{\pi}{DA}$, then: $$\E\left[   \matmom{\bm \Psi}{\bm w + \bm{\ell}_{k+1}}{\pi}{\bm a} \right] = 0,$$ and hence, $\mathsf{IV} = 0$.
    \item Subsampled Haar Sensing: Observe that, since $\|\bm w\| + \|\bm{\ell}_{k+1}\|= 1 + |\singleblks{\pi}| + k$, we have:
    \begin{align*}
        \E\left[   \matmom{\bm \Psi}{\bm w + \bm{\ell}_{k+1}}{\pi}{\bm a} \right] & = \frac{\E\left[   \matmom{\sqrt{m}\bm \Psi}{\bm w + \bm{\ell}_{k+1}}{\pi}{\bm a} \right]}{m^{\frac{1 + |\singleblks{\pi}| + k}{2}}}.
    \end{align*}
    By Proposition \ref{prop: clt_random_ortho}, we know that:
    \begin{align*}
        \left| \E\left[   \matmom{\sqrt{m}\bm \Psi}{\bm w + \bm{\ell}_{k+1}}{\pi}{\bm a} \right] - \prod_{\substack{s,t \in [|\pi|] \\ s \leq t}} \E \left[ Z_{st}^{W_{st}(\bm w + \bm{\ell}_{k+1}, \pi)} \right] \right| & \leq \frac{K_1 \log^{K_2}(m)}{m^\tbd}, 
    \end{align*}
    where $K_1,K_2,K_3$ are universal constants depending only on $k$. Note that since $\bm w + \bm{\ell}_{k+1} \notin \weightedGnum{\pi}{DA}$, we must have some $s \in [|\pi|]$ such that:
    \begin{align*}
        W_{ss}(\bm w + \bm{\ell}_{k+1}, \pi) \geq 1.
    \end{align*}
    Recall that $\degree_i(\bm w) = 0$ for any $i \not\in \{1,k+1\} \cup \singleblks{\pi}$ (since $\bm w \in \weightedGnum{\pi}{1}$), and furthermore, $|\pi(i)| = 1 \; \forall \; i \; \in \; \{1,k+1\} \cup \singleblks{\pi}$ (since $\pi \in \part{1}{k+1}$). Hence, we have $\bm w \in \weightedGnum{\pi}{DA}$ and in particular, $W_{ss}(\bm w, \pi) =0$. Consequently, we must have $W_{ss}(\bm{\ell}_{k+1}, \pi) \geq 1$. Recall that $\bm\ell_{k+1}$ is the weight matrix of a line graph:
    \begin{align*}
        (\bm\ell_{k+1})_{ij} & = \begin{cases}1 : & |i-j| = 1 \\ 0 : & \text{otherwise} \end{cases}.
    \end{align*}
    Consequently, since $W_{ss}(\bm{\ell}_{k+1}, \pi) \geq 1$, we must have for some $i \in [k]$, $\pi(i) = \pi(i+1) = \blocks_s$. However, since $\pi \in \part{1}{k+1}$, $|\blocks_s| \leq 2$, and hence, $\blocks_s = \{i,i+1\}$. This means that $W_{ss}(\bm{\ell}_{k+1},\pi) = 1 =W_{ss}(\bm w + \bm{\ell}_{k+1}, \pi)$. Consequently, since $\E Z_{ss} = 0$, we have:
    \begin{align*}
        \prod_{\substack{s,t \in [|\pi|] \\ s \leq t}} \E \left[ Z_{st}^{W_{st}(\bm w + \bm{\ell}_{k+1}, \pi)} \right] = 0,
    \end{align*}
    or
    \begin{align*}
        |\E\left[   \matmom{\bm \Psi}{\bm w + \bm{\ell}_{k+1}}{\pi}{\bm a} \right]| & \leq \frac{C_k \log^{K}(m)}{m^{\frac{1 + |\singleblks{\pi}| + k}{2}+\tbd}},
    \end{align*}
    where $C_k,K$ are constants that depend only on $k$. Recalling Lemma \ref{lemma: qf_cardinality_bounds},
    \begin{align*}
        |\cset{\pi}| & \leq m^{|\pi|} \leq m^{\frac{k+3+|\singleblks{\pi}| }{2}},
    \end{align*}
    we obtain:
    \begin{align*}
        |\mathsf{IV}| & \leq \frac{C(\altprod)}{m} \sum_{\pi \in \part{1}{[k+1]}} |\cset{\pi}|  \cdot  \frac{C_k \log^{K}(m)}{m^{\frac{1 + |\singleblks{\pi}| + k}{2}+\tbd}} = O\left( \frac{\polylog(m)}{m^\tbd}\right) \rightarrow 0.
    \end{align*}
\end{enumerate}
\item [Step 2: $\mathsf{V} \rightarrow 0$. ]  Using Lemma \ref{lemma: cf_size_bound}, we know that $$|\cset{\pi} \backslash \labelling{CF}(\bm w+\bm{\ell}_{k+1}, \pi)| \leq (k+1)^4 m^{|\pi|-1}.$$
In Lemma \ref{lemma: qf_cardinality_bounds}, we showed that for any $\pi \in \part{1}{[k+1]}$, 
\begin{align*}
    |\pi| & = \frac{k+3+|\singleblks{\pi}| }{2}.
\end{align*}
Hence,
\begin{align*}
    |\cset{\pi} \backslash \labelling{CF}(\bm w+\bm{\ell}_{k+1}, \pi)| &\leq (k+1)^4 \cdot m^{\frac{k+1+|\singleblks{\pi}| }{2}}.
\end{align*}
We already know from Lemma \ref{lemma: matrix_moment_ub} that:
\begin{align*}
    | \E\left[   \matmom{\bm \Psi}{\bm w + \bm{\ell}_{k+1}}{\pi}{\bm a} \right]| & \leq \left( \frac{C_k \log^2(m)}{m} \right)^{\frac{\|\bm w\| + \|\ell_{k+1}\|}{2}} \explain{}{\leq} \left( \frac{C_k \log^2(m)}{m} \right)^{\frac{|\singleblks{\pi}| + 1 + k}{2}}.
\end{align*}
This gives us:
\begin{align*}
    |\mathsf{V}| & \leq \frac{C}{m} \sum_{\pi \in \part{1}{[k+1]}} \sum_{\substack{\bm w \in \weightedGnum{\pi}{1}\\ \bm w + \bm{\ell}_{k+1} \in \weightedGnum{\pi}{DA}}}  |\cset{\pi} \backslash \labelling{CF}(\bm w+\bm{\ell}_{k+1}, \pi)| \cdot \left( \frac{C_k \log^2(m)}{m} \right)^{\frac{|\singleblks{\pi}| + 1 + k}{2}} \\
    & = O \left( \frac{\polylog(m)}{m} \right)
\end{align*}
which goes to zero as claimed.
\end{description}
\end{proof}
To conclude, we have shown that:
\begin{align*}
     &\lim_{m \rightarrow \infty} \frac{\E \bm z^\UT \altprod \bm z}{m}   = \\& \lim_{m \rightarrow \infty} \frac{1}{m} \sum_{\pi \in \part{1}{[k+1]}} \sum_{\substack{\bm w \in \weightedGnum{\pi}{1}\\ \bm w + \bm{\ell}_{k+1} \in \weightedGnum{\pi}{DA}}} \sum_{a \in \labelling{CF}(\bm w + \bm{\ell}_{k+1}, \pi)}  {\coeff}(\bm w, \pi) \cdot \E\left[   \matmom{\bm \Psi}{\bm w + \bm{\ell}_{k+1}}{\pi}{\bm a} \right],
\end{align*}
provided the limit on the RHS exists. In the following lemma we explicitly evaluate the limit on the RHS, and in particular, show it exists and is identical for the two sensing models. 
\begin{lem} \label{lemma: mom1_complex_formula} For both the subsampled Haar sensing and Hadamard sensing model, we have:
\begin{align*}
    \lim_{m \rightarrow \infty} \frac{\E \bm z^\UT \altprod \bm z}{m}  & = \sum_{\pi \in \part{1}{[k+1]}} \sum_{\substack{\bm w \in \weightedGnum{\pi}{1}\\ \bm w + \bm{\ell}_{k+1} \in \weightedGnum{\pi}{DA}}}   {\coeff}(\bm w, \pi) \cdot\limmom{\bm w + \bm{\ell}_{k+1}}{\pi},
\end{align*}
where,
\begin{align*}
    \limmom{\bm w+ \bm{\ell}_{k+1}}{\pi} &\explain{def}{=}  \prod_{\substack{s,t \in [|\pi|] \\ s < t}} \E \left[ Z^{W_{st}(\bm w + \bm{\ell}_{k+1}, \pi)} \right], \; Z \sim \gauss{0}{\kappa(1-\kappa)}.
\end{align*}
\end{lem}
\begin{proof}
By Propositions \ref{prop: clt_hadamard} (for the subsampled Hadamard model) and \ref{prop: clt_random_ortho} (for the subsampled  Haar model) we know that, if $\bm w+ \bm{\ell}_{k+1} \in \weightedGnum{\pi}{DA}$ and $\bm a  \in \labelling{CF}(\bm w+ \bm{\ell}_{k+1}, \pi)$, we have:
\begin{align*}
    \matmom{\sqrt{m}\bm \Psi}{\bm w + \bm{\ell}_{k+1}}{\pi}{\bm a} & = \limmom{\bm w+ \bm{\ell}_{k+1}}{\pi} + \epsilon(\bm w, \pi, \bm a),
\end{align*}
where
\begin{align*}
     |\epsilon(\bm w, \pi, \bm a)| & \leq \frac{K_1 \log^{K_2}(m)}{m^\tbd}, \; \forall \; m \geq K_3,
\end{align*}
for some constants $K_1,K_2,K_3$ depending only on $k$. Hence, we can consider the decomposition:
\begin{align*}
    \frac{1}{m} \sum_{\pi \in \part{1}{[k+1]}} \sum_{\substack{\bm w \in \weightedGnum{\pi}{1}\\ \bm w + \bm{\ell}_{k+1} \in \weightedGnum{\pi}{DA}}} \sum_{a \in \labelling{CF}(\bm w+\bm{\ell}_{k+1}, \pi)}  {\coeff}(\bm w, \pi)  \E\left[   \matmom{\bm \Psi}{\bm w + \bm{\ell}_{k+1}}{\pi}{\bm a} \right] & = \mathsf{VI} + \mathsf{VII},
\end{align*}
where:
\begin{align*}
    \mathsf{VI} & \explain{def}{=} \frac{1}{m} \sum_{\pi \in \part{1}{[k+1]}} \sum_{\substack{\bm w \in \weightedGnum{\pi}{1}\\ \bm w + \bm{\ell}_{k+1} \in \weightedGnum{\pi}{DA}}} \sum_{a \in \labelling{CF}(\bm w+\bm{\ell}_{k+1}, \pi)}  {\coeff}(\bm w, \pi) \cdot \frac{\limmom{\bm w+ \bm{\ell}_{k+1}}{\pi}}{m^{\frac{1 + \singleblks{\pi} + k}{2}}}, \\
    \mathsf{VII} & \explain{def}{=} \frac{1}{m} \sum_{\pi \in \part{1}{[k+1]}} \sum_{\substack{\bm w \in \weightedGnum{\pi}{1}\\ \bm w + \bm{\ell}_{k+1} \in \weightedGnum{\pi}{DA}}} \sum_{a \in \labelling{CF}(\bm w+\bm{\ell}_{k+1}, \pi)}  {\coeff}(\bm w, \pi) \cdot \frac{\epsilon(\bm w, \pi, \bm a)}{m^{\frac{1 + \singleblks{\pi} + k}{2}}}.
\end{align*}
We can upper bound $|\mathsf{VII}|$ as follows:
\begin{align*}
    |\labelling{CF}(\bm w+\bm{\ell}_{k+1}, \pi)| & \leq |\cset{\pi}| \explain{}{\leq} m^{\frac{k+3+|\singleblks{\pi}|}{2}}.
\end{align*}
Thus:
\begin{align*}
|\mathsf{VII}| & \leq \frac{C(\altprod)}{m} \cdot C_k \cdot |\labelling{CF}(\bm w+\bm{\ell}_{k+1}, \pi)| \cdot \frac{1}{m^{\frac{1 + |\singleblks{\pi}| + k}{2}}} \cdot  \frac{K_1 \log^{K_2}(m)}{m^\tbd} \\&= O\left( \frac{\polylog(m)}{m^\tbd} \right) \rightarrow 0.
\end{align*}
Moreover, can compute:
\begin{align*}
    &\lim_{m \rightarrow \infty} (\mathsf{VI})  = \lim_{m \rightarrow \infty} \frac{1}{m} \sum_{\pi \in \part{1}{[k+1]}} \sum_{\substack{\bm w \in \weightedGnum{\pi}{1}\\ \bm w + \bm{\ell}_{k+1} \in \weightedGnum{\pi}{DA}}} \sum_{a \in \labelling{CF}(\bm w+\bm{\ell}_{k+1}, \pi)}  {\coeff}(\bm w, \pi) \cdot \frac{\limmom{\bm w+ \bm{\ell}_{k+1}}{\pi}}{m^{\frac{1 + \singleblks{\pi} + k}{2}}} \\
    & = \lim_{m \rightarrow \infty} \frac{1}{m} \sum_{\pi \in \part{1}{[k+1]}} \sum_{\substack{\bm w \in \weightedGnum{\pi}{1}\\ \bm w + \bm{\ell}_{k+1} \in \weightedGnum{\pi}{DA}}}   {\coeff}(\bm w, \pi) \cdot \frac{\limmom{\bm w+ \bm{\ell}_{k+1}}{\pi}}{m^{\frac{1 + |\singleblks{\pi}| + k}{2}}} \cdot |\labelling{CF}(\bm w+\bm{\ell}_{k+1}, \pi)| \\
    % & = \lim_{m \rightarrow \infty} \frac{1}{m} \sum_{\pi \in \part{1}{[k+1]}} \sum_{\substack{\bm w \in \weightedGnum{\pi}{1}\\ \bm w + \bm{\ell}_{k+1} \in \weightedGnum{\pi}{DA}}}   {\coeff}(\bm w, \pi) \cdot\limmom{\bm w+ \bm{\ell}_{k+1}}{\pi} \cdot  \frac{m^{|\pi|}}{m^{\frac{1 + |\singleblks{\pi}| + k}{2}}} \cdot \frac{|\labelling{CF}(\bm w+\bm{\ell}_{k+1}, \pi)|}{m^{|\pi|}} \\
    & \explain{(a)}{=} \lim_{m \rightarrow \infty}  \sum_{\pi \in \part{1}{[k+1]}} \sum_{\substack{\bm w \in \weightedGnum{\pi}{1}\\ \bm w + \bm{\ell}_{k+1} \in \weightedGnum{\pi}{DA}}}   {\coeff}(\bm w, \pi) \cdot\limmom{\bm w+ \bm{\ell}_{k+1}}{\pi} \cdot   \frac{|\labelling{CF}(\bm w+\bm{\ell}_{k+1}, \pi)|}{m^{|\pi|}} \\
    & \explain{(b)}{=}   \sum_{\pi \in \part{1}{[k+1]}} \sum_{\substack{\bm w \in \weightedGnum{\pi}{1}\\ \bm w + \bm{\ell}_{k+1} \in \weightedGnum{\pi}{DA}}}   {\coeff}(\bm w, \pi) \cdot\limmom{\bm w+ \bm{\ell}_{k+1}}{\pi}.
\end{align*}
In the step marked (a) we used the fact that $|\pi| = (3 + |\singleblks{\pi}| +k)/2$ for any $\pi \in \part{1}{[k+1]}$ (Lemma \ref{lemma: qf_cardinality_bounds}), and in step (b) we used Lemma \ref{lemma: cf_size_bound} ($|\labelling{CF}(\bm w+\bm{\ell}_{k+1}, \pi)|/m^{|\pi|} \rightarrow 1$).
This proves the claim of the lemma.
\end{proof}

In the following lemma, we show that the combinatorial sum obtained in Lemma \ref{lemma: mom1_complex_formula} can be significantly simplified. 

\begin{lem}For both the subsampled Haar sensing and Hadamard sensing models, we have:
\begin{align*}
     \lim_{m \rightarrow \infty} \frac{\E \bm z^\UT \altprod \bm z}{m}& =  (1-\kappa)^k \cdot \prod_{i=1}^{k-1} \hat{q}_i(2).
\end{align*}
In particular,  Proposition \ref{prop: qf_univ_mom1} holds.
\end{lem}
\begin{proof}
We claim that the only partition with a non-zero contribution is:
\begin{align*}
    \pi & = \bigsqcup_{i=1}^{k+1}\{i\}.
\end{align*}
In order to see this, suppose $\pi$ is not entirely composed of singleton blocks. Define:
\begin{align*}
    i_\star & \explain{def}{=}\min  \{i \in [k+1]: |\pi(i)| > 1  \}.
\end{align*}
Note that $i_\star > 1$ since we know that $|\pi(1)| = |\firstblk{\pi}| = 1$ for any $\pi \in \part{1}{k+1}$.
Since $\pi \in \part{1}{[k+1]}$, we must have $|\pi(i_\star)| = 2$, hence, denote:
\begin{align*}
    \pi(i_\star)= \{i_\star, j_\star\},
\end{align*}
for some $j_\star > i_\star+1$ ($i_\star \leq j_\star$ since it is the first index which is not in a singleton block, and $j_\star \neq i_\star + 1$ since otherwise $\bm w + \bm{\ell}_{k+1}$ will not be disassortative). 
Let us label the first few blocks of $\pi$ as:
\begin{align*}
    \blocks_1 = \{1\}, \; \blocks_2 = \{2\}, \dots, \blocks_{i_\star - 1} = \{i_\star - 1\}, \; \blocks_{i_\star} = \{i_\star, j_\star\}.
\end{align*}
Next, we compute:
\begin{align*}
    W_{i_\star-1,i_\star}(\bm w + \bm{\ell}_{k+1}, \pi) & = W_{i_\star-1,i_\star}(\bm{\ell}_{k+1}, \pi) + W_{i_\star-1,i_\star}(\bm w , \pi) \\
    & \explain{(a)}{=} W_{i_\star-1,i_\star}(\bm{\ell}_{k+1}, \pi) \\
    & \explain{(b)}{=} \mathbf{1}_{i_{\star} - 1 \in \blocks_{i_\star-1}}  + \mathbf{1}_{i_{\star} + 1 \in \blocks_{i_\star-1}} + \mathbf{1}_{j_{\star} - 1 \in \blocks_{i_\star-1}}  + \mathbf{1}_{j_{\star} + 1 \in \blocks_{i_\star-1}} \\
    & \explain{(c)}{=} \mathbf{1}_{i_{\star} - 1 = {i_\star-1}}  + \mathbf{1}_{i_{\star} + 1 ={i_\star-1}} + \mathbf{1}_{j_{\star} - 1 ={i_\star-1}}  + \mathbf{1}_{j_{\star} + 1 = {i_\star-1}} \\
    & \explain{(d)}{=} 1.
\end{align*}
In the step marked (a), we used the fact that since $\bm w \in \weightedGnum{\pi}{1}$ and $|\pi(i_\star)| = |\pi(j_\star)| = 2$, we must have $d_{i_\star}(\bm w) = d_{j_\star}(\bm w) = 0$ and $W_{i_\star-1,i_\star}(\bm w, \pi) = 0$. In the step marked (b), we used the definition of $\bm{\ell}_{k+1}$ (that it is the line graph). In the step marked (c), we used the fact that $\blocks_{i_\star-1} = \{i_{\star-1}\}$. In the step marked (d), we used the fact that $j_\star > i_\star + 1$.

Hence, we have shown that for any $\pi \neq \sqcup_{i=1}^{k+1} \{i\}$, we have:
\begin{align*}
    \mu(\bm w, \pi) = 0 \; \forall \; \bm w \text{ such that}  \; \bm w \in \weightedGnum{\pi}{1}, \; \bm w + \bm{\ell}_{k+1} \in \weightedGnum{\pi}{DA}.
\end{align*}
Next, let $\pi = \sqcup_{i=1}^{k+1} \{i\}$. We observe for any $\bm w$ such that $\bm w \in \weightedGnum{\pi}{1}, \; \bm w + \bm{\ell}_{k+1} \in \weightedGnum{\pi}{DA}$, we have:
\begin{align*}
    \limmom{\bm w+ \bm{\ell}_{k+1}}{\pi} &\explain{}{=}  \prod_{\substack{s,t \in [|\pi|] \\ s < t}} \E \left[ Z^{W_{st}(\bm w + \bm{\ell}_{k+1}, \pi)} \right], \; Z \sim \gauss{0}{\kappa(1-\kappa)} \\
    & = \prod_{\substack{i,j \in [k+1] \\ i < j}} \E \left[ Z^{w_{ij} + ({\ell}_{k+1})_{ij}, \pi)} \right], \; Z \sim \gauss{0}{\kappa(1-\kappa)}.
\end{align*}
Note that since $\E Z = 0$, for $\limmom{\bm w + \bm{\ell}_{k+1}}{\pi} \neq 0$, we must have:
\begin{align*}
    w_{ij} \geq (\ell_{k+1})_{ij}, \; \forall \; i,j \; \in \; [k].
\end{align*}
However, since $\bm w \in \weightedGnum{\pi}{1}$ we have:
\begin{align*}
    \degree_1(\bm w) = \degree_{k+1}(\bm w) = 1, \; \degree_i(\bm w) = 2 \; \forall \; i \; \in \; [2:k],
\end{align*}
so, $\bm w = \bm{\ell}_{k+1}$. Hence, recalling the formula for $\coeff(\bm w, \pi)$ from Lemma \ref{lemma: qf_mehler_conclusion}, we obtain:
\begin{align*}
     \lim_{m \rightarrow \infty} \frac{\E \bm z^\UT \altprod \bm z}{m}& =  (1-\kappa)^k \cdot \prod_{i=1}^{k-1} \hat{q}_i(2).
\end{align*}
This proves the statement of the lemma and also Proposition \ref{prop: qf_univ_mom1} (see Remark \ref{remark: all_types_qf_mom1} regarding how the analysis extends to other types).
\end{proof}

Throughout this section, we assumed that the alternating product $\altprod$ was of Type I. The following remark outlines how the analysis of this section extends to other types.

\begin{rem} \label{remark: all_types_qf_mom1} The analysis of the other cases can be reduced to Type 1 as follows: Consider an  alternating form $\altprod(\bm \Psi, \bm Z)$ of Type 1:
    \begin{align*}
        \altprod = p_1(\bm \Psi) q_1(\bm Z) p_1(\bm \Psi) \cdots  q_{k-1}(\bm Z) p_k(\bm \Psi),
    \end{align*}
    but the more general quadratic form:
    \begin{align} \label{eq: general_qf}
        \frac{1}{m} \E \alpha(\bm z)^\UT \altprod(\bm \Psi, \bm Z) \beta(\bm z),
    \end{align}
    where $\alpha, \beta: \R \rightarrow \R$ are \underline{odd functions} whose absolute values can be upper bounded by a polynomial. They act on the vector $\bm z$ entry-wise. This covers all the types in a unified way:
    \begin{enumerate}
        \item For Type 1 case: We  take $\alpha(z) = \beta(z) = z$.
        \item For the Type 2 case, we write: $$ \bm z^\UT p_1(\bm \Psi) q_1(\bm Z) p_1(\bm \Psi) \cdots  q_k(\bm Z) p_k(\bm \Psi) q_{k}(\bm Z) \bm z = \alpha(\bm z)^\UT \altprod(\bm \Psi, \bm Z) \beta(\bm z),$$
        where $\alpha(z) = z, \beta(z) = z q_{k}(z)$.
        \item For the Type 3 case: $$ \bm z^\UT  q_0(\bm Z) p_1(\bm \Psi) q_1(\bm Z) p_1(\bm \Psi) \cdots  q_{k-1}(\bm Z) p_k(\bm \Psi) q_{k}(\bm Z) \bm z = \alpha(\bm z)^\UT \altprod(\bm \Psi, \bm Z) \beta(\bm z),$$
        where $\alpha(z) = z q_0(z), \beta(z) = z q_{k}(z)$.
        \item For the Type 4 case: $$ \bm z^\UT q_0(\bm Z) p_1(\bm \Psi) q_1(\bm Z) p_2(\bm \Psi) \cdots  q_{k-1}(\bm Z) p_k(\bm \Psi) \bm z = \alpha(\bm z)^\UT \altprod(\bm \Psi, \bm Z) \beta(\bm z),$$
        where $\alpha(z) = z q_0(z), \beta(z) = z$.
    \end{enumerate}
    The analysis of the more general quadratic form in \eqref{eq: general_qf} is analogous to the analysis outlined in this section. Lemmas \ref{lemma : qf_variance_normalization} and \ref{lemma: qf_diagonal_removal} extend straightforwardly. Inspecting the proof of Lemma \ref{lemma: qf_mehler_conclusion} shows that the same error bound continues to hold (after suitably redefining $c(\bm w, \pi)$), since $\alpha,\beta$ are odd (as in the case $\alpha(z) = \beta(z) = z$). The subsequent lemmas after that hold verbatim for the more general quadratic form \eqref{eq: general_qf}.
\end{rem}

\section{Conclusion and Future Work}
\label{seq:conclusion}
\rishabh{In this work, we analyzed the dynamics of linearized approximate message passing algorithms for phase retrieval when the sensing matrix $\bm A$ is generated by sub-sampling $n = \kappa m$ columns of a $m \times m$ orthogonal matrix $\bm U$, and the signal $\bm x$ is drawn from a Gaussian prior $\bm x \sim \gauss{\bm 0}{\bm I_n/\kappa}$. We focused on two particular choices of the orthogonal matrix $\bm U$, which led to the following specific sensing models:
\begin{enumerate}[(a)]
    \item The sub-sampled Haar model: In this case $\bm U = \bm O$, a uniformly random orthogonal matrix $\bm O \sim  \unif{\{\mathbb{O}(m)\}}$. 
    \item The sub-sampled Hadamard model: In this case $\bm U = \bm H$, the $m \times m$ Hadamard-Walsh matrix. 
\end{enumerate}
We showed that the dynamics of linearized AMP algorithms for these two sensing ensembles are asymptotically indistinguishable. Our analysis uncovered the following probabilistic mechanism behind this underlying universality phenomenon:
\begin{enumerate}
    \item The relevant observables of interest for linearized AMP algorithms can be written as functions of the matrix $\bm \Psi \explain{def}{=}  \bm A \bm A^\UT - \kappa \bm I_m$ and $\bm z$, the vector of signed measurements $\bm z \explain{def}{=} \bm A \bm x$. These functions are the normalized trace $\Tr(\altprod(\bm \Psi, \bm z))/m$ and the quadratic form $\bm z^\UT \altprod(\bm \Psi, \bm z) \bm z/m$ of the alternating product $\altprod(\bm \Psi, \bm z)$ introduced in Definition \ref{def: alternating_product}. 
    \item When the signal $\bm x$ is drawn from the Gaussian prior, the law of the signed measurements conditioned on $\bm A$ is a correlated Gaussian distribution $\bm z \sim \gauss{\bm 0}{\bm I + \bm \Psi/\kappa}$. A consequence of Gaussianity is that expectations of arbitrary functions of $\bm z$ can be expressed in terms of its covariance matrix, which is determined by $\bm \Psi$, using Mehler's Formula (Proposition \ref{proposition: mehler}). Hence, the expectations of the observables of interest for linearized AMP algorithms can be written as certain polynomials in the entries of the matrix $\bm \Psi$.
    \item The observables of interest behave universally since the matrix $\bm \Psi$ has similar probabilistic properties under the sub-sampled Haar sensing and sub-sampled Hadamard sensing models. These properties are stated below.
    \begin{enumerate}[i)]
        \item \emph{Delocalization.} The entries of the matrix $\bm \Psi$ are delocalized in the sense:
        \begin{align} \label{eq:psi-delocalized}
     \|\bm \Psi\|_\infty & \leq O \left( \frac{ \polylog(m) }{\sqrt{m}}\right) \; \text{ with high probability}.
        \end{align}
        This was shown in Lemma \ref{concentration}, which crucially used the fact that both the Haar matrix $\bm O$ (with high probability) and the Hadamard-Walsh matrix are themselves delocalized:
        \begin{align}\label{eq:delocalized}
            \|\bm H\|_\infty \leq \frac{1}{\sqrt{m}}, \; \|\bm O\|_\infty & \leq O \left( \frac{ \polylog(m) }{\sqrt{m}}\right) \; \text{ with high probability}.
        \end{align}
        \item \emph{CLT Behavior.} As shown in Propositions \ref{prop: clt_random_ortho} and \ref{prop: clt_hadamard}  and Lemma \ref{lemma: cf_size_bound}, \emph{most} entries of $\bm \Psi$ satisfy the same central limit theorem under the two sensing models. The proof of these results relied on the delocalization properties of the Haar matrices and Hadamard-Walsh matrices (cf. \eqref{eq:delocalized}) and the following structural property of Hadamard-Walsh matrices (cf. Lemma \ref{lemma: hadamard_key_property}), which expresses the entry-wise product of two rows of the Hadamard-Walsh matrix, in terms of another row of the Hadamard-Walsh matrix:
    \begin{align}\label{eq:hadamard-key-property}
         \sqrt{m}\bm h_i \odot \bm h_j & = \bm h_{i \oplus j}.
    \end{align}
This formula allowed us to verify that most pairs of distinct entries of  $\bm \Psi$ converge in distribution to a pair of asymptotically uncorrelated Gaussians in the sub-sampled Hadamard model; as is true for all distinct pairs of entries of $\bm \Psi$ in the sub-sampled Haar model.   
    \end{enumerate}
    Due to these similarities in the behavior of $\bm \Psi$ under the two sensing models, the leading order behavior of the relevant polynomials of $\bm \Psi$ (which determine the observables of interest for linearized AMP algorithms) is identical in these two models, leading to universality in the dynamics of linearized AMP algorithms.
\end{enumerate}
In the following paragraphs, we discuss some interesting directions for future work.
\paragraph{Other structured ensembles} While we focused on the sub-sampled Hadamard sensing model in this paper, we believe our proof techniques should extend to structured sensing matrices with orthogonal columns, particularly those constructed by randomly sub-sampling other orthogonal matrices like the Discrete Fourier Transform (DFT) matrix and the Discrete Cosine Transform (DCT) matrix. To do so, one would need to verify that the matrix $\bm \Psi$ under these models satisfies the properties outlined in item (3) of the probabilistic mechanism discussed above. Indeed, it is straightforward to check that the matrix $\bm \Psi$ is \emph{delocalized} in the sense of \eqref{eq:psi-delocalized} since DFT and DCT matrices satisfy similar delocalization estimates as Hadamard-Walsh matrices (cf. \eqref{eq:delocalized}). Furthermore, since DCT and DFT matrices have convenient formulae for their entries like Hadamard-Walsh matrices, we expect that it should be possible to verify that most entries of $\bm \Psi$ have identical \emph{CLT behavior} under the sub-sampled DFT and DCT models and the sub-sampled Haar model. Specifically, the rows $\bm f_{1:m}$ of DFT matrices satisfy the following analog of \eqref{eq:hadamard-key-property}:
\begin{align*}
    \bm f_i \odot \bm f_j = \bm f_{(i + j - 2 \mod{m}) + 1},
\end{align*}
and for DCT matrices, we anticipate a suitable analog of the above result can be proved using trignometric identities. 
\paragraph{Non-linear AMP Algorithms} Our results hold for linearized AMP algorithms, which are not the state-of-the-art message-passing algorithms for phase retrieval. It would be interesting to extend our results to include general non-linear AMP algorithms such as the algorithm in \eqref{eq: NL_AMP}, which also seems to exhibit universality (see \citep[Figure 2]{maillard2020phase}). The key challenge in doing so is that while the relevant observables for non-linear AMP algorithms such as the one in \eqref{eq: NL_AMP} can still be expressed as functions of the matrix $\bm \Psi$ and the vector $\bm z$, these functions appear to be significantly more complicated than the normalized trace $\Tr(\altprod(\bm \Psi, \bm z))/m$ and the quadratic form $\bm z^\UT \altprod(\bm \Psi, \bm z) \bm z/m$ of the alternating products $\altprod(\bm \Psi, \bm z)$ that appeared in the analysis of linearized AMP algorithms. 
\paragraph{Non-Gaussian Priors} Simulations show that the universality of the dynamics of linearized AMP algorithms continues to hold even if the signal is not drawn from a Gaussian prior and is an actual image. However, a limitation of the current proof technique is that it crucially uses the Gaussian prior assumption on the signal $\bm x$. This assumption is used in item (2) of the probabilistic mechanism for universality described above: when the signal $\bm x \sim \gauss{\bm 0}{\bm I/\kappa}$ the law of $\bm z$ conditioned on the randomness in the sensing matrix is a correlated Gaussian distribution with a covariance matrix determined by $\bm \Psi$. As a consequence of Gaussianity, expectations of the observables of interest for linearized AMP algorithms can be expressed as polynomials in the entries of the matrix $\bm \Psi$ using Mehler's formula. An exciting direction for future work is to extend our results beyond i.i.d. Gaussian signals to the situation when the signal is drawn from a general i.i.d. prior. In this situation, due to the central limit theorem, the entries of $\bm z$ are no longer precisely Gaussian but only approximately so. It would be interesting to investigate if approximate Gaussianity of $\bm z$ is sufficient to obtain similar results.}
%%%%%%%%%%%%%%%%%%%%%%%%%%%%%%%%%%%%%%%%%%%%%%%%%%%%%%%%%%%%%
%%%%%%%%%%%%%%%%%%%%%%%%%%%%%%%%%%%%%%%%%%%%%%%%%%%%%%%%%%%%%
%%%%%%%%%%%%%%%%%%%%%%%%%%%%%%%%%%%%%%%%%%%%%%%%%%%%%%%%%%%%%

\bibliographystyle{plainnat}
\bibliography{ref}

\appendix

%%%%%%%%%%%%%%%%%%%%%%%%%%%%%%%%%%%%%%%%%%%%%%%%%%%%%%%%%%%%
%%%%%%%%%%%%% APPENDIX-TRACE%%%%%%%%%%%%%%%%%%%%%%%%%%%%%%%%
%%%%%%%%%%%%%%%%%%%%%%%%%%%%%%%%%%%%%%%%%%%%%%%%%%%%%%%%%%%%
\section{Proof of Lemmas \ref{lemma: trace_good_event} and \ref{lemma: continuity_trace}}
\label{appendix: free_probability_trace}

\subsection{Proof of Lemma \ref{lemma: trace_good_event}}

\begin{proof}[Proof of Lemma \ref{lemma: trace_good_event}]
Recall that, $\bm A \bm A^\UT = \bm U \bm B \bm U^\UT, \; \bm \Psi = \bm A \bm A^\UT - \E [\bm A \bm A^\UT | \bm U] =  \bm U (\bm B - \kappa \bm I_m) \bm U^\UT$ where $\bm B$ is a uniformly random $m \times m$ diagonal matrix with exactly $n$ entries set to $1$ and the remaining entries set to $0$. 
Using the concentration inequality of Lemma \ref{concentration}:
\begin{align}
   \P \left( | (\bm A \bm A^\UT)_{ij} - \E (\bm A \bm A^\UT)_{ij}| > \epsilon \;  \big| \; \bm U\right) & \leq  4 \exp\left( - \frac{\epsilon^2}{8m \|\bm U\|_\infty^4} \right), \label{eq: conc_1}.
\end{align}
Setting $\epsilon = \sqrt{32 \cdot m \cdot  \|\bm U\|_\infty^4 \cdot  \log(m)}$ in \eqref{eq: conc_1} we obtain,
\begin{align*}
     \P \left( | (\bm A \bm A^\UT)_{ij} - \E (\bm A \bm A^\UT)_{ij}| > \sqrt{32 \cdot m \cdot  \|\bm U\|_\infty^4 \cdot  \log(m)} \;  \big| \; \bm U\right) & \leq  \frac{4}{m^4}.
\end{align*}
 By a union bound, $\P(\mathcal{E}^c | \bm U) \leq 4/m^2 \rightarrow 0$. In order to prove the claim of the lemma for the subsampled Haar model, we first note that by Fact \ref{fact: concentration_sphere} we have,
\begin{align*}
    \P\left( |O_{ij}| > \sqrt{\frac{8\log(m)}{m}} \right) & \leq  \frac{2}{m^4}.
\end{align*}
By a union bound $\P(\|\bm O\|_\infty > \sqrt{8\log(m)/m}) \leq 2m^{-2}$. This gives us:
\begin{align*}
    \P\left( \left\{\|\bm O\|_\infty \leq \sqrt{\frac{8\log(m)}{m}} \right\} \cap \mathcal{E} \right) &\geq 1 - \P\left(\|\bm O\|_\infty > \sqrt{\frac{8\log(m)}{m}}\right) - \P(\mathcal{E}^c) \\
    & \geq 1 - \frac{2}{m^2} - \E \P(\mathcal{E}^c|\bm U) \\
    & \geq 1 - \frac{6}{m^2}.
\end{align*}
This concludes the proof of the lemma. 
\end{proof}
\subsection{Proof of Lemma \ref{lemma: continuity_trace}}
\begin{proof}[Proof of Lemma \ref{lemma: continuity_trace}] Consider any alternating product $\mathcal{A}$ (see Definition \ref{def: alternating_product}):
\begin{align*}
    \altprod(\bm \Psi,\bm Z) = (\bm \Psi) q_1(\bm Z) (\bm \Psi) \cdots  q_k(\bm Z).
\end{align*}
Note that in the above expression, we have assumed the alternating product is of Type 2 but the following argument applies to all the other types too. We define:
\begin{align*}
    \altprod_i & = (\bm \Psi) q_1(\bm Z) (\bm \Psi) q_2(\bm Z) \cdots (\bm \Psi) q_i(\bm Z) (\bm \Psi) q_{i+1}(\bm Z^\prime) (\bm \Psi) q_{i+2}(\bm Z^\prime) \cdots (\bm \Psi ) q_k(\bm Z^\prime).
\end{align*}
Then we can express $\altprod(\bm \Psi, \bm Z^\prime)  - \altprod(\bm \Psi, \bm Z) $ as a telescoping sum:
\begin{align*}
    \altprod(\bm \Psi, \bm Z)  - \altprod(\bm \Psi, \bm Z^\prime) & = \sum_{i=1}^k (\altprod_{i}-\altprod_{i-1}).
\end{align*}
Hence,
\begin{align*}
    \left| \frac{\Tr \altprod(\bm \Psi, \bm Z)}{m} - \frac{\Tr \altprod(\bm \Psi, \bm Z^\prime)}{m} \right| & \leq \frac{1}{m} \sum_{i=1}^k |\Tr(\altprod_i - \altprod_{i-1})|.
\end{align*}
Next we observe that:
\begin{align*}
    &|\Tr(\altprod_i - \altprod_{i-1})| \\& = |\Tr( (\bm \Psi) q_1(\bm Z) \cdots (\bm \Psi) q_{i-1}(\bm Z) \cdot (q_i(\bm Z) - q_{i}(\bm Z^\prime)) \cdot (\bm \Psi) q_{i+1}(\bm Z^\prime) \cdots (\bm \Psi) q_{k}(\bm Z^\prime))| \\
    & \leq \left\|(\bm \Psi) q_1(\bm Z) \cdots (\bm \Psi) q_{i-1}(\bm Z) \cdot  (\bm \Psi) q_{i+1}(\bm Z^\prime) \cdots (\bm \Psi) q_{k}(\bm Z^\prime) \right\|_\op \cdot \left( \sum_{j=1}^m |q_i(z_j) - q_i(z_j^\prime)| \right) \\
    & \leq \|(\bm \Psi)\|_{\op} \|q_1(\bm Z)\|_\op \cdots \| (\bm \Psi)\|_\op \| q_{k}(\bm Z^\prime)\|_\op \cdot \left( \sum_{j=1}^m |q_i(z_j) - q_i(z_j^\prime)| \right)\\
    & \explain{(a)}{\leq} \left(\prod_{j=1}^k \|q_j\|_\infty \right) \cdot \|q_i\|_{\mathsf{Lip}} \cdot \left( \sum_{j=1}^m |z_j - z_j^\prime| \right) \\
    & \leq \sqrt{m} \cdot C(\altprod) \cdot \|\bm Z- \bm Z^\prime\|_\fr. 
\end{align*}
In the step marked (a), we observed that: $\|(\bm \Psi)\|_{\op} = \|\bm U (\barB) \bm U^\UT\|_{\op} \leq \max(|\kappa)|, |1-\kappa|) \leq 1$. Similarly, $\|q_j(\bm Z)\|_{\op} \leq \|q_j\|_\infty \explain{def}{=} \sup_{\xi \in \R} |q_j(\xi)|$. We also recalled the functions $q_i$ are assumed to be Lipchitz and denoted the Lipchitz constant of $q_i$ by $\|q_i\|_{\mathsf{Lip}}$. Hence we obtain:
\begin{align*}
    \left| \frac{\Tr \altprod(\bm \Psi, \bm Z)}{m} - \frac{\Tr \altprod(\bm \Psi, \bm Z^\prime)}{m} \right| & \leq \frac{k \cdot  C(\altprod)}{\sqrt{m}} \cdot  \|\bm Z- \bm Z^\prime\|_\fr.
\end{align*}
This concludes the proof of the lemma.
\end{proof}

%%%%%%%%%%%%%%%%%%%%%%%%%%%%%%%%%%%%%%%%%%%%%%%%%%%%%%%%%%%%%
%%%%%%%%%%%%%%%%%%%%%%%%%%%%%%%%%%%%%%%%%%%%%%%%%%%%%%%%%%%%%
%%%%%%%%%%%%%%%%%%%%%%%%%%%%%%%%%%%%%%%%%%%%%%%%%%%%%%%%%%%%%

%%%%%%%%%%%%%%%%%%%%%%%%%%%%%%%%%%%%%%%%%%%%%%%%%%%%%%%%%%%%
%%%%%%%%%%%%%%%%%% APPENDIX-QF-VARIANCE %%%%%%%%%%%%%%%%%%%%%%%%%%%%%%%%
%%%%%%%%%%%%%%%%%%%%%%%%%%%%%%%%%%%%%%%%%%%%%%%%%%%%%%%%%%%%
\section{Proof of Proposition \ref{prop: qf_univ_mom2}}
\label{supplement: qf_univ_mom2}
The proof of Proposition \ref{prop: qf_univ_mom2} is very similar to the proof of Proposition \ref{prop: qf_univ_mom1} and hence we will be brief in our arguments. 

As discussed in the proof of Proposition \ref{prop: qf_univ_mom1}, we will assume that alternating form is of Type 1. The other types are handled as outlined in Remark \ref{remark: all_types_qf_mom1}. Furthermore, in light of Lemma \ref{lemma: poly_psi_simple} we can further assume that all polynomials $p_i(\psi) = \psi$. Hence we assume that $\altprod$ is of the form: $$\altprod(\bm \Psi, \bm Z) = \bm \Psi q_1(\bm Z)\bm \Psi \cdots  q_{k-1}(\bm Z) \bm \Psi.$$

The proof of Proposition \ref{prop: qf_univ_mom2} consists of various steps which will be organized as separate lemmas. We begin by recall that 
\begin{align*}
    \bm z & \sim \gauss{0}{\frac{\bm A \bm A^\UT}{\kappa}}.
\end{align*}
Define the event:
    \begin{align}
        \mathcal{E}  &= \left\{ \max_{i \neq j} | (\bm A \bm A^\UT|)_{ij} \leq \sqrt{\frac{2048 \cdot  \log^{3}(m)}{m}}, \; \max_{i \in [m]} | (\bm A \bm A^\UT)_{ii} - \kappa | \leq \sqrt{\frac{2048 \cdot  \log^{3}(m)}{m}} \right\} %\label{eq: good_event_qf_firstmom}
    \end{align}
    By Lemma \ref{lemma: trace_good_event} we know that $\P(\mathcal{E}^c) \rightarrow 0$ for both the subsampled Haar sensing and the subsampled Hadamard model. 
    We define the normalized random vector $\widetilde{\bm z}$ as:
\begin{align*}
    \widetilde{z}_i & = \frac{z_i}{\sigma_i}, \; \sigma_i^2 = \frac{(\bm A \bm A^\UT)_{ii}}{\kappa} %\label{eq: tilde_z_distribution}
\end{align*}
Note that conditional on $\bm A$, $\widetilde{\bm z}$ is a zero mean Gaussian vector with: $$\E[\widetilde{z_i}^2 | \bm A] =1, \; \E[ \widetilde{z}_i \widetilde{z_j} | \bm A] = \frac{(\bm A \bm A^\UT)_{ij}/\kappa}{\sigma_i \sigma_j}.$$ We define the diagonal matrix $\widetilde{\bm Z} = \diag{\widetilde{\bm z}}$.

\begin{lem}\label{lemma : qf_mom2_variance_normalization} We have, \begin{align*}
    \lim_{m \rightarrow \infty} \frac{\E (\bm z^\UT \altprod(\bm \Psi, \bm Z) \bm z)^2}{m^2}&=\lim_{m \rightarrow \infty} \frac{ \E(\widetilde{\bm z}^{\UT} \altprod(\bm\Psi, \widetilde{\bm Z}) \widetilde{\bm z})^2}{m^2} \Indicator{\mathcal{E}},
\end{align*}
provided the latter limit exists.
\end{lem}
The proof of this lemma is analogous the proof of Lemma \ref{lemma : qf_variance_normalization} and is omitted.   The advantage of Lemma \ref{lemma : qf_mom2_variance_normalization} is that $\widetilde{z}_i \sim \gauss{0}{1}$ and on the event $\mathcal{E}$ the coordinates of  $\widetilde{\bm z}$ have weak correlations. Consequently, Mehler's Formula (Proposition \ref{proposition: mehler}) can be used to analyze the leading order term in $\E[\widetilde{\bm z}^{\UT} \altprod(\bm\Psi, \widetilde{\bm Z}) \widetilde{\bm z} \;  \Indicator{\mathcal{E}}]$. Before we do so, we do one additional preprocessing step. 

\begin{lem}\label{lemma: qf_mom2_diagonal_removal} We have, 
\begin{align*}
    \lim_{m \rightarrow \infty} \frac{ \E(\widetilde{\bm z}^{\UT} \altprod(\bm\Psi, \widetilde{\bm Z}) \widetilde{\bm z})^2}{m^2} \Indicator{\mathcal{E}}&=\lim_{m \rightarrow \infty} \frac{\E \; \Tr (\altprod \cdot (\widetilde{\bm z} \widetilde{\bm z}^\UT -  \widetilde{\bm Z}^2) \cdot \altprod \cdot (\widetilde{\bm z} \widetilde{\bm z}^\UT -  \widetilde{\bm Z}^2)  ) \Indicator{\mathcal{E}}}{m^2},
\end{align*}
provided the latter limit exists. 
\end{lem}
\begin{proof}[Proof Sketch] Observe that we can write:
    \begin{align*}
        &(\widetilde{\bm z}^\UT \altprod \widetilde{\bm z})^2  =\Tr(\altprod \cdot \widetilde{\bm z} \widetilde{\bm z}^\UT \cdot \altprod \cdot \widetilde{\bm z} \widetilde{\bm z}^\UT)  \\
        & = \Tr(\altprod \cdot (\widetilde{\bm z}\widetilde{\bm z}^\UT -\widetilde{\bm Z}^2 + \widetilde{\bm Z}^2) \cdot \altprod \cdot (\widetilde{\bm z} \widetilde{\bm z}^\UT-\widetilde{\bm Z}^2 + \widetilde{\bm Z}^2)) \\
        & = \Tr(\altprod \cdot (\widetilde{\bm z}\widetilde{\bm z}^\UT -\widetilde{\bm Z}^2) \cdot \altprod \cdot (\widetilde{\bm z} \widetilde{\bm z}^\UT-\widetilde{\bm Z}^2)) + \Tr(\altprod \cdot  \widetilde{\bm Z}^2 \cdot \altprod \cdot \widetilde{\bm z} \widetilde{\bm z}^\UT) + \Tr(\altprod \cdot \widetilde{\bm z} \widetilde{\bm z}^\UT \cdot  \widetilde{\bm Z}^2 \cdot\altprod  )  \\& \hspace{6cm}- \Tr(\altprod \cdot \widetilde{\bm Z}^2  \cdot \altprod \cdot \widetilde{\bm Z}^2) \\
        & = \Tr(\altprod \cdot (\widetilde{\bm z}\widetilde{\bm z}^\UT -\widetilde{\bm Z}^2) \cdot \altprod \cdot (\widetilde{\bm z} \widetilde{\bm z}^\UT-\widetilde{\bm Z}^2))  + 2 \widetilde{\bm z}^\UT\altprod \cdot  \widetilde{\bm Z}^2 \cdot \altprod \cdot \widetilde{\bm z} - \Tr(\altprod \cdot \widetilde{\bm Z}^2  \cdot \altprod \cdot \widetilde{\bm Z}^2).
    \end{align*}
    Next we note that:
    \begin{align*}
         |\widetilde{\bm z}^\UT\altprod \cdot  \widetilde{\bm Z}^2 \cdot \altprod \cdot \widetilde{\bm z}| & \leq \|\widetilde{\bm z}\|^2 \cdot \|\altprod\|_\op^2 \cdot \left( \max_{i \in [m]} |\widetilde{z}_i|^2 \right) \leq O_P(m) \cdot O(1) \cdot O_P(\polylog(m)),
    \end{align*}
    Hence it can be shown that,
    \begin{align*}
        \frac{\E |\widetilde{\bm z}^\UT\altprod \cdot  \widetilde{\bm Z}^2 \cdot \altprod \cdot \widetilde{\bm z}|}{m^2} \rightarrow 0.
    \end{align*}
    Similarly,
    \begin{align*}
        |\Tr(\altprod \cdot \widetilde{\bm Z}^2  \cdot \altprod \cdot \widetilde{\bm Z}^2)| & \leq m \|\altprod \cdot \widetilde{\bm Z}^2  \cdot \altprod \cdot \widetilde{\bm Z}^2\|_\op \leq m\|\altprod\|_\op^2 \cdot \left( \max_{i \in [m]} |\widetilde{z}_i|^4 \right)\\ &\leq O(m) \cdot O(1) \cdot O_P(\polylog(m)),
    \end{align*}
    and hence one expects that,
    \begin{align*}
        \frac{\E|\Tr(\altprod \cdot \widetilde{\bm Z}^2  \cdot \altprod \cdot \widetilde{\bm Z}^2)|}{m^2} \rightarrow 0.
    \end{align*}
    We omit the detailed arguments. This concludes the proof of the lemma.
\end{proof}
Note that, so far, we have shown that:
\begin{align*}
       \lim_{m \rightarrow \infty} \frac{\E (\bm z^\UT \altprod(\bm \Psi, \bm Z) \bm z)^2}{m^2}&=\lim_{m \rightarrow \infty}\frac{\E \; \Tr (\altprod \cdot (\widetilde{\bm z} \widetilde{\bm z}^\UT -  \widetilde{\bm Z}^2) \cdot \altprod \cdot (\widetilde{\bm z} \widetilde{\bm z}^\UT -  \widetilde{\bm Z}^2)  ) \Indicator{\mathcal{E}}}{m^2},
\end{align*}
provided the latter limit exists. We now focus on analyzing the RHS. We expand 
    \begin{align*}
        & \; \Tr (\altprod \cdot (\widetilde{\bm z} \widetilde{\bm z}^\UT -  \widetilde{\bm Z}^2) \cdot \altprod \cdot (\widetilde{\bm z} \widetilde{\bm z}^\UT -  \widetilde{\bm Z}^2)  ) = \\& \sum_{\substack{a_{1:2k+2} \in [m]\\ a_1 \neq a_{2k+2}\\ a_{k+1} \neq a_{k+2}}}  (\bm \Psi)_{a_1,a_2} q_1(\widetilde{z}_{a_2}) \cdots  (\bm \Psi)_{a_k,a_{k+1}} \widetilde{z}_{a_{k+1}} \widetilde{z}_{a_{k+2}} (\bm \Psi)_{a_{k+2},a_{k+3}} q_1(\widetilde{z}_{a_{k+3}})  \cdots   (\bm \Psi)_{a_{2k+1},a_{2k+2}} \widetilde{z}_{a_{2k+2}} \widetilde{z}_{a_1}.
    \end{align*}
    This can be written compactly in terms of matrix moments (Definition \ref{def: matrix moment}) as follows: Let $\dlgraph \in \weightedG{2k+2}$ denote the graph formed by combining two disconnected copies of the simple line graph on vertices $[1:k+1]$ and $[k+2: 2k +2]$:
    \begin{align*}
        (\dlgraph)_{ij} & = \begin{cases} 1 : & |i-j| = 1, \; \{i,j\} \neq \{k+1,k+2\}, \\ 0 : &\text{otherwise} \end{cases}.
    \end{align*}
    Recall the notation for partitions introduced in Section \ref{section: partition_notation}. Observe that:
    \begin{align*}
        \{(a_1 \dots a_{2k+2}) \in [m]^{2k+2} : \; a_1 \neq a_{2k+2}, \; a_{k+1} \neq a_{k+2} \} & = \bigsqcup_{\pi \in \part{0}{[2k+2]}}  \cset{\pi},
    \end{align*}
    where,
    \begin{align*}
        \part{0}{[2k+2]} &\explain{def}{=} \{\pi \in \part{}{2k+2}: \pi(1) \neq \pi(2k+2), \; \pi(k+1) \neq \pi(k+2)\}.
    \end{align*}
    Recalling Definition \ref{def: matrix moment}, we have,
    \begin{align*}
        (\bm \Psi)_{a_1,a_2} \cdots (\bm \Psi)_{a_k,a_{k+1}} (\bm \Psi)_{a_{k+2},a_{k+3}}   \cdots   (\bm \Psi)_{a_{2k+1},a_{2k+2}} & = \matmom{\bm \Psi}{\dlgraph}{\pi}{\bm a}
    \end{align*}
    Hence,
    \begin{align*}
        &\frac{\E \; \Tr (\altprod \cdot (\widetilde{\bm z} \widetilde{\bm z}^\UT -  \widetilde{\bm Z}^2) \cdot \altprod \cdot (\widetilde{\bm z} \widetilde{\bm z}^\UT -  \widetilde{\bm Z}^2)  ) \Indicator{\mathcal{E}}}{m^2} = \\& \hspace{0cm} \frac{1}{m^2} \sum_{\substack{\pi \in \part{0}{2k+2}\\ \bm a \in \cset{\pi}}} \E \; \matmom{\bm \Psi}{\dlgraph}{\pi}{\bm a}  \cdot(\widetilde{z}_{a_1} q_1(\widetilde{z}_{a_2}) \cdots    \widetilde{z}_{a_{k+1}} \widetilde{z}_{a_{k+2}}  q_1(\widetilde{z}_{a_{k+3}})  \cdots   \widetilde{z}_{a_{2k+2}}) \cdot  \Indicator{\mathcal{E}}.
    \end{align*}
    By the tower property,
 \begin{align*}
        &\E \; \matmom{\bm \Psi}{\dlgraph}{\pi}{\bm a}  \cdot(\widetilde{z}_{a_1} q_1(\widetilde{z}_{a_2}) \cdots    \widetilde{z}_{a_{k+1}} \widetilde{z}_{a_{k+2}}  q_1(\widetilde{z}_{a_{k+3}})  \cdots   \widetilde{z}_{a_{2k+2}}) \cdot  \Indicator{\mathcal{E}} =  \\&\hspace{1cm}\E \left[ \matmom{\bm \Psi}{\dlgraph}{\pi}{\bm a}  \cdot\E[ \widetilde{z}_{a_1} q_1(\widetilde{z}_{a_2}) \cdots    \widetilde{z}_{a_{k+1}} \widetilde{z}_{a_{k+2}}  q_1(\widetilde{z}_{a_{k+3}})  \cdots   \widetilde{z}_{a_{2k+2}} |  \bm A ] \Indicator{\mathcal{E}} \right].
    \end{align*}
    We will now use Mehler's formula (Proposition \ref{proposition: mehler}) to evaluate $\E[ \cdots |  \bm A ]$ upto leading order. Note that some of the random variables $\widetilde{z}_{a_{1:2k+2}}$ are equal (as given by the partition $\pi$). Hence we group them together and recenter the resulting functions. The blocks corresponding to $a_1, a_{k+1},a_{k+2},a_{2k+2}$ need to be treated specially due to the presence of $\widetilde{z}_{a_1},\widetilde{z}_{a_{k+1}},\widetilde{z}_{a_{k+2}}, \widetilde{z}_{a_{2k+2}}$ in the above expectations. Hence, we introduce the following notations:
     We introduce the following notations:
    \begin{align*}
        \fblknum{1}{\pi} & = \pi(1), \; \lblknum{1}{\pi}= \pi(k+1), \; \fblknum{2}{\pi}= \pi(k+2), \; \lblknum{2}{\pi} = \pi(2k+2)  \\ \singleblks{\pi} &= \{i \in [1:2k+2]\backslash \{1,k+1,k+2,2k+2\}: |\pi(i)| = 1 \}.
    \end{align*}
    We label all the remaining blocks of $\pi$ as $\blocks_1,\blocks_2 \dots \blocks_{|\pi| - |\singleblks{\pi}| - 4}$. Hence the partition $\pi$ is given by:
    \begin{align*}
        \pi = \fblknum{1}{\pi} \sqcup \lblknum{1}{\pi} \sqcup \fblknum{2}{\pi} \sqcup \lblknum{2}{\pi} \sqcup \left( \bigsqcup_{i \in \singleblks{\pi}} \{i\} \right) \sqcup \left( \bigsqcup_{t=1}^{|\pi| - |\singleblks{\pi}| - 4} \blocks_i\right).
    \end{align*}
    To simplify notation, we additionally define:
    \begin{align*}
        q_{k+1 + i}(\xi) \explain{def}{=} q_{i}(\xi), \; i = 1,2 \dots k-1.
    \end{align*}
    Note that: 
    \begin{align*}
        &\widetilde{z}_{a_1} \widetilde{z}_{a_{k+1}}\widetilde{z}_{a_{k+2}}\widetilde{z}_{a_{2k+2}}\prod_{\substack{i=1\\ i \neq k,k+1}}^{2k} q_i(\widetilde{z}_{a_{i+1}})  = \\& Q_\ffncnum{1}(\widetilde{z}_{a_1})  Q_\lfncnum{1}(\widetilde{z}_{a_{k+1}})  Q_\ffncnum{2}(\widetilde{z}_{a_{k+2}})   Q_\lfncnum{2}(\widetilde{z}_{a_{2k+2}}) \left( \prod_{i \in \singleblks{\pi}} q_{i-1}(\widetilde{z}_{a_i}) \right)  \prod_{i=1}^{|\pi| - |\singleblks{\pi}| - 4} (Q_{\blocks_i}(z_{a_{\blocks_i}}) + \mu_{\blocks_i}),
    \end{align*}
    where,
    \begin{align*}
        Q_\ffncnum{1}(\xi) & = \xi \cdot \prod_{i \in \fblknum{1}{\pi}, i \neq 1} q_{i-1}(\xi), \\
        Q_\lfncnum{1}(\xi) & = \xi \cdot \prod_{i \in \lblknum{1}{\pi}, i \neq k+1} q_{i-1}(\xi), \\
        Q_\ffncnum{2}(\xi) & = \xi \cdot \prod_{i \in \fblknum{2}{\pi}, i \neq k+2} q_{i-1}(\xi), \\
        Q_\lfncnum{2}(\xi) & = \xi \cdot \prod_{i \in \lblknum{2}{\pi}, i \neq 2k+2} q_{i-1}(\xi), \\
        \mu_{\blocks_i} & = \E_{\xi \sim \gauss{0}{1}} \left[ \prod_{j \in \blocks_i} q_{j-1}(\xi) \right], \\
        Q_{\blocks_i}(\xi) & = \prod_{j \in \blocks_i} q_{j-1}(\xi) - \mu_{\blocks_i}, \\
    \end{align*}
     With this notation in place we can apply Mehler's formula. The result is summarized in the following lemma.
     \begin{lem} \label{lemma: qf_mehler_conclusion_mom2} For any $\pi \in \part{0}{[2k+2]}$ and any $\bm a \in \cset{\pi}$ we have,
\begin{subequations}
    \begin{align*}
        &\Indicator{\mathcal{E}}  \left| \E[ \widetilde{z}_{a_1} q_1(\widetilde{z}_{a_2}) \cdots    \widetilde{z}_{a_{k+1}} \widetilde{z}_{a_{k+2}}  q_1(\widetilde{z}_{a_{k+3}})  \cdots  \widetilde{z}_{a_{2k+2}} |  \bm A ] - \sum_{\bm w \in \weightedGnum{\pi}{2}} {\Coeff}(\bm w, \pi) \cdot \matmom{\bm \Psi}{\bm w}{\pi}{\bm a}  \right| \nonumber \\
             &\hspace{8cm }  \leq C(\altprod) \cdot \left( \frac{\log^3(m)}{m \kappa^2} \right)^{\frac{3 + |\singleblks{\pi}|}{2}},
    \end{align*}
    where, $\matmom{\bm \Psi}{\bm w}{\pi}{\bm a}$ is the matrix moment as defined in Definition \ref{def: matrix moment}, 
    \begin{align*}
        {\Coeff}(\bm w, \pi) & = \frac{1}{{\kappa^{\|\bm w\|}} \bm w!}   \left(\hat{Q}_\ffncnum{1}(1) \hat{Q}_\lfncnum{1}(1) \hat{Q}_\ffncnum{2}(1)\hat{Q}_\lfncnum{2}(1) \prod_{i \in \singleblks{\pi}} \hat{q}_{i-1}(2)   \right)   \left( \prod_{i \in [|\pi| - |\singleblks{\pi}| - 4]} \mu_{\blocks_i} \right)
    \end{align*}
    \begin{align*}
            \weightedGnum{\pi}{2} & \explain{def}{=} \left\{\bm w \in \weightedG{2k+2}: \degree_i(\bm w) = 1 \; \forall \; i \; \in \; \{1,k+1,k+2,2k+2\}, \right. \nonumber\\ &\hspace{1cm} \degree_i(\bm w) = 2 \; \forall \; i \; \in \; \singleblks{\pi},  \left.\degree_i(\bm w) = 0 \; \forall \; i \; \notin \; \{1,k+1,k+2,2k+2\} \cup \singleblks{\pi} \right\},
        \end{align*}
    %\label{eq: mehler_conclusion_imp}
    \end{subequations}
     \end{lem}
The proof of the lemma involves instantiating Mehler's formula for this situation and identifying the leading order term. Since the proof is analogous to the proof of Lemma \ref{lemma: qf_mehler_conclusion} provided in Appendix \ref{proof: qf_mehler_conclusion}, we omit it.

We return to our analysis of:
\begin{align*}
        &\frac{\E \; \Tr (\altprod \cdot (\widetilde{\bm z} \widetilde{\bm z}^\UT -  \widetilde{\bm Z}^2) \cdot \altprod \cdot (\widetilde{\bm z} \widetilde{\bm z}^\UT -  \widetilde{\bm Z}^2)  ) \Indicator{\mathcal{E}}}{m^2} = \\& \hspace{0cm} \frac{1}{m^2} \sum_{\substack{\pi \in \part{0}{2k+2}\\\bm a \in \cset{\pi}}} \E \; \matmom{\bm \Psi}{\dlgraph}{\pi}{\bm a}  \cdot(\widetilde{z}_{a_1} q_1(\widetilde{z}_{a_2}) \cdots    \widetilde{z}_{a_{k+1}} \widetilde{z}_{a_{k+2}}  q_1(\widetilde{z}_{a_{k+3}})  \cdots   \widetilde{z}_{a_{2k+2}}) \cdot  \Indicator{\mathcal{E}}.
\end{align*}
We define the following  subsets of $\part{0}{2k+2}$ as:
\begin{subequations}
\label{eq: qf_firstmom_goodpartitions_mom2}
\begin{align}
    \part{1}{[2k+2]} & \explain{def}{=} \left\{\pi \in \part{0}{2k+2
    }: \; |\pi(i)| = 1, \; \forall \; i \; \in \; \{1,k+1,k+2,2k+2\},\right. \\ \nonumber&\left. \hspace{8cm} |\pi(j)| \leq 2 \; \forall \; j \; \in \; [k+1]\right\}, \\
    \part{2}{[2k+2]} & \explain{def}{=} \part{0}{[2k+2]} \backslash \part{1}{[2k+2]},
\end{align}
\end{subequations}
and the error term which was controlled in Lemma \ref{lemma: qf_mehler_conclusion}:
\begin{align*}
        &\epsilon(\bm \Psi, \bm a) \explain{def}{=}\\& \Indicator{\mathcal{E}}  \left( \E[ \widetilde{z}_{a_1} q_1(\widetilde{z}_{a_2}) \cdots    \widetilde{z}_{a_{k+1}} \widetilde{z}_{a_{k+2}}  q_1(\widetilde{z}_{a_{k+3}})  \cdots   \widetilde{z}_{a_{2k+2}} |  \bm A ] - \sum_{\bm w \in \weightedGnum{\pi}{2}} {\Coeff}(\bm w, \pi) \cdot \matmom{\bm \Psi}{\bm w}{\pi}{\bm a}  \right)\end{align*}.

With these definitions we consider the decomposition:
\begin{align*}
    &\frac{\E \; \Tr (\altprod \cdot (\widetilde{\bm z} \widetilde{\bm z}^\UT -  \widetilde{\bm Z}^2) \cdot \altprod \cdot (\widetilde{\bm z} \widetilde{\bm z}^\UT -  \widetilde{\bm Z}^2)  ) \Indicator{\mathcal{E}}}{m^2} = \\& \hspace{1cm} \frac{1}{m^2} \sum_{\pi \in \part{1}{[2k+2]}} \sum_{a \in \cset{\pi}} \sum_{\bm w \in \weightedGnum{\pi}{2}} {\Coeff}(\bm w, \pi) \cdot \E\left[\matmom{\bm \Psi}{\bm w + \dlgraph}{\pi}{\bm a} \right] - \mathsf{I} + \mathsf{II} + \mathsf{III},
\end{align*}
where, 
\begin{align*}
    \mathsf{I} &\explain{}{=}\frac{1}{m^2} \sum_{\substack{\pi \in \part{0}{[2k+2]}}}\sum_{a \in \cset{\pi}} \sum_{\bm w \in \weightedGnum{\pi}{2}} {\Coeff}(\bm w, \pi) \cdot \E\left[   \matmom{\bm \Psi}{\bm w + \dlgraph}{\pi}{\bm a} \Indicator{\mathcal{E}^c} \right], \\
    \mathsf{II} & = \frac{1}{m^2} \sum_{\pi \in \part{0}{2k+2]}} \sum_{a \in \cset{\pi}} \E\left[   \matmom{\bm \Psi}{\dlgraph}{\pi}{\bm a}\epsilon(\bm \Psi, \bm a) \Indicator{\mathcal{E}} \right], \\
     \mathsf{III} & \explain{}{=} \frac{1}{m^2} \sum_{\pi \in \part{2}{[2k+2]}} \sum_{a \in \cset{\pi}} \sum_{\bm w \in \weightedGnum{\pi}{2}} {\Coeff}(\bm w, \pi) \cdot \E\left[   \matmom{\bm \Psi}{\bm w + \dlgraph}{\pi}{\bm a} \right].
\end{align*}
We will show that $\mathsf{I}, \mathsf{II}, \mathsf{III} \rightarrow 0$. Showing this involves the following components:
\begin{enumerate}
    \item Bounds on matrix moments $\E\left[   \matmom{\bm \Psi}{\bm w + \dlgraph}{\pi}{\bm a} \right]$ which have been developed in Lemma \ref{lemma: matrix_moment_ub}. 
    \item Controlling the size of the set $|\cset{\pi}|$ (since we sum over $\bm a \in \cset{\pi}$ in the above terms). Since, 
    \begin{align*}
        |\cset{\pi}| & = m(m-1) \cdots (m-|\pi| + 1) \asymp m^{|\pi|},
    \end{align*}
    we need to develop bounds on $|\pi|$. This is done in the following lemma. In contrast, the sums over $\pi \in \part{0}{[2k+2]}$ and $\bm w \in \weightedGnum{\pi}{1}$ are not a cause of concern since $|\part{0}{[2k+2]}|,|\weightedGnum{\pi}{1}|$ depend only on $k$ (which is held fixed) and not on $m$. 
\end{enumerate}
\begin{lem} \label{lemma: qf_cardinality_bounds_mom2} For any $\pi \in \part{1}{[2k+2]}$ we have,
\begin{align*}
    |\pi| & = \frac{2k+6+|\singleblks{\pi}| }{2} \implies |\cset{\pi}| \leq m^{\frac{2k+6+|\singleblks{\pi}| }{2}}. 
\end{align*}
For any $\pi \in \part{2}{[2k+2]}$, we have,
\begin{align*}
    |\pi| & \leq \frac{2k+5+|\singleblks{\pi}| }{2} \implies |\cset{\pi}| \leq m^{\frac{2k+5+|\singleblks{\pi}| }{2}}. 
\end{align*}
\end{lem}
\begin{proof}
Consider any $\pi \in \part{0}{[2k+2]}$.  Recall that the disjoint blocks of $|\pi|$ were given by:
\begin{align*}
        \pi = \fblknum{1}{\pi} \sqcup \lblknum{1}{\pi} \sqcup \fblknum{2}{\pi} \sqcup \lblknum{2}{\pi} \sqcup \left( \bigsqcup_{i \in \singleblks{\pi}} \{i\} \right) \sqcup \left( \bigsqcup_{t=1}^{|\pi| - |\singleblks{\pi}| - 4} \blocks_i\right).
    \end{align*}

Hence,
\begin{align*}
    2k+2 & = |\fblknum{1}{\pi}|+|\fblknum{2}{\pi}| + |\lblknum{1}{\pi}| + |\lblknum{2}{\pi}| + |\singleblks{\pi}| + \sum_{t=1}^{|\pi| - |\singleblks{\pi}| - 4} | \blocks_i|.
\end{align*}
Note that:
\begin{subequations}
\label{eq: qf_firstmom_blocksizeUB_observe_mom2}
\begin{align}
    |\fblknum{1}{\pi}| \geq 1 & \qquad \text{ (Since $1 \in \fblknum{1}{\pi}$)} \\
    |\fblknum{2}{\pi}| \geq 1 & \qquad \text{ (Since $k+2 \in \fblknum{2}{\pi}$)} \\
    |\lblknum{1}{\pi}| \geq 1 & \qquad \text{ (Since $k+1 \in \lblknum{1}{\pi}$)} \\
    |\lblknum{2}{\pi}| \geq 1 & \qquad \text{ (Since $2k+2 \in \lblknum{1}{\pi}$)} \\
    |\blocks_i| \geq 2 & \qquad \text{ (Since $\blocks_i$ are not singletons)}.
\end{align}\end{subequations}
Hence,
\begin{align*}
    2k+2 & \geq 4 +  2 |\pi| - |\singleblks{\pi}| - 8,
\end{align*}
which implies,
\begin{align}
     |\pi| &\leq  \frac{2k+6+|\singleblks{\pi}| }{2} \label{eq: qf_firstmom_blocksizeUB_mom2}, 
\end{align}
and hence,
\begin{align*}
    |\cset{\pi}| & \leq m^{|\pi|} \leq m^{\frac{2k+6+|\singleblks{\pi}| }{2}}.
\end{align*}
Finally observe that:
\begin{enumerate}
    \item For any $\pi \in \part{2}{[2k+2]}$ each of the inequalities in \eqref{eq: qf_firstmom_blocksizeUB_observe_mom2} are exactly tight by the definition of $\part{1}{[k+1]}$ in \eqref{eq: qf_firstmom_goodpartitions_mom2}, and hence,
    \begin{align*}
        |\pi| & = \frac{2k+6+|\singleblks{\pi}| }{2}.
    \end{align*}
    \item For any $\pi \in \part{2}{[2k+2]}$, one of the inequalities in \eqref{eq: qf_firstmom_blocksizeUB_observe_mom2} must be strict (see \eqref{eq: qf_firstmom_goodpartitions_mom2}). Hence, when $\pi \in \part{2}{[k+1]}$ we have the improved bound:
\begin{align*}
    |\pi| & \leq \frac{2k+5+|\singleblks{\pi}| }{2}.
\end{align*}
\end{enumerate}
This proves the claims of the lemma. 
\end{proof}
We will now show that $\mathsf{I}, \mathsf{II}, \mathsf{III} \rightarrow 0$.
\begin{lem} We have,
\begin{align*}
    \mathsf{I} \rightarrow 0, \; \mathsf{II} \rightarrow 0, \; \mathsf{III} \rightarrow 0 \;\text{ as $m \rightarrow \infty$},
\end{align*}
and hence,
\begin{align*}
    &\lim_{m \rightarrow \infty} \frac{\E (\bm z^\UT \altprod(\bm \Psi, \bm Z) \bm z)^2}{m^2}   = \\&\hspace{1cm} \lim_{m \rightarrow \infty} \frac{1}{m^2} \sum_{\pi \in \part{1}{[2k+2]}} \sum_{a \in \cset{\pi}} \sum_{\bm w \in \weightedGnum{\pi}{2}} {\Coeff}(\bm w, \pi) \cdot \E\left[\matmom{\bm \Psi}{\bm w + \dlgraph}{\pi}{\bm a} \right],
\end{align*}
provided the latter limit exists. 
\end{lem}
\begin{proof}
First note that for any $\bm w \in \weightedGnum{\pi}{1}$, we have,
\begin{align*}
    \|\bm w\| = \frac{1}{2} \sum_{i=1}^{2k+2} \degree_i(\bm w) = \frac{1+1 +1+1+ 2 |\singleblks{\pi}|}{2} = 2+ |\singleblks{\pi}| \; \text{ (See Lemma \ref{lemma: qf_mehler_conclusion_mom2})}.
\end{align*}
Furthermore recalling the definition of $\dlgraph$, $\|\dlgraph\| = 2k$.
Now we apply Lemma \ref{lemma: matrix_moment_ub} to obtain:
\begin{align*}
    |\E\left[   \matmom{\bm \Psi}{\bm w + \dlgraph}{\pi}{\bm a} \Indicator{\mathcal{E}^c} \right]|& \leq \sqrt{\E\left[   \matmom{\bm \Psi}{2\bm w + 2\dlgraph}{\pi}{\bm a} \right] } \sqrt{\P(\mathcal{E}^c)} \\
    &  \explain{}{\leq} \left( \frac{C_k \log^2(m)}{m} \right)^{\frac{|\singleblks{\pi}| + 2 + 2k}{2}} \cdot \sqrt{\P(\mathcal{E}^c)}, \\
    &\explain{(a)}{\leq} \left( \frac{C_k \log^2(m)}{m} \right)^{\frac{|\singleblks{\pi}| + 2 + 2k}{2}} \cdot \frac{C_k}{m}. \\
    \E|\matmom{\bm \Psi}{\dlgraph}{\pi}{\bm a}| &  \leq  \left( \frac{C_k \log^2(m)}{m} \right)^{k},\\
    \E\left[|\matmom{\bm \Psi}{\bm w + \bm{\ell}_{k+1}}{\pi}{\bm a}| \right] & \explain{}{\leq}  \left( \frac{C_k \log^2(m)}{m} \right)^{\frac{|\singleblks{\pi}| + 2 + 2k}{2}}
\end{align*}
In the step marked (a) we used Lemma \ref{lemma: trace_good_event}.
Further recall that by Lemma \ref{lemma: qf_mehler_conclusion} we have,
\begin{align*}
    |\epsilon(\bm \Psi, \bm a)| & \leq C(\altprod) \cdot \left( \frac{\log^3(m)}{m \kappa^2} \right)^{\frac{3 + |\singleblks{\pi}|}{2}}.
\end{align*}
Using these estimates, we obtain,
\begin{align*}
    |\mathsf{I}| & \leq  \frac{C(\altprod) \cdot }{m^2}  \cdot \sum_{\substack{\pi: \part{0}{[2k+2]}} }|\cset{\pi}| \cdot   \left( \frac{C_k \log^2(m)}{m}  \right)^{\frac{|\singleblks{\pi}| + 2 + 2k}{2}} \cdot\frac{C_k}{m}  \\& \leq   \frac{C(\altprod) \cdot }{m^2}  \cdot \sum_{\substack{\pi: \part{0}{[2k+2]}} }m^{\frac{2k+6+|\singleblks{\pi}|}{2}} \cdot   \left( \frac{C_k \log^2(m)}{m} \right)^{\frac{|\singleblks{\pi}| + 2 + 2k}{2}} \cdot \frac{C_k}{m}\\
    & = O \left( \frac{\polylog(m)}{m} \right)\\
    |\mathsf{II}| & \leq \frac{C(\altprod)}{m^2} \cdot \left( \frac{C_k \log^2(m)}{m} \right)^{k} \cdot  \sum_{\substack{\pi: \part{0}{[2k+2]}}}  |\cset{\pi}|  \cdot \left( \frac{\log^3(m)}{m \kappa^2} \right)^{\frac{3 + |\singleblks{\pi}|}{2}} \\&\leq \frac{C(\altprod)}{m^2} \cdot \left( \frac{C_k \log^2(m)}{m} \right)^{{k}} \cdot  \sum_{\substack{\pi: \part{0}{[2k+2]} } }  m^{\frac{2k+6+|\singleblks{\pi}|}{2}}  \cdot \left( \frac{\log^3(m)}{m \kappa^2} \right)^{\frac{3 + |\singleblks{\pi}|}{2}} \\
     & = O \left( \frac{\polylog(m)}{\sqrt{m}} \right)\\
    |\mathsf{III}| & \leq \frac{C(\altprod) \cdot }{m^2}  \cdot \sum_{\pi: \part{2}{[2k+2]}  }|\cset{\pi}| \cdot   \left( \frac{C_k \log^2(m)}{m} \right)^{\frac{|\singleblks{\pi}| + 1 + k}{2}} \\
    & \leq \frac{C(\altprod) \cdot }{m^2}  \cdot \sum_{\pi: \part{2}{[2k+2]}  } m^{\frac{2k+5+|\singleblks{\pi}|}{2}} \cdot   \left( \frac{C_k \log^2(m)}{m} \right)^{\frac{|\singleblks{\pi}| + 2 + 2k}{2}} \\
    &=O \left( \frac{\polylog(m)}{\sqrt{m}} \right).
\end{align*}
This concludes the proof of this lemma.
\end{proof}

Next, we consider the decomposition:
\begin{align*}
    & \frac{1}{m^2} \sum_{\pi \in \part{1}{[2k+2]}} \sum_{a \in \cset{\pi}} \sum_{\bm w \in \weightedGnum{\pi}{2}} {\Coeff}(\bm w, \pi) \cdot \E\left[\matmom{\bm \Psi}{\bm w + \dlgraph}{\pi}{\bm a} \right] = \\& \hspace{0cm} \frac{1}{m^2} \sum_{\pi \in \part{1}{[2k+2]}} \sum_{\substack{\bm w \in \weightedGnum{\pi}{2}\\ \bm w + \dlgraph  \in \weightedGnum{\pi}{DA}}} \sum_{a \in \labelling{CF}(\bm w + \dlgraph, \pi)}  {\Coeff}(\bm w, \pi) \cdot \E\left[   \matmom{\bm \Psi}{\bm w + \dlgraph}{\pi}{\bm a} \right] + \mathsf{IV} + \mathsf{V},
\end{align*}
where,
\begin{align*}
    \mathsf{IV} &\explain{def}{=}  \frac{1}{m^2} \sum_{\pi \in \part{1}{[2k+2]}} \sum_{a \in \cset{\pi}} \sum_{\substack{\bm w \in \weightedGnum{\pi}{2}\\ \bm w + \dlgraph \notin \weightedGnum{\pi}{DA}}} {\Coeff}(\bm w, \pi) \cdot \E\left[   \matmom{\bm \Psi}{\bm w + \dlgraph}{\pi}{\bm a} \right], \\
    \mathsf{V} &\explain{def}{=} \frac{1}{m^2} \sum_{\pi \in \part{1}{[2k+2]}} \sum_{\substack{\bm w \in \weightedGnum{\pi}{2}\\ \bm w + \dlgraph \in \weightedGnum{\pi}{DA}}} \sum_{a \in \cset{\pi} \backslash \labelling{CF}(\bm w + \dlgraph, \pi)}  {\Coeff}(\bm w, \pi) \cdot \E\left[   \matmom{\bm \Psi}{\bm w + \dlgraph}{\pi}{\bm a} \right].
\end{align*}
\begin{lem}
We have, $\mathsf{IV} \rightarrow 0, \mathsf{V} \rightarrow 0$ as $m \rightarrow \infty$, and hence,
\begin{align*}
    &\lim_{m \rightarrow \infty} \frac{\E (\bm z^\UT \altprod \bm z)^2}{m^2}  =\\& \hspace{0cm} \lim_{m \rightarrow \infty}  \frac{1}{m^2} \sum_{\pi \in \part{1}{[2k+2]}} \sum_{\substack{\bm w \in \weightedGnum{\pi}{2}\\ \bm w + \dlgraph  \in \weightedGnum{\pi}{DA}}} \sum_{a \in \labelling{CF}(\bm w + \dlgraph, \pi)}  {\Coeff}(\bm w, \pi) \cdot \E\left[   \matmom{\bm \Psi}{\bm w + \dlgraph}{\pi}{\bm a} \right],
\end{align*}
provided the latter limit exists. 
\end{lem}
\begin{proof}
We will prove this in two steps.
\begin{description}
\item [Step 1: $\mathsf{IV} \rightarrow 0$. ] We consider the two sensing models separately:
\begin{enumerate}
    \item Subsampled Hadamard Sensing: In this case, Proposition \ref{prop: clt_hadamard} tells us that if $\bm w + \dlgraph \not\in \weightedGnum{\pi}{DA}$, then, $$\E\left[   \matmom{\bm \Psi}{\bm w + \dlgraph}{\pi}{\bm a} \right] = 0$$ and hence $\mathsf{IV} = 0$.
    \item Subsampled Haar Sensing: Observe that, since $\|\bm w\| + \|\dlgraph\|= 2 + |\singleblks{\pi}| + 2k$, we have,
    \begin{align*}
        \E\left[   \matmom{\bm \Psi}{\bm w + \dlgraph}{\pi}{\bm a} \right] & = \frac{\E\left[   \matmom{\sqrt{m}\bm \Psi}{\bm w + \dlgraph}{\pi}{\bm a} \right]}{m^{\frac{2 + |\singleblks{\pi}| + 2k}{2}}}.
    \end{align*}
    By Proposition \ref{prop: clt_random_ortho} we know that,
    \begin{align*}
        \left| \E\left[   \matmom{\sqrt{m}\bm \Psi}{\bm w + \dlgraph}{\pi}{\bm a} \right] - \prod_{\substack{s,t \in [|\pi|] \\ s \leq t}} \E \left[ Z_{st}^{W_{st}(\bm w + \dlgraph, \pi)} \right] \right| & \leq \frac{K_1 \log^{K_2}(m)}{m^\tbd},
    \end{align*}
    $ \; \forall \; m \geq K_3$, where $K_1,K_2,K_3$ are universal constants depending only on $k$. Note that since $\bm w + \dlgraph \notin \weightedGnum{\pi}{DA}$, must have some $s \in [|\pi|]$ such that:
    \begin{align*}
        W_{ss}(\bm w + \dlgraph, \pi) \geq 1.
    \end{align*}
    Recall that, $\degree_i(\bm w) = 0$ for any $i \not\in \{1,k+1,k+2,2k+2\} \cup \singleblks{\pi}$ (since $\bm w \in \weightedGnum{\pi}{2}$) and furthermore, $|\pi(i)| = 1 \forall \; i \; \in \; \{1,k+1,k+2,2k+2\} \cup \singleblks{\pi}$ (since $\pi \in \part{1}{2k+2}$). Hence, we have $\bm w \in \weightedGnum{\pi}{DA}$ and in particular, $W_{ss}(\bm w, \pi) =0$. Consequently, we must have $W_{ss}(\dlgraph, \pi) \geq 1$. Recall the definition of $\dlgraph$, since $W_{ss}(\bm{\ell}_{k+1}, \pi) \geq 1$ we must have that for some $i \in [2k+2]$, we have, $\pi(i) = \pi(i+1) = \blocks_s$. However, since $\pi \in \part{1}{2k+2}$, $|\blocks_s| \leq 2$, and hence $\blocks_s = \{i,i+1\}$. This means that $W_{ss}(\dlgraph,\pi) = 1 =W_{ss}(\bm w + \dlgraph, \pi)$. Consequently since $\E Z_{ss} = 0$, we have,
    \begin{align*}
        \prod_{\substack{s,t \in [|\pi|] \\ s \leq t}} \E \left[ Z_{st}^{W_{st}(\bm w + \dlgraph, \pi)} \right] = 0,
    \end{align*}
    or,
    \begin{align*}
        \left|\E\left[   \matmom{\bm \Psi}{\bm w + \dlgraph}{\pi}{\bm a} \right]\right| & = \frac{\polylog(m)}{m^{\frac{2 + |\singleblks{\pi}| + 2k}{2}+\tbd}}.
    \end{align*}Recalling Lemma \ref{lemma: qf_cardinality_bounds_mom2},
    \begin{align*}
        |\cset{\pi}| & \leq m^{|\pi|} \leq m^{\frac{2k+6+|\singleblks{\pi}| }{2}},
    \end{align*}
    we obtain,
    \begin{align*}
        |\mathsf{IV}| & \leq \frac{C(\altprod)}{m^2} \sum_{\pi \in \part{1}{[2k+2]}} |\cset{\pi}|  \cdot  \frac{\polylog(m)}{m^{\frac{2 + |\singleblks{\pi}| + 2k}{2}+\tbd}} = O\left( \frac{\polylog(m)}{m^\tbd}\right) \rightarrow 0.
    \end{align*}
\end{enumerate}
\item [Step 2: $\mathsf{V} \rightarrow 0$. ]  Using Lemma \ref{lemma: cf_size_bound}, we know that $$|\cset{\pi} \backslash \labelling{CF}(\bm w+\dlgraph, \pi)| \leq O( m^{|\pi|-1})$$
In Lemma \ref{lemma: qf_cardinality_bounds_mom2}, we showed that for any $\pi \in \part{1}{[k+1]}$, 
\begin{align*}
    |\pi| & = \frac{2k+6+|\singleblks{\pi}| }{2}.
\end{align*}
Hence,
\begin{align*}
    |\cset{\pi} \backslash \labelling{CF}(\bm w+\dlgraph, \pi)| &\leq O( m^{\frac{2k+4+|\singleblks{\pi}| }{2}}).
\end{align*}
We already know from Lemma \ref{lemma: matrix_moment_ub} that,
\begin{align*}
    | \E\left[   \matmom{\bm \Psi}{\bm w + \dlgraph}{\pi}{\bm a} \right]| & \leq \left( \frac{C_k \log^2(m)}{m} \right)^{\frac{\|\bm w\| + \|\dlgraph\|}{2}} \explain{}{\leq} \left( \frac{C_k \log^2(m)}{m} \right)^{\frac{|\singleblks{\pi}| + 2 + 2k}{2}},
\end{align*}
This gives us:
\begin{align*}
    |\mathsf{V}| & \leq \frac{C}{m^2} \sum_{\pi \in \part{1}{[2k+2]}} \sum_{\substack{\bm w \in \weightedGnum{\pi}{2}\\ \bm w + \dlgraph \in \weightedGnum{\pi}{DA}}}  |\cset{\pi} \backslash \labelling{CF}(\bm w+\dlgraph, \pi)|  \left( \frac{C_k \log^2(m)}{m} \right)^{\frac{|\singleblks{\pi}| + 2 + 2k}{2}} \\
    & = O \left( \frac{\polylog(m)}{m} \right)
\end{align*}
which goes to zero as claimed.
\end{description}
This concludes the proof of the lemma.
\end{proof}
So far we have shown that:
\begin{align*}
    &\lim_{m \rightarrow \infty} \frac{\E (\bm z^\UT \altprod \bm z)^2}{m^2}   = \\& \lim_{m \rightarrow \infty}  \frac{1}{m^2} \sum_{\pi \in \part{1}{[2k+2]}} \sum_{\substack{\bm w \in \weightedGnum{\pi}{2}\\ \bm w + \dlgraph  \in \weightedGnum{\pi}{DA}}} \sum_{a \in \labelling{CF}(\bm w + \dlgraph, \pi)}  {\Coeff}(\bm w, \pi) \cdot \E\left[   \matmom{\bm \Psi}{\bm w + \dlgraph}{\pi}{\bm a} \right].
\end{align*}
provided the latter limit exists. In the following lemma we explicitly calculate the limit on the RHS and hence show that it exists and is same for the subsampled Haar and subsampled Hadamard sensing models. 
\begin{lem}\label{lemma: mom2_complex_formula} For both the subsampled Haar sensing and Hadamard sensing model, we have,
\begin{align*}
    \lim_{m \rightarrow \infty} \frac{\E (\bm z^\UT \altprod \bm z)^2}{m^2}  & = \sum_{\pi \in \part{1}{[2k+2]}} \sum_{\substack{\bm w \in \weightedGnum{\pi}{2}\\ \bm w + \dlgraph \in \weightedGnum{\pi}{DA}}}   {\Coeff}(\bm w, \pi) \cdot\limmom{\bm w+\dlgraph}{\pi},
\end{align*}
where,
\begin{align*}
    \limmom{\bm w+\dlgraph}{\pi} &\explain{def}{=}  \prod_{\substack{s,t \in [|\pi|] \\ s < t}} \E \left[ Z^{W_{st}(\bm w + \dlgraph, \pi)} \right], \; Z \sim \gauss{0}{\kappa(1-\kappa)}.
\end{align*}
\end{lem}
\begin{proof}
By Propositions \ref{prop: clt_hadamard} (for the subsampled Hadamard model) and \ref{prop: clt_random_ortho} (for the subsampled  Haar model) we know that, if $\bm w + \dlgraph \in \weightedGnum{\pi}{DA}, \; \bm a  \in \labelling{CF}(\bm w+\dlgraph, \pi)$, we have,
\begin{align*}
    \matmom{\sqrt{m}\bm \Psi}{\bm w + \dlgraph}{\pi}{\bm a} & = \limmom{\bm w+\dlgraph}{\pi} + \epsilon(\bm w, \pi, \bm a),
\end{align*}
where
\begin{align*}
     |\epsilon(\bm w, \pi, \bm a)| & \leq \frac{K_1 \log^{K_2}(m)}{m^\tbd}, \; \forall \; m \geq K_3,
\end{align*}
for some constants $K_1,K_2,K_3$ depending only on $k$. Hence, we can consider the decomposition:
\begin{align*}
   \frac{1}{m^2} \sum_{\pi \in \part{1}{[2k+2]}} \sum_{\substack{\bm w \in \weightedGnum{\pi}{2}\\ \bm w + \dlgraph  \in \weightedGnum{\pi}{DA}}} \sum_{a \in \labelling{CF}(\bm w + \dlgraph, \pi)}  {\Coeff}(\bm w, \pi) \cdot & \E\left[   \matmom{\bm \Psi}{\bm w + \dlgraph}{\pi}{\bm a} \right] \\& = \mathsf{VI} + \mathsf{VII},
\end{align*}
where,
\begin{align*}
    \mathsf{VI} & \explain{def}{=} \frac{1}{m^2} \sum_{\pi \in \part{1}{[2k+2]}} \sum_{\substack{\bm w \in \weightedGnum{\pi}{2}\\ \bm w + \dlgraph  \in \weightedGnum{\pi}{DA}}} \sum_{a \in \labelling{CF}(\bm w + \dlgraph, \pi)}  {\Coeff}(\bm w, \pi) \cdot  \frac{\limmom{\bm w+\dlgraph}{\pi}}{m^{\frac{2 + \singleblks{\pi} + 2k}{2}}}, \\
    \mathsf{VII} & \explain{def}{=} \frac{1}{m^2} \sum_{\pi \in \part{1}{[2k+2]}} \sum_{\substack{\bm w \in \weightedGnum{\pi}{2}\\ \bm w + \dlgraph  \in \weightedGnum{\pi}{DA}}} \sum_{a \in \labelling{CF}(\bm w + \dlgraph, \pi)}  {\Coeff}(\bm w, \pi) \cdot   \frac{\epsilon(\bm w, \pi, \bm a)}{m^{\frac{2 + \singleblks{\pi} + 2k}{2}}}
\end{align*}
We can upper bound $|\mathsf{VII}|$ as follows:
\begin{align*}
    |\labelling{CF}(\bm w+\dlgraph, \pi)| & \leq |\cset{\pi}| \explain{}{\leq} m^{\frac{2k+6+|\singleblks{\pi}|}{2}}, \\
|\mathsf{VII}| & \leq \frac{C(\altprod)}{m^2} \cdot C_k \cdot |\labelling{CF}(\bm w+\dlgraph, \pi)| \cdot \frac{1}{m^{\frac{2 + |\singleblks{\pi}| + 2k}{2}}} \cdot  \frac{K_1 \log^{K_2}(m)}{m^\tbd} \\&= O\left( \frac{\polylog(m)}{m^\tbd} \right) \rightarrow 0.
\end{align*}
We can compute:
\begin{align*}
    &\lim_{m \rightarrow \infty} (\mathsf{VI})  = \lim_{m \rightarrow \infty} \frac{1}{m^2} \sum_{\pi \in \part{1}{[2k+2]}} \sum_{\substack{\bm w \in \weightedGnum{\pi}{2}\\ \bm w + \dlgraph  \in \weightedGnum{\pi}{DA}}} \sum_{a \in \labelling{CF}(\bm w + \dlgraph, \pi)}  {\Coeff}(\bm w, \pi) \cdot  \frac{\limmom{\bm w+\dlgraph}{\pi}}{m^{\frac{2 + \singleblks{\pi} + 2k}{2}}} \\
    & = \lim_{m \rightarrow \infty} \frac{1}{m^2} \sum_{\pi \in \part{1}{[2k+2]}} \sum_{\substack{\bm w \in \weightedGnum{\pi}{2}\\ \bm w + \dlgraph  \in \weightedGnum{\pi}{DA}}}   {\Coeff}(\bm w, \pi) \cdot  \frac{\limmom{\bm w+\dlgraph}{\pi}}{m^{\frac{2 + \singleblks{\pi} + 2k}{2}}} \cdot |\labelling{CF}(\bm w+\dlgraph, \pi)| \\
    & =  \sum_{\pi \in \part{1}{[2k+2]}} \sum_{\substack{\bm w \in \weightedGnum{\pi}{2}\\ \bm w + \dlgraph  \in \weightedGnum{\pi}{DA}}}  {\Coeff}(\bm w, \pi) \cdot \limmom{\bm w+\dlgraph}{\pi} \cdot \frac{m^{|\pi|}}{m^{\frac{6 + \singleblks{\pi} + 2k}{2}}} \cdot \frac{|\labelling{CF}(\bm w+\dlgraph, \pi)|}{m^{|\pi|}} \\
    & \explain{(a)}{=} \sum_{\pi \in \part{1}{[2k+2]}} \sum_{\substack{\bm w \in \weightedGnum{\pi}{2}\\ \bm w + \dlgraph  \in \weightedGnum{\pi}{DA}}}  {\Coeff}(\bm w, \pi) \cdot \limmom{\bm w+\dlgraph}{\pi} \cdot  \frac{|\labelling{CF}(\bm w+\dlgraph, \pi)|}{m^{|\pi|}} \\
    & \explain{(b)}{=}  \sum_{\pi \in \part{1}{[2k+2]}} \sum_{\substack{\bm w \in \weightedGnum{\pi}{2}\\ \bm w + \dlgraph  \in \weightedGnum{\pi}{DA}}}  {\Coeff}(\bm w, \pi) \cdot \limmom{\bm w+\dlgraph}{\pi} .
\end{align*}
In the step marked (a) we used the fact that $|\pi| = (6 + |\singleblks{\pi}| +2k)/2$ for any $\pi \in \part{1}{[2k+2]}$ (Lemma \ref{lemma: qf_cardinality_bounds_mom2}) and in step (b) we used Lemma \ref{lemma: cf_size_bound} ($|\labelling{CF}(\bm w+\dlgraph, \pi)|/m^{|\pi|} \rightarrow 1$).
This proves the claim of the lemma and Proposition \ref{prop: qf_univ_mom2}.
\end{proof}
We can actually significantly simply the combinatorial sum obtained in Lemma \ref{lemma: mom2_complex_formula} which we do so in the following lemma.

\begin{lem}For both the subsampled Haar sensing and Hadamard sensing models, we have,
\begin{align*}
     \lim_{m \rightarrow \infty} \frac{\E (\bm z^\UT \altprod \bm z)^2}{m^2}& =  (1-\kappa)^{2k} \cdot \prod_{i=1}^{k-1} \hat{q}^2_i(2).
\end{align*}
In particular, Proposition \ref{prop: qf_univ_mom2} holds. 
\end{lem}
\begin{proof}
We claim that the only partition with a non-zero contribution is:
\begin{align*}
    \pi & = \bigsqcup_{i=1}^{2k+2}\{i\}.
\end{align*}
In order to see this suppose $\pi$ is not entirely composed of singleton blocks. Define:
\begin{align*}
    i_\star & \explain{def}{=}\min  \{i \in [2k+2]: |\pi(i)| > 1  \}.
\end{align*}
Note $i_\star > 1$ since we know that $|\pi(1)| = |\fblknum{1}{\pi}| = 1$ for any $\pi \in \part{1}{2k+2}$.
Since $\pi \in \part{1}{[2k+2]}$ we must have $|\pi(i_\star)| = 2$, hence denote:
\begin{align*}
    \pi(i_\star)= \{i_\star, j_\star\}.
\end{align*}
for some $j_\star > i_\star+1$ ($i_\star \leq j_\star$ since it is the first index which is not in a singleton block, and $j_\star \neq i_\star + 1$ since otherwise $\bm w + \dlgraph$ will not be disassortative. Similarly we know that $i_\star,j_\star \neq k+1,k+2,2k+2$ because $|\pi(k+1)| = |\pi(k+2)| = |\pi(2k+2)| = 1$ since $\pi \in \part{1}{[2k+2]}$.
Let us label the first few blocks of $\pi$ as:
\begin{align*}
    \blocks_1 = \{1\}, \; \blocks_2 = \{2\}, \dots \blocks_{i_\star - 1} = \{i_\star - 1\}, \; \blocks_{i_\star} = \{i_\star, j_\star\}.
\end{align*}
Next we compute:
\begin{align*}
    W_{i_\star-1,i_\star}(\bm w + \dlgraph, \pi) & = W_{i_\star-1,i_\star}(\dlgraph, \pi) + W_{i_\star-1,i_\star}(\bm w , \pi) \\
    & \explain{(a)}{=} W_{i_\star-1,i_\star}(\dlgraph, \pi) \\
    & \explain{(b)}{=} \mathbf{1}_{i_{\star} - 1 \in \blocks_{i_\star-1}}  + \mathbf{1}_{i_{\star} + 1 \in \blocks_{i_\star-1}} + \mathbf{1}_{j_{\star} - 1 \in \blocks_{i_\star-1}}  + \mathbf{1}_{j_{\star} + 1 \in \blocks_{i_\star-1}} \\
    & \explain{(c)}{=} \mathbf{1}_{i_{\star} - 1 = {i_\star-1}}  + \mathbf{1}_{i_{\star} + 1 ={i_\star-1}} + \mathbf{1}_{j_{\star} - 1 ={i_\star-1}}  + \mathbf{1}_{j_{\star} + 1 = {i_\star-1}} \\
    & \explain{(d)}{=} 1.
\end{align*}
In the step marked (a), we used the fact that since $\bm w \in \weightedGnum{\pi}{2}$ and $|\pi(i_\star)| = |\pi(j_\star)| = 2$, we must have $d_{i_\star}(\bm w) = d_{j_\star}(\bm w) = 0$ and $W_{i_\star-1,i_\star}(\bm w, \pi) = 0$. In the step marked (b) we used the definition of $\dlgraph$. In the step marked (c) we used the fact that $\blocks_{i_\star-1} = \{i_{\star-1}\}$. In the step marked (d) we used the fact that $j_\star > i_\star + 1$.

Hence we have shown that for any $\pi \neq \sqcup_{i=1}^{2k+2} \{i\}$, we have 
\begin{align*}
    \mu(\bm w, \pi) = 0 \; \forall \; \bm w \text{ such that}  \; \bm w \in \weightedGnum{\pi}{2}, \; \bm w + \dlgraph \in \weightedGnum{\pi}{DA}.
\end{align*}
Next, let $\pi = \sqcup_{i=1}^{2k+2} \{i\}$. We observe for any $\bm w$ such that $\bm w \in \weightedGnum{\pi}{2}, \; \bm w + \dlgraph \in \weightedGnum{\pi}{DA}$, we have,
\begin{align*}
    \limmom{\bm w+ \dlgraph}{\pi} &\explain{}{=}  \prod_{\substack{s,t \in [|\pi|] \\ s < t}} \E \left[ Z^{W_{st}(\bm w + \dlgraph, \pi)} \right], \; Z \sim \gauss{0}{\kappa(1-\kappa)} \\
    & = \prod_{\substack{i,j \in [2k+2] \\ i < j}} \E \left[ Z^{w_{ij} + ({\ell}_{k+1})_{ij}, \pi)} \right], \; Z \sim \gauss{0}{\kappa(1-\kappa)}
\end{align*}
Note that since $\E Z = 0$, for $\limmom{\bm w + \dlgraph}{\pi} \neq 0$ we must have:
\begin{align*}
    w_{ij} \geq (\dlgraph)_{ij}, \; \forall \; i,j \; \in \; [2k+2].
\end{align*}
However since $\bm w \in \weightedGnum{\pi}{2}$ we have,
\begin{align*}
    \degree_1(\bm w) &= \degree_{k+1}(\bm w)=\degree_{k+2}(\bm w) = \degree_{2k+2}(\bm w) = 1, \\ \degree_i(\bm w) &= 2 \; \forall \; i \; \in \; [2k+2]\backslash\{1,k+1,k+2,2k+2\},
\end{align*}
hence $\bm w = \dlgraph$. Hence, recalling the formula for $\coeff(\bm w, \pi)$ from Lemma \ref{lemma: qf_mehler_conclusion} we obtain:
\begin{align*}
     \lim_{m \rightarrow \infty} \frac{\E (\bm z^\UT \altprod \bm z)^2}{m^2}& =  (1-\kappa)^{2k} \cdot \prod_{i=1}^{k-1} \hat{q}^2_i(2).
\end{align*}
This proves the statement of the lemma and also Proposition \ref{prop: qf_univ_mom1} (see Remark \ref{remark: all_types_qf_mom1} regarding how the analysis extends to other types).
\end{proof}
%%%%%%%%%%%%%%%%%%%%%%%%%%%%%%%%%%%%%%%%%%%%%%%%%%%%%%%%%%%%
%%%%%%%%%%%%%%%%%%%%%%%%%%%%%%%%%%%%%%%%%%%%%%%%%%%%%%%%%%%%
%%%%%%%%%%%%%%%%%%%%%%%%%%%%%%%%%%%%%%%%%%%%%%%%%%%%%%%%%%%%

%%%%%%%%%%%%%%%%%%%%%%%%%%%%%%%%%%%%%%%%%%%%%%%%%%%%%%%%%%%
%%%%%%%%%%%%%%%%%% PROOF-CLT %%%%%%%%%%%%%%%%%%%%%%%%%%%%%%
%%%%%%%%%%%%%%%%%%%%%%%%%%%%%%%%%%%%%%%%%%%%%%%%%%%%%%%%%%%
\section{Proofs from Section \ref{section: clt}}
\label{appendix: clt}

\subsection{Proof of Lemma \ref{lemma: matrix_moment_ub}} \label{proof: matrix_moment}

\begin{proof}[Proof of Lemma \ref{lemma: matrix_moment_ub}]
Recall that,
\begin{align*}
    \E|\matmom{\bm \Psi}{\bm w}{\pi}{\bm a}| & = \E\prod_{\substack{i,j \in [k] \\ i < j}} |\Psi_{a_i,a_j}^{w_{ij}}| \\
    & \explain{(a)}{\leq} \sum_{\substack{i,j \in [k] \\ i < j}} \frac{w_{ij}}{\|\bm w\|} \E |\Psi_{a_i,a_j}^{\|\bm w\|_1}| \\
    & \leq \max_{i,j \in [m] } \E |\Psi_{ij}|^{\|\bm w\|},
\end{align*}
where step $(a)$ follows from the AM-GM inequality. We now consider the subsampled Haar and Hadamard cases separately.
\begin{description}
\item [Hadamard Case: ] By Lemma \ref{concentration}, $\Psi_{ij}$ is subgaussian with with variance proxy bounded by $C/m$ for some universal constant $C$. Hence, 
\begin{align*}
     \E|\matmom{\bm \Psi}{\bm w}{\pi}{\bm a}| &  \leq  \left( \frac{C\|\bm w\|}{m} \right)^{\frac{\|\bm w\|}{2}}.
\end{align*}
\item [Haar Case: ] By Lemma \ref{concentration}, conditional on $\bm O$, $\Psi_{ij}$ is subgaussian with variance proxy $C m \|\bm o_i\|_\infty^2 \|\bm o_j\|_\infty^2$. Hence,
\begin{align*}
    \E|\matmom{\bm \Psi}{\bm w}{\pi}{\bm a}| & \leq \max_{i,j \in [m] } \E |\Psi_{ij}|^{\|\bm w\|}\\
    & = \max_{i,j \in [m] } \E [\E[|\Psi_{ij}|^{\|\bm w\|} | \bm O]] \\
    & \leq \max_{i,j \in [m]} (C\|\bm w\| m)^{\frac{\|\bm w\|}{2}} \E\left[ \|\bm o_i\|_\infty^{\|\bm w\|} \|\bm o_j\|_\infty^{\|\bm w\|}  \right] \\
    & \leq \max_{i,j \in [m]} (C\|\bm w\| m)^{\frac{\|\bm w\|}{2}} \left( \E \|\bm o_i\|_\infty^{2 \|\bm w\|} + \E \|\bm o_j\|_\infty^{2\|\bm w\|} \right).
\end{align*}
Note that $\bm o_i \explain{d}{=} \bm o_j \explain{d}{=} \bm u \sim \unif{\mathbb{S}_{m-1}}$. Applying Fact \ref{fact: infty_norm_uniform_unit} gives us,
\begin{align*}
    \E |\matmom{\bm \Psi}{\bm w}{\pi}{\bm a}|  & \leq \left( \sqrt{\frac{C \|\bm w\| \log^2(m)}{m}} \right)^{\|\bm w\|}.
\end{align*}
\end{description}
\end{proof}

\subsection{Proofs of Propositions \ref{prop: clt_random_ortho} and \ref{prop: clt_hadamard}}
\label{proof: clt_props}
This section is dedicated to the proof of Propositions \ref{prop: clt_random_ortho} and \ref{prop: clt_hadamard}. We consider the following general setup. Let $\bm v_1, \bm v_2 \cdots, \bm v_m $ be fixed vectors in $\R^d$ for a fixed $d \in \N$. Define the statistic:
\begin{align*}
    \bm T & = \sqrt{m}\sum_{i=1}^m \overline{B}_{ii} \bm v_i,
\end{align*}
where $\barB$ denotes a diagonal matrix whose $n$ diagonal entries are set to $1-\kappa$ uniformly at random and the remaining $m-n$ are set to $-\kappa$.

Analogously, we define the statistic:
\begin{align*}
    \hat{\bm T} & = \sqrt{m}\sum_{i=1}^m \hat{B}_{ii} \bm v_i,
\end{align*}
where,
\begin{align*}
    \hat{B}_{ii} & \explain{i.i.d.}{\sim} \begin{cases} 1- \kappa : & \text{with prob. } \kappa \\ -\kappa: & \text{with prob. } 1-\kappa \end{cases}.
\end{align*}
As in the proof of Lemma \ref{concentration} we define $\barB$ and $\hat{\bm B}$ in the same probability space as follows:
\begin{enumerate}
    \item We first sample $\barB$. Let $S = \{i \in [m]: \overline{B}_{ii} = 1-\kappa \}$
    \item Next sample $N \sim \bnomdistr{m}{\kappa}$.
    \item Sample a subset $\hat{S} \subset [m]$ with $|\hat{S}| = N$ as follows:
    \begin{itemize}
        \item If $N\leq n$, then set $\hat{S}$ to be a uniformly random subset of $S$ of size $N$.
        \item If $N>n$ first sample a uniformly random subset $A$ of $S^c$ of size $N-n$ and set $\hat{S} = S \cup A $
    \end{itemize}
    \item Set $\hat{\bm B}$ as follows:
    \begin{align*}
        \hat{B}_{ii} & = \begin{cases} -\kappa &: i \not\in \hat{S} \\ 1-\kappa &: i \in \hat{S}. \end{cases}.
    \end{align*}
\end{enumerate}
We stack the vectors $\bm v_{1:m}$ along the rows of a matrix $\bm V \in \R^{m\times d}$ and refer to the columns of $\bm V$ as $\bm V_1, \bm V_2 \cdots \bm V_d$:
\begin{align*}
    \bm V = [\bm V_1, \bm V_2 \cdots \bm V_d] = \begin{bmatrix} \bm v_1^\UT \\ \bm v_2^\UT \\ \vdots \\ \bm v_m^\UT \end{bmatrix}.
\end{align*}

Lastly we introduce the matrix $\hat{\bm \Sigma} \in \R^{d \times d}$:
\begin{align*}
    \hat{\bm \Sigma} &\explain{def}{=} \E[\hat{\bm T} \hat{\bm T}^\UT | \bm V] = m \kappa(1-\kappa) \bm V^\UT \bm V.
\end{align*}
These definitions are intended to capture the matrix moments $\matmom{\bm \Psi}{\bm w}{\pi}{\bm a}$ as follows:  Consider any $k \in \N, \pi \in \part{}{[k]}, \bm w \in \weightedG{k}$ and any $\bm a \in \cset{\pi}$. Let the disjoint blocks of $\pi$ be given by $\pi = \blocks_1 \sqcup \blocks_2 \cdots \sqcup \blocks_{|\pi|}$.

In order to capture $\matmom{\bm \Psi}{\bm w}{\pi}{\bm a}$ in the subsampled Hadamard case $\bm \Psi = \bm H \barB \bm H^\UT$ and the subsampled Haar case $\bm \Psi = \bm O \barB \bm O^\UT$ we will set $\bm V_{1:d}$ as follows:
\begin{enumerate}
    \item In the subsampled Haar case, we set:
    \begin{align*}
        \{\bm V_1, \bm V_2, \cdots \bm V_d \} = \{ (\bm o_{a_{\blocks_s}} \odot \bm o_{a_{\blocks_t}}) - \delta(s,t) \hat{\bm e}: s,t \in [|\pi|], \; s \leq t, \; W_{st}(\bm w, \pi) > 0 \},
    \end{align*}
    where,
    \begin{align*}
        \bm e^\UT = \left( \frac{1}{{m}}, \frac{1}{{m}} \cdots \frac{1}{{m}} \right), \; \delta(s,t) = \begin{cases} 1 : & s = t \\ 0 : & s \neq t \end{cases}.
    \end{align*}
    If for some $i \in [d]$ and some $s,t \in [|\pi|]$ we have $\bm V_i =  \bm o_{a_{\blocks_s}} \odot \bm o_{a_{\blocks_t}} - \delta(s,t) \hat{\bm e}$,
    we will abuse notation and often refer to $\bm V_i$ as $\bm V_{st}$. Likewise the corresponding entries of $\bm T, \hat{\bm T}$, $T_i, \hat{T}_i$ will be referred to as $T_{st}, \hat{T}_{st}$.
    \item In the subsampled Hadamard case, we set:
    \begin{align*}
        \{\bm V_1, \bm V_2, \cdots \bm V_d \} = \{  \bm h_{a_{\blocks_s}} \odot \bm h_{a_{\blocks_t}} - \delta(s,t) \hat{\bm e}: s,t \in [|\pi|], \; s \leq t, \; W_{st}(\bm w, \pi) > 0 \}.
    \end{align*}
    If for some $i \in [d]$ and some $s,t \in [|\pi|]$ we have $\bm V_i =  \bm h_{a_{\blocks_s}} \odot \bm h_{a_{\blocks_t}}-\delta(s,t) \hat{\bm e}$, 
    we will abuse notation and often refer to $\bm V_i$ as $\bm V_{st}$. Likewise the corresponding entries of $\bm T, \hat{\bm T}$:  $T_i, \hat{T}_i$ will be referred to as $T_{st}, \hat{T}_{st}$.
\end{enumerate}
With the above conventions and the observation that $\sum_{i=1}^m \overline{B}_{ii} = 0$ we have:
\begin{align*}
    \matmom{ \sqrt{m}\bm \Psi}{\bm w}{\pi}{\bm a} & = \prod_{\substack{s,t \in [|\pi|]\\ s \leq t \\ W_{st} (\bm w, \pi) > 0}} T_{st}^{W_{st}(\bm w, \pi)}.
\end{align*}

The remainder of this section is organized as follows:
\begin{enumerate}
    \item First, in Lemma \ref{lemma: asymptotic covariance} we show that $\hat{\bm \Sigma}$ converges to a fixed deterministic matrix $\bm \Sigma$ and bound the rate of convergence in terms of $\E \|\hat{\bm \Sigma} - \bm \Sigma \|_\fr^2$.
    \item In Lemma \ref{lemma: coupling exact form} we upper bound $\E \|\hat{\bm T} - \bm T\|_2^2$. Consequently a Gaussian approximation result for $\hat{\bm T}$ implies a Gaussian approximation result for $\bm T$.
    \item In Lemma \ref{lemma: berry_eseen}, we use a standard Berry-Esseen bound of \citet{bhattacharya1975errors} to derive a Gaussian approximation result for $\hat{\bm T}$ since it is a weighted sum of i.i.d. centered random variables.
    \item Finally we conclude by using the above lemmas to provide a proof for Propositions \ref{prop: clt_hadamard} and \ref{prop: clt_random_ortho}.
\end{enumerate}

\begin{lem}  \label{lemma: asymptotic covariance}
\begin{enumerate}
    \item For the Hadamard case suppose $\bm w$ is disassortative with respect to $\pi$ and $\bm a$ is a conflict free labelling of $(\bm w, \pi)$. Then,
    \begin{align*}
        \hat{\bm \Sigma} = \kappa(1-\kappa)\bm I_d.
    \end{align*}
    \item For the Haar case there exists a universal constant $C < \infty$ such that for any partition $\pi \in \part{}{[k]}$, any weight matrix $\bm w \in \weightedG{k}$ and any labelling $\bm a \in \cset{\pi}$ we have,
    \begin{align*}
        \E \|\hat{\bm \Sigma} - \bm \Sigma \|_\fr^2 & \leq  \frac{C\cdot k^4 \cdot (\kappa^2 (1-\kappa)^2)}{m}.
    \end{align*}
    where the matrix $\bm \Sigma$ is a diagonal matrix whose diagonal entries are given by:
    \begin{align*}
        \Sigma_{st,st} & = \begin{cases} \kappa(1-\kappa) : & s\neq t \\ 2 \kappa(1-\kappa) : & s = t \end{cases}.
    \end{align*}
\end{enumerate}
\end{lem}

\begin{proof} Recall that,
\begin{align*}
    \hat{\bm \Sigma} & = m \kappa(1-\kappa) \bm V^\UT \bm V.
\end{align*}

We consider the Hadamard and the Haar case separately. 
\begin{description}
\item [Hadamard Case: ] Consider two pairs $(s,t)$ and $(s^\prime,t^\prime)$ such that:
\begin{align*}
    s \leq t, \; W_{st}(\bm w, \pi) > 0, \; s, t \; \in \;  [|\pi|].
\end{align*}
and the analogous assumptions on the pair $(s^\prime,t^\prime)$.
Then the entry $\hat{\Sigma}_{st,s^\prime t^\prime}$ is given by:
\begin{align*}
    \hat{\Sigma}_{st,s^\prime t^\prime} & = m \kappa(1-\kappa) \ip{\bm V_{st}}{\bm V_{s^\prime t^\prime}} \\
    & = m \kappa(1-\kappa) \ip{\bm h_{a_{\blocks_s}}  \odot \bm  h_{a_{\blocks_t}} - \delta(s,t) \hat{\bm e}}{\bm h_{a_{\blocks_s^\prime}}  \odot \bm  h_{a_{\blocks_t^\prime}}-\delta(s^\prime,t^\prime) \hat{\bm e}} \\
    & \explain{(a)}{=} \kappa(1-\kappa)\ip{\bm h_{a_{\blocks_s} \oplus a_{\blocks_t}}  - \sqrt{m}\delta(s,t) \hat{\bm e}}{\bm h_{a_{\blocks_s^\prime}\oplus a_{\blocks_t^\prime}}-\sqrt{m}\delta(s^\prime,t^\prime) \hat{\bm e}} \\
    & \explain{(b)}{=}\kappa(1-\kappa) \ip{\bm h_{a_{\blocks_s} \oplus a_{\blocks_t}}}{\bm h_{a_{\blocks_s^\prime}\oplus a_{\blocks_t^\prime}}} \\
    &\explain{(c)}{=} \kappa(1-\kappa)\delta(s,s^\prime) \delta(t,t^\prime).
\end{align*}
In the step marked (a) we appealed to Lemma \ref{lemma: hadamard_key_property}. In the step marked (b), we noted that $\hat{\bm e} = \bm h_1/\sqrt{m}$ and $\hat{\bm e} \perp \bm h_{a_{\blocks_s} \oplus a_{\blocks_t}}$ unless $s = t$ which is ruled out by the fact that $\bm w$ is disassortative with respect to $\pi$ i.e. $W_{ss}(\bm w,\pi) = 0$. In the step marked (c) we used the fact that $\bm a$ is a conflict free labelling. Consequently, we have shown that $\hat{\bm \Sigma} = \kappa(1-\kappa)\bm I_d$.
\item [Haar case: ] By the bias-variance decomposition:
\begin{align*}
    \E \|\hat{\bm \Sigma} - \bm \Sigma \|_\fr^2 & = \E \|\hat{\bm \Sigma} - \E\hat{\bm \Sigma} \|_\fr^2 + \|\E\hat{\bm \Sigma} - \bm \Sigma \|_\fr^2.
\end{align*}
We will first compute $\E \hat{\bm \Sigma} $. Consider the $(st,s^\prime t^\prime)$ entry of $\hat{\bm \Sigma}$:
\begin{align*}
     \hat{\Sigma}_{st,s^\prime t^\prime} & = m \kappa(1-\kappa) \ip{\bm V_{st}}{\bm V_{s^\prime t^\prime}} \\
    & = m \kappa(1-\kappa) \ip{\bm o_{a_{\blocks_s}}  \odot \bm  o_{a_{\blocks_t}} - \delta(s,t) \hat{\bm e}}{\bm o_{a_{\blocks_s^\prime}}  \odot \bm  o_{a_{\blocks_t^\prime}}-\delta(s^\prime,t^\prime) \hat{\bm e}} \\
    & = m \kappa(1-\kappa) \left[ \sum_{i=1}^m \left( \bm (\bm o_{a_{\blocks_s}})_i  (\bm  o_{a_{\blocks_t}})_i -  \frac{\delta(s,t)}{m}\right) \left( (\bm o_{a_{\blocks_s^\prime}})_i   (\bm  o_{a_{\blocks_t^\prime}})_i- \frac{\delta(s^\prime,t^\prime)}{m}\right)  \right].  
\end{align*}
Note that $\bm O_i$ is a uniformly random unit vector. Hence we can compute $\E \hat{\bm \Sigma}$ using Fact \ref{fact: unit_vector_moments}. We obtain:
\begin{align*}
    \frac{\E \hat{\Sigma}_{st,s^\prime t^\prime}}{\kappa(1-\kappa)}& = \begin{cases} 2 - \frac{6}{m+2}: & s = s^\prime = t = t^\prime \\
    \frac{2}{(m-1)(m+2)}:  & s = t, s^\prime = t^\prime, s \neq s^\prime \\
    1 + \frac{2}{(m-1)(m+2)}: & s = s^\prime, t = t^\prime,  s \neq t \\ 0 : & \text{otherwise} \end{cases}.
\end{align*}
Hence, the bias term can be bounded by:
\begin{align*}
     \|\E\hat{\bm \Sigma} - \bm \Sigma \|_\fr^2 & \leq \frac{36\cdot k^4 \cdot \kappa^2 (1-\kappa)^2}{(m+2)^2}.
\end{align*}
On the other hand, applying the Poincare Inequality (Fact \ref{fact: poincare_haar}) and a tedious calculation involving 6th moments of a random unit vector (see for example Proposition 2.5 of \citet{meckes2019random}) shows that,
\begin{align*}
    \var(\hat{\Sigma}_{st,s^\prime t^\prime}) & \leq \frac{C\cdot \kappa^2(1-\kappa)^2}{m},
\end{align*}
for some universal constant $C$.
Hence,
\begin{align*}
    \E \|\hat{\bm \Sigma} - \E\hat{\bm \Sigma} \|_\fr^2 & \leq \frac{C\cdot k^4\cdot \kappa^2(1-\kappa)^2}{m},
\end{align*}
for some universal constant $C$, and consequently the claim of the lemma holds. 
\end{description}
\end{proof}

\begin{lem} \label{lemma: coupling exact form} 
 We have,
 \begin{align*}
     \E \left[ \|\bm T - \hat{\bm T} \|^2_2\right] & \leq \frac{Ck^3}{\sqrt{m}}, 
\end{align*}
for a universal constant $C$.
\end{lem}

\begin{proof}
Let $\bar {\bm b}, \hat{\bm b} \in \R^m$ be the vectors formed by  the diagonals of $\barB, \hat{\bm B}$, respectively. Define:
\begin{align*}
    p_1 & = \P( \bar{b}_1 \neq \hat{b}_1), \; p_2 = \P(\bar{b}_1 \neq \hat{b}_1, \; \bar{b}_2 \neq \hat{b}_2).
\end{align*}
We have,
\begin{align*}
    \E \left[ \|\bm T - \hat{\bm T} \|^2_2 \; | \;  \bm V \right] & =
    m \e{(\barb - \hatb)^\UT {\bm V \bm V^\UT} (\barb - \hatb)}
    \\ &
    = m \Tr \intoo{{\bm V \bm V^\UT} \e{(\barb - \hatb) (\barb - \hatb)^\UT}}
    \\ &
    = m \Tr \intoo{{\bm V \bm V^\UT} (1 - 2 \kappa)^2 \intoo{p_2 {\bm 1 \bm 1^\UT + (p_1 - p_2) \bm I_m}}}
    \\ &
    = m (1 - 2 \kappa)^2 \intoo{p_2 \norm{\bm V^\UT \bm 1}^2_2 + (p_1 - p_2) \Tr \intoo{\bm V \bm V^\UT} }.
\end{align*}
Now, since $\bm V^\UT$ has centered coordinate-wise product of columns of an orthogonal matrix we have $ \bm V^\UT \bm 1 = 0 $. Hence,
\begin{align*}
    \E \left[ \|\bm T - \hat{\bm T} \|^2_2 \; | \;  \bm V \right] &  = (p_1 - p_2) \Tr \intoo{\bm V \bm V^\UT} .
\end{align*}
Next we compute $p_1 = \P(\bar{b}_1 \neq \hat{b}_1)$. Observe that conditional on $N$, the symmetric difference $S \triangle \hat{S}$ is a uniformly random set of size $|N-n|$. Hence,
\begin{align*}
    \P(\bar{b}_1 \neq \hat{b}_1 | N) & = \P( 1 \in  S \triangle \hat{S} | N) = \frac{|n-N|}{m}.
\end{align*}
Therefore
\begin{align*}
    p_1 = \frac{\e{N - n}}{m} \leq \frac{\sqrt{{\rm Var}(N))}}{m} = \frac{\sqrt{\kappa (1 - \kappa)}}{\sqrt{m}}.
\end{align*}
Hence, we  obtain
\begin{equation}
    \E \left[ \|\bm T - \hat{\bm T} \|^2_2 |\bm V\right] \leq  \frac{(1-2\kappa)^2 }{\sqrt{m \cdot \kappa(1-\kappa)}} \cdot  \Tr(\hat{\bm \Sigma}).
\end{equation}
By Lemma \ref{lemma: asymptotic covariance} we have,
\begin{align*}
    \E \Tr(\hat{\bm \Sigma}) & \leq \E \Tr({\bm \Sigma}) + \sqrt{d \cdot \E \|\hat{\bm \Sigma} - \bm \Sigma \|_\fr^2} \\
    & \leq  C \kappa(1-\kappa) k^3.
\end{align*}
where constant $C_{\kappa, d}$ depends only on $\kappa, d$.
And hence,
\begin{align*}
     \E \left[ \|\bm T - \hat{\bm T} \|^2_2\right] & \leq \frac{Ck^3}{\sqrt{m}}, 
\end{align*}
for a universal constant $C$.
\end{proof}

\begin{lem} \label{lemma: berry_eseen} Under the assumptions and notations of Lemma \ref{lemma: asymptotic covariance} for both the subsampled Haar sensing and the subsampled Hadamard sensing models, we have, for any bounded Lipschitz function $f: \R^d \rightarrow \R$:
\begin{align} \label{eq:lemma berry eseen}
    \E\left| \E[ f(\hat{\bm T}) | \bm V] - \E f(\hat{\bm {\Sigma}}^{1/2} \bm Z) \right| \leq   \frac{C_k \cdot (\| f\|_\infty + \|f\|_{\mathsf{Lip}})}{\sqrt{m}}.
\end{align}
where $\bm Z \sim \gauss{\bm 0}{\bm I_d}$, $C_k$ is a constant depending only on $k$.
\end{lem}

\begin{proof}
Note that $\hat{\bm T} = \sqrt{m} {\bm V}^\UT \hatb$ and $ \sqrt{m} \hat{\bm \Sigma}^{\frac{-1}{2}} \bm{V}^\UT \hatb $ has the identity covariance matrix.  Hence, by the Berry-Esseen bound of \citet{bhattacharya1975errors} for any bounded and Lipschitz function $g$ we have
\begin{align} \label{eq: berry esseen normalized}
    \abs{\e{g \intoo{\hat{\bm \Sigma}^{\frac{-1}{2}}\hat{\bm T}}} - \e{g(\bm Z)}} \leq \frac{C_d \cdot \rho_3' \cdot  \intoo{\norm{g}_{\infty} + \norm{g}_{Lip}}}{\sqrt{m}},
\end{align}
where $C_d$ is a constant only dependent on $d$ and
\begin{align*}
    \rho_3' &= m^2 \sum_{i = 1}^m \e{|\hat{b}_i|^3 \cdot  \|\hat{\bm \Sigma}^{\frac{-1}{2}} \bm v_i\|^3_2 | \bm V}
    \\ &
    = m^2 \intoo{\kappa (1 - \kappa)^3 + (1 - \kappa) \kappa^3} \sum_{i
    = 1}^m { \|\hat{\bm \Sigma}^{\frac{-1}{2}} \bm v_i\|^3_2 }
    \\ &
    \leq m^2 \cdot \sqrt{d} \cdot \|\hat{\bm \Sigma}^{-\frac{1}{2}}\|_{\op}^3 \cdot (\kappa(1-\kappa)) \cdot  \sum_{i=1}^m  \|\bm v_i\|_3^3\end{align*}.

%Moreover, since $ \hat{\bm \Sigma} = m \kappa(1-\kappa) \bm V^\UT \bm V $, we can rewrite $\rho_3'$ as
%\begin{align} \label{eq:rho' definition}
%    \rho_3' 
%    = \frac{c_\kappa}{m} \sum_{i = 1}^m {\norm{\intoo{\bm V^\UT \bm V}^{\frac{-1}{2}} {\bm V^\UT} {\bm e_i}}}^3.
%\end{align}

Define $g(X) \define f \intoo{{\hat{\bm \Sigma}}^{\frac{1}{2}} {\bm X} } $, hence, $g \intoo{\hat{\bm \Sigma}^{\frac{-1}{2}} {\bm V}^\UT \hatb} = f \intoo{\hat{T}}$.  Moreover, $\norm{g}_{\infty} \leq \norm{f}_{\infty}$ and $\norm{g}_{Lip} \leq  \norm{\bm \Sigma}_{op}^{\frac{1}{2}} \norm{f}_{Lip}$. Hence we obtain:
\begin{align} \label{eq: berry_eseen_conclusion}
    &\left| \E[ f(\bm \hat{T}) | \bm V] - \E f(\hat{\bm {\Sigma}}^{1/2} \bm Z) \right|  \leq \nonumber  \\& \hspace{2cm}
    {C_d   (\kappa(1-\kappa)) \cdot  m^{\frac{3}{2}} \cdot  (\norm{f}_{\infty} + \|\hat{\bm \Sigma}\|_{\op}^{\frac{1}{2}} \norm{f}_{\mathsf{Lip}} } ) \cdot \|\hat{\bm \Sigma}^{-\frac{1}{2}}\|_{\op}^3 \cdot \sum_{i=1}^m  \|\bm v_i\|_3^3.
\end{align}
We define the event:
\begin{align*}
   \mathcal{E} \explain{def}{=} \left\{\bm V: \|\hat{\bm \Sigma} - \bm \Sigma \|_\fr^2 \leq \frac{\kappa^2(1-\kappa)^2}{4}\right\}.
\end{align*}
By Markov Inequality and Lemma \ref{lemma: asymptotic covariance}, we know that, $\P(\mathcal{E}^c) \leq Ck^4/m$ for some universal constant $C$. Hence,
\begin{align*}
     \E\left| \E[ f(\bm \hat{T}) | \bm V] - \E f(\hat{\bm {\Sigma}}^{1/2} \bm Z) \right| & \leq \frac{2 C \cdot  \|f\|_\infty \cdot  k^4}{m} + \E \left| \E[ f(\bm \hat{T}) | \bm V] - \E f(\hat{\bm {\Sigma}}^{1/2} \bm Z) \right| \Indicator{\mathcal{E}}.
\end{align*}
On the event $\mathcal{E}$ we have,
\begin{align*}
    \|\hat{\bm \Sigma}\|_{\op} & \leq \|\bm \Sigma\|_{\op} + \frac{\kappa(1-\kappa)}{2} \leq \frac{5\kappa(1-\kappa)}{2}, \\
     \|\hat{\bm \Sigma}^{-\frac{1}{2}}\|_{\op} & \leq \|{\bm \Sigma}^{-\frac{1}{2}}\|_{\op}  +  \|\hat{\bm \Sigma}^{-\frac{1}{2}} -  {\bm \Sigma}^{-\frac{1}{2}}\|_{\op} \explain{(a)}{\leq} \frac{1}{\kappa(1-\kappa)} + \frac{1}{2} \leq \frac{9}{8(\kappa(1-\kappa))}, \\
     \E\|\bm v_i\|^3  & = \sum_{j=1}^d \E |v_{ij}|^3 \explain{(b)}{\leq} \frac{Cd}{m^3}.
\end{align*}
In the step marked (a) we used the continuity estimate for matrix square root in Fact \ref{fact: matrix_sqrt}. In the step marked (b), we recalled the definition of $\bm v_i$ and used the moment bounds for a coordinate of a random unit vector from Fact \ref{fact: unit_vector_moments}. Substituting these estimates in \eqref{eq: berry_eseen_conclusion} we obtain:
\begin{align*}
    \E\left| \E[ f(\hat{\bm T}) | \bm V] - \E f(\hat{\bm {\Sigma}}^{1/2} \bm Z) \right| \leq \frac{2 C \cdot  \|f\|_\infty \cdot  k^4}{m} +  \frac{C_k \cdot (\| f\|_\infty + \|f\|_{\mathsf{Lip}})}{\sqrt{m}}.
\end{align*}
\end{proof}

Using the above lemmas, we can now provide a proof of Propositions \ref{prop: clt_hadamard} and \ref{prop: clt_random_ortho}.

\begin{proof}[Proof of Propositions \ref{prop: clt_hadamard} and \ref{prop: clt_random_ortho}]
Define the polynomial $p(\bm z)$ as:
\begin{align*}
    p(\bm z) \explain{def}{=} \prod_{\substack{s,t \in [|\pi|]\\ s \leq t \\ W_{st} (\bm w, \pi) > 0}} z_{st}^{W_{st}(\bm w, \pi)},
\end{align*}
and the indicator function:
\begin{align*}
    \Indicator{\mathcal{E}}(\bm z) & \explain{def}{=} \begin{cases} 1 : & \bm z \in \mathcal{E} \\ 0 : & \bm z \not \in \mathcal{E} \end{cases},
\end{align*}
where:
\begin{align*}
    \mathcal{E} &\explain{def}{=} \left\{ \max_{s,t} |z_{st}|  \leq \left({2048\log^3(m)} \right)^{\frac{1}{2}}\right\}.
\end{align*}
Recall that we had,
\begin{align*}
    \matmom{ \sqrt{m}\bm \Psi}{\bm w}{\pi}{\bm a} & = \prod_{\substack{s,t \in [|\pi|]\\ s \leq t \\ W_{st} (\bm w, \pi) > 0}} T_{st}^{W_{st}(\bm w, \pi)} \explain{}{=} p(\bm T),
\end{align*}
and in Lemma \ref{lemma: trace_good_event} we showed that,
\begin{align*}
    \P(\bm T \notin \mathcal{E}) & \leq \frac{C}{m^2}. 
\end{align*}
We additionally define the function $\widetilde{p}(\bm z) \explain{def}{=} p(\bm z) \Indicator{\mathcal{E}}(\bm z)$. observe that:
\begin{align*}
    \|\widetilde{p}\|_\infty & \leq \left({2048\log^3(m)} \right)^{\frac{\|\bm w\|}{2}}, \; \|\widetilde{p}\|_\mathsf{Lip}  \leq \|\bm w\| \left({2048\log^3(m)} \right)^{\frac{\|\bm w\|}{2}}.
\end{align*}
Let $\bm Z \sim \gauss{\bm 0}{\bm I_d}$. Then, we can write:
\begin{align*}
    &\left| \E p(\bm T) - \E p(\bm \Sigma^{\frac{1}{2}} \bm Z) \right|  \leq  \left| \E \widetilde{p}(\bm T) - \E \widetilde{p}(\bm \Sigma^{\frac{1}{2}} \bm Z) \right|  + |\E p(\bm T) \Indicator{\mathcal{E}^c}(\bm T)| + |\E p(\bm T) \Indicator{\mathcal{E}^c}(\bm \Sigma^{\frac{1}{2}} Z)| \\
     &\leq\underbrace{\left| \E \widetilde{p}(\bm T) - \E \widetilde{p}(\hat{\bm T}) \right|}_{\mathsf{(I)}}+  \underbrace{\left| \E \widetilde{p}(\bm T) - \E \widetilde{p}(\hat{\bm \Sigma}^{\frac{1}{2}} \bm Z) \right|}_{\mathsf{(II)}} + \underbrace{\left| \E \widetilde{p}({\bm \Sigma}^{\frac{1}{2}} \bm Z) - \E \widetilde{p}(\hat{\bm \Sigma}^{\frac{1}{2}} \bm Z) \right|}_{\mathsf{(III)}}  \\&\hspace{7cm} + \underbrace{|\E p(\bm T) \Indicator{\mathcal{E}^c}(\bm T)|}_{\mathsf{(IV)}} + \underbrace{|\E p( {\bm \Sigma}^{\frac{1}{2}}\bm Z) \Indicator{\mathcal{E}^c}(\bm \Sigma^{\frac{1}{2}}\bm  Z)|}_{\mathsf{(V)}}.
\end{align*}
We control each of these terms separately.
\begin{description}
\item [Analysis of $(\mathsf{I})$:] In order to control $\mathsf{I}$ observe that:
\begin{align*}
    \mathsf{(I)} & \leq \|\widetilde{p}\|_\mathsf{Lip} \E \|\bm T - \hat{\bm T}\|_2  \\& \leq \|\widetilde{p}\|_\mathsf{Lip} \cdot (\E \|\bm T - \hat{\bm T}\|_2^2)^{\frac{1}{2}}\\  & \leq C \cdot \|\bm w\|\cdot  \left({2048\log^3(m)} \right)^{\frac{\|\bm w\|}{2}} \cdot \frac{\sqrt{k^3}}{m^\tbd}.
\end{align*}
In the last step, we appealed to Lemma \ref{lemma: coupling exact form}.
\item [Analysis of $(\mathsf{II})$:] In order to control $\mathsf{I}$, recall that:
\begin{align*}
    \|\widetilde{p}\|_\infty & \leq \left({2048\log^3(m)} \right)^{\frac{\|\bm w\|}{2}}, \; \|\widetilde{p}\|_\mathsf{Lip}  \leq \|\bm w\| \left({2048\log^3(m)} \right)^{\frac{\|\bm w\|}{2}}.
\end{align*}
Hence, by Lemma \ref{lemma: berry_eseen} we have,
\begin{align*}
    \mathsf{(II)} & \leq \frac{C_k \cdot ({2048\log^3(m)} )^{\frac{\|\bm w\|}{2}} (1 + \|\bm w\|)}{\sqrt{m}}.
\end{align*}
\item [Analysis of $(\mathsf{III})$: ]Again using the Lipchitz bound on $\widetilde{p}$ we have,
\begin{align*}
    (\mathsf{III}) & \leq \E |\widetilde{p}({\bm \Sigma}^{\frac{1}{2}} \bm Z) -  \widetilde{p}(\hat{\bm \Sigma}^{\frac{1}{2}} \bm Z)| \\
    & \leq \|\bm w\| \left({2048\log^3(m)} \right)^{\frac{\|\bm w\|}{2}} \cdot \E \|(\hat{\bm \Sigma}^{\frac{1}{2}} - {\bm \Sigma}^{\frac{1}{2}}) \bm Z\|_2 \\
    & \leq \|\bm w\| \left({2048\log^3(m)} \right)^{\frac{\|\bm w\|}{2}} \cdot \sqrt{\E \|(\hat{\bm \Sigma}^{\frac{1}{2}} - {\bm \Sigma}^{\frac{1}{2}}) \bm Z\|^2_2} \\
    & \leq \|\bm w\| \left({2048\log^3(m)} \right)^{\frac{\|\bm w\|}{2}} \cdot \sqrt{\E \|\hat{\bm \Sigma}^{\frac{1}{2}} - {\bm \Sigma}^{\frac{1}{2}} \|^2_\fr} \\
    &\explain{(a)}{\leq} \|\bm w\| \left({2048\log^3(m)} \right)^{\frac{\|\bm w\|}{2}} \cdot  \frac{k^2}{\lambda_{\max}(\bm \Sigma)} \cdot \E \|\hat{\bm \Sigma} - \bm \Sigma\|_\fr^2 \\
    & \explain{(b)}{\leq} \frac{C \cdot k^6 \cdot \|\bm w\| ({2048\log^3(m)} )^{\frac{\|\bm w\|}{2}}}{m}.
\end{align*}
In the step marked (a) we used the fact that the continuity estimate for matrix square roots given in Fact \ref{fact: matrix_sqrt}. In the step marked (b) we recalled the definition of $\bm \Sigma$ and observed that $\lambda_{\max}(\bm \Sigma) \geq \kappa(1-\kappa)$ for the subsampled Haar and the Hadamard sensing model. We also used the bound on $\E \|\hat{\bm \Sigma} - \bm \Sigma \|_{\fr}^2$ obtained in Lemma \ref{lemma: asymptotic covariance}.
\item  [Analysis of $(\mathsf{IV})$: ] We can control $(\mathsf{III})$ as follows:
\begin{align*}
    (\mathsf{IV}) & \leq \sqrt{\E p^2(\bm T) } \cdot \sqrt{\P(\bm T \not \in \mathcal{E})} \\
    & \explain{(c)}{\leq} \frac{C  \sqrt{\E \matmom{\sqrt{m}\bm \Psi}{2 \bm w}{\pi}{\bm a}}}{m} \\
    & \explain{(d)}\leq \frac{(C\|\bm w\| \log^2(m))^{\frac{\|\bm w\|}{2}}}{m}
\end{align*}
In the step marked (c) we recalled that $\P(\bm T \notin \mathcal{E}) \leq C/m^2$ and expressed $p^2(\bm T)$ as a matrix moment. In the step marked (d) we used the bounds on matrix moments obtained in Lemma \ref{lemma: matrix_moment_ub}.
\item  [Analysis of $(\mathsf{IV})$: ] We recall that $\bm \Sigma$ was a diagonal matrix with $|\Sigma_{ii}| \leq 2 \kappa(1-\kappa) \leq 1$. Hence,
\begin{align*}
    (\mathsf{V}) & \leq \sqrt{\E p^2(\bm \Sigma^{\frac{1}{2}})} \cdot \sqrt{\P (\bm \Sigma^{\frac{1}{2}} \bm Z \notin \mathcal{E})} \\
    & \explain{(e)}{\leq} \frac{k\|\bm w\|^{\frac{\|\bm w\|}{2}}}{m}.
\end{align*}
In the step marked (e) we used standard moment and tail bounds on Gaussian random variables. 
\end{description}
Combining the bounds on $\mathsf{I}-\mathsf{V}$ immediately yields the claims of Proposition \ref{prop: clt_hadamard} and \ref{prop: clt_random_ortho}.
\end{proof}
%%%%%%%%%%%%%%%%%%%%%%%%%%%%%%%%%%%%%%%%%%%%%%%%%%%%%%%%%%%%%
%%%%%%%%%%%%%%%%%%%%%%%%%%%%%%%%%%%%%%%%%%%%%%%%%%%%%%%%%%%%%
%%%%%%%%%%%%%%%%%%%%%%%%%%%%%%%%%%%%%%%%%%%%%%%%%%%%%%%%%%%%%

%%%%%%%%%%%%%%%%%%%%%%%%%%%%%%%%%%%%%%%%%%%%%%%%%%%%%%%%%%
%%%%%%%%%%%%%%%%%% APPENDIX-QF%%%%%%%%%%%%%%%%%%%%%%%%%%%%
%%%%%%%%%%%%%%%%%%%%%%%%%%%%%%%%%%%%%%%%%%%%%%%%%%%%%%%%%%
\section{Missing Proofs from Section \ref{section: qf_proof}}
\label{appendix: qf}

\subsection{Proof of Lemma \ref{lemma: continuity_qf}}
\label{proof: continuity_qf}
\begin{proof}[Proof of Lemma \ref{lemma: continuity_qf}] We will assume that $\altprod$ is of Type 1 (the proof of the other types is analogous): 
\begin{align*}
    \altprod(\bm \Psi, \bm Z) = p_1(\bm \Psi) q_1(\bm Z) p_2(\bm \Psi) \cdots  q_{k-1}(\bm Z) p_k(\bm \Psi).
\end{align*}
% \begin{enumerate}
    % \item (Continuity with respect to $\bm z$) 
    Define for any $i \in [k]$:
    \begin{align*}
        \altprod_0 & \explain{def}{=} p_1(\bm \Psi) q_1(\diag{\bm z}) p_2(\bm \Psi) \cdots  q_{k-1}(\diag{\bm z}) p_k(\bm \Psi), \\
        \altprod_i & \explain{def}{=} p_1(\bm \Psi) q_1(\diag{\widetilde{\bm z}})  \cdots  q_i(\diag{\widetilde{\bm z}}) p_{i+1}(\bm \Psi) q_{i+1}(\diag{{\bm z}}) \cdots q_{k-1}(\diag{\bm z}) p_k(\bm \Psi).
    \end{align*}
    where $\bm \Psi = \bm U \barB \bm U^\UT$.
    Observe that we can write:
    \begin{align*}
        &\bm z^\UT \altprod(\bm U \barB \bm U^\UT, \diag{\bm z}) \bm z - \widetilde{\bm z}^{\UT} \altprod(\bm U \barB \bm U^\UT, \diag{\widetilde{\bm z}}) \widetilde{\bm z}  = \bm z^\UT \altprod_0 \bm z - \widetilde{\bm z}^\UT \altprod_{k-1} \widetilde{\bm z} \\
        & = \bm z^\UT \altprod_0 \bm z - \bm z^\UT \altprod_{k-1} \bm z + \bm z^\UT \altprod_{k-1} \bm z +  \widetilde{\bm z}^\UT \altprod_{k-1} \widetilde{\bm z} \\
        & = \left(\sum_{i=0}^{k-2} \bm z^\UT (\altprod_i - \altprod_{i+1}) \bm z \right) + \ip{\altprod_{k-1}}{\bm z \bm z^\UT - \widetilde{\bm z}\widetilde{\bm z}^\UT}. 
    \end{align*}
    We bound each of these terms separately. First observe that:
    \begin{align*}
        |\bm z^\UT (\altprod_i - \altprod_{i+1}) \bm z| & \leq \|\bm z\|_2^2 \cdot \| \altprod_i - \altprod_{i+1}\|_\op \\
        & \leq  C(\altprod) \cdot  \|\bm z\|_2^2 \cdot \|\bm z - \widetilde{\bm z}\|_\infty.
    \end{align*}
    Next we note that,
    \begin{align*}
        |\ip{\altprod_{k-1}}{\bm z \bm z^\UT - \widetilde{\bm z}\widetilde{\bm z}^\UT}| & \leq 2\|\altprod_{k-1}\|_\op \cdot \| \bm z \bm z^\UT - \widetilde{\bm z}\widetilde{\bm z}^\UT\|_\op \\
        & = C(\altprod) \cdot \|\bm z - \widetilde{\bm z}\|_2 \cdot (\|\bm z\|_2 + \|\widetilde{\bm z}\|_2).
    \end{align*}
    This gives is the estimate:
    \begin{align*}
         &\left| \frac{\bm z^\UT \altprod(\bm U \barB \bm U^\UT, \diag{\bm z}) \bm z}{m}  - \frac{\widetilde{\bm z}^{\UT} \altprod(\bm U \barB \bm U^\UT, \diag{\widetilde{\bm z}}) \widetilde{\bm z}}{m}   \right|  \leq \\ & \hspace{5cm} \frac{C(\altprod)}{m} \cdot \left(  \|\bm z\|_2^2 \cdot \|\bm z - \widetilde{\bm z}\|_\infty +   \|\bm z - \widetilde{\bm z}\|_2 \cdot (\|\bm z\|_2 + \|\widetilde{\bm z}\|_2) \right),
    \end{align*}
    where $C(\altprod)$ denotes a finite constant depending only on the $\|\|_\infty$ norms and Lipchitz constants of the functions appearing in $\altprod$.
\end{proof}

\subsection{Proof of Lemma \ref{lemma : qf_variance_normalization}}
\label{proof: qf_variance_normalization}
\begin{proof}[Proof of Lemma \ref{lemma : qf_variance_normalization}] Using the continuity estimate from Lemma \ref{lemma: continuity_qf} we know that on the event $\mathcal{E}$,
\begin{align*}
     &\left| \frac{\bm z^\UT \altprod(\bm \Psi, \bm Z) \bm z}{m}  - \frac{\widetilde{\bm z}^{\UT} \altprod(\bm\Psi, \widetilde{\bm Z}) \widetilde{\bm z}}{m}   \right|  \leq \frac{C(\altprod)}{m} \cdot \left(  \|\bm z\|_2^2 \cdot \|\bm z - \widetilde{\bm z}\|_\infty +   \|\bm z - \widetilde{\bm z}\|_2 \cdot (\|\bm z\|_2 + \|\widetilde{\bm z}\|_2) \right) \\
     & \leq \frac{C(\altprod)}{m} \cdot \left(  \|\bm z\|_2^2 \cdot \|\bm z\|_\infty +   \|\bm z\|_2 \cdot (\|\bm z\|_2 + \|\widetilde{\bm z}\|_2) \right) \cdot \left( \max_{i \in [m]} \left| \frac{1}{\sigma_i} - 1 \right| \right) \\
     & \leq \frac{C(\altprod)}{m \kappa } \cdot \left(  \|\bm z\|_2^2 \cdot \|\bm z\|_\infty +   \|\bm z\|_2 \cdot (\|\bm z\|_2 + \|\widetilde{\bm z}\|_2) \right) \cdot \sqrt{\frac{\log^{3}(m)}{m}} 
\end{align*}
Hence,
\begin{align*}
    &\left| \E \frac{\bm z^\UT \altprod(\bm \Psi, \bm Z) \bm z}{m}  - \E \frac{\widetilde{\bm z}^{\UT} \altprod(\bm\Psi, \widetilde{\bm Z}) \widetilde{\bm z}}{m} \Indicator{\mathcal{E}} \right| \leq \left| \E \frac{\bm z^\UT \altprod(\bm \Psi, \bm Z) \bm z}{m} \Indicator{\mathcal{E}^c}\right|  \\ &\hspace{4cm}+ \frac{C(\altprod)\log^{\frac{3}{2}}(m)}{m\sqrt{m} \kappa } \cdot \left(  \E\|\bm z\|_2^2 \cdot \|\bm z\|_\infty +   \E\|\bm z\|_2 \cdot (\|\bm z\|_2 + \|\widetilde{\bm z}\|_2) \right).
\end{align*}
Observe that $\bm z^\UT \altprod \bm z \leq  \|\altprod\|_\op \|\bm z\|^2 \leq C(\altprod) \|\bm z\|^2_2 \leq C(\altprod) \|\bm x\|^2_2 $. 
Hence,
\begin{align*}
    \left| \E \frac{\bm z^\UT \altprod(\bm \Psi, \bm Z) \bm z}{m} \Indicator{\mathcal{E}^c}\right| & \leq C(\altprod)\frac{\sqrt{\E\|\bm x\|^4_2 \cdot \P(\mathcal{E}^c)}}{m} \leq \frac{C(\altprod) \sqrt{\P(\mathcal{E}^c)}}{\kappa^2} \rightarrow 0, \\
    \E\|\bm z\|^2_2 + \E \|\bm z\|_2 \|\widetilde{\bm z}\|_2  & \leq 2\E \|\bm z\|_2^2 + \E \|\widetilde{\bm z}\|_2^2 \leq 2\E \|\bm x\|_2^2 + \E \|\widetilde{\bm z}\|_2^2 = \frac{2m}{\kappa} + m, \\
    \E\|\bm z\|_2^2 \cdot \|\bm z\|_\infty & \leq m \E \|\bm z\|_\infty^3 \leq m \left(\E \|\bm z\|^9_9 \right)^{\frac{1}{3}} \leq C m^{\frac{4}{3}}.
\end{align*}
This gives us,
\begin{align*}
     \left| \E \frac{\bm z^\UT \altprod(\bm \Psi, \bm Z) \bm z}{m}  - \E \frac{\widetilde{\bm z}^{\UT} \altprod(\bm\Psi, \widetilde{\bm Z}) \widetilde{\bm z}}{m} \Indicator{\mathcal{E}} \right|  \rightarrow 0,
\end{align*}
and hence we have shown,
\begin{align*}
    \lim_{m \rightarrow \infty} \frac{\E \bm z^\UT \altprod(\bm \Psi, \bm Z) \bm z}{m}&=\lim_{m \rightarrow \infty}\E \frac{\widetilde{\bm z}^{\UT} \altprod(\bm\Psi, \widetilde{\bm Z}) \widetilde{\bm z}}{m} \Indicator{\mathcal{E}},
\end{align*}
provided the latter limit exists.
\end{proof}

\subsection{Proof of Lemma \ref{lemma: qf_mehler_conclusion}}
\label{proof: qf_mehler_conclusion}
\begin{proof}[Proof of Lemma \ref{lemma: qf_mehler_conclusion}] Recall that:
\begin{align*}
        \widetilde{z}_{a_1} \widetilde{z}_{a_{k+1}}\prod_{i=1}^k q_i(\widetilde{z}_{a_i}) & = Q_\firstfnc(\widetilde{z}_{a_1}) \cdot  Q_\lastfnc(\widetilde{z}_{a_{k+1}})  \left( \prod_{i \in \singleblks{\pi}} q_{i-1}(\widetilde{z}_{a_i}) \right)  \prod_{i=1}^{|\pi| - |\singleblks{\pi}| - 2} (Q_{\blocks_i}(z_{a_{\blocks_i}}) + \mu_{\blocks_i})
\end{align*}
Hence,
    \begin{align}
        &\E[ \widetilde{z}_{a_1} q_1(\widetilde{z}_{a_2}) q_2(\widetilde{z}_{a_3}) \cdots q_{k-1}(\widetilde{z}_{a_k})  \widetilde{z}_{a_{k+1}} |  \bm A ]   = \nonumber \\ & \sum_{V \subset [|\pi| - |\singleblks{\pi}| - 2] }  \E\left[ Q_\firstfnc(\widetilde{z}_{a_1})   Q_\lastfnc(\widetilde{z}_{a_{k+1}})  \left( \prod_{i \in \singleblks{\pi}} q_{i-1}(\widetilde{z}_{a_i}) \right)  \prod_{i \in V} (Q_{\blocks_i}(\widetilde{z}_{a_{\blocks_i}}))   \bigg| \bm A \right]  \left( \prod_{i \notin V} \mu_{\blocks_i} \right) \label{eq: qf_lemma_firstmom_Vexpansion}
    \end{align}
    We now apply Mehler's formula to estimate the above conditional expectations. We first check the conditions for Mehler's formula:
    \begin{enumerate}
        \item The random variables $\widetilde{\bm z}$ are marginally $\gauss{0}{1}$. Define $\bm \Sigma = \E [\widetilde{\bm z} \widetilde{\bm z}^\UT | \bm A]$. $\bm{\widetilde{z}}$ and are weakly correlated on the event $\mathcal{E}$ since:
        \begin{align*}
            \max_{i \neq j} |\Sigma_{ij}| & = \left| \frac{(\bm A \bm A^\UT)_{ij}/\kappa}{\sigma_i \sigma_j} \right| \\
            & = \left| \frac{(\bm\Psi)_{ij}/\kappa}{\sigma_i \sigma_j} \right| \\
            & \leq C\sqrt{\frac{\log^{3}(m)}{m \kappa^2}}, \; \text{for $m$ large enough},
        \end{align*}
        where $C$ denotes a universal constant. 
        \item Let $S \subset [m]$ with $|S| \leq k+2$. Let $\bm \Sigma_{S,S}$ denote the principal submatrix of $\bm \Sigma $ formed by picking rows and columns in $S$. Then by Gershgorin's Circle theorem, on the event $\mathcal{E}$,
        \begin{align*}
            \lambda_{\min}(\bm \Sigma) & \geq 1 - (k+1) \max_{i \neq j} | \Sigma_{ij}| \\
            & \geq 1 - C(k+1)\sqrt{\frac{\log^{3}(m)}{m \kappa^2}}\\
            & \geq \frac{1}{2}, \; \text{ for $m$ large enough}.
        \end{align*}
        \item Note that for $\xi \sim \gauss{0}{1}$, we have, 
        \begin{align*}
            &\E Q_\firstfnc(\xi) = 0, \; \E Q_\lastfnc(\xi) = 0 \; \text{ (Since they are odd functions, see \eqref{eq: q_block_first}, \eqref{eq: q_block_last})}, \\
            &\E q_{i-1}(\xi) = \E \xi q_{i-1}(\xi) = 0 \; \forall \; i \in \singleblks{\pi} \; \text{ (They are centered, even functions, see Def. \ref{def: alternating_product})}, \\
            & \E Q_{\blocks_i}(\xi) = \E \xi Q_{\blocks_i}(\xi) = 0 \; \forall \; i \; \in \; [|\pi| - |\singleblks{\pi}| - 2] \; \text{ (See \eqref{eq: q_block})} 
        \end{align*}
        Hence applying the first non-zero term in Mehler's Expansion (Proposition \ref{proposition: mehler}) of the conditional expectation:
        \begin{align*}
            \E\left[ Q_\firstfnc(\widetilde{z}_{a_1}) \cdot  Q_\lastfnc(\widetilde{z}_{a_{k+1}}) \cdot \left( \prod_{i \in \singleblks{\pi}} q_{i-1}(\widetilde{z}_{a_i}) \right) \cdot \prod_{i \in V} (Q_{\blocks_i}(\widetilde{z}_{a_{\blocks_i}}))   \bigg| \bm A \right]
        \end{align*}has total weight $\|\bm w\|$ given by:
        \begin{align*}
            \|\bm w\| & \geq  \frac{1 + 1 + 2 |\singleblks{\pi}| + 2 |V|}{2} = 1 + |\singleblks{\pi}| + |V|.
        \end{align*}
    \end{enumerate}
     Hence, by Proposition \ref{proposition: mehler} we have,
        \begin{align}
             &\Indicator{\mathcal{E}} \cdot \left| \E\left[ Q_\firstfnc(\widetilde{z}_{a_1}) \cdot  Q_\lastfnc(\widetilde{z}_{a_{k+1}}) \cdot \left( \prod_{i \in \singleblks{\pi}} q_{i-1}(\widetilde{z}_{a_i}) \right) \cdot \prod_{i \in V} (Q_{\blocks_i}(\widetilde{z}_{a_{\blocks_i}}))   \bigg| \bm A \right] \right|  \nonumber \\& \hspace{3cm} \leq  C(\altprod) (\max_{i \neq j} | \Sigma_{i,j} |)^{1 + |\singleblks{\pi}| + |V|} 
             \leq C(\altprod) \cdot \left( \frac{\log^2(m)}{m \kappa^2} \right)^{\frac{1 + |\singleblks{\pi}| + |V|}{2}}, \label{eq: mehler_conclusion_notimp}
        \end{align}
        where $C(\altprod)$ denotes a finite constant depending only on the functions $q_{1:k}$. When $V = \emptyset$ we will also need to estimate the leading order term more accurately. Define,
        \begin{align*}
            \weightedGnum{\pi}{1} & \explain{def}{=} \left\{\bm w \in \weightedG{k+1}: \degree_1(\bm w) = 1, \; \degree_{k+1}(\bm w) = 1, \; \degree_i(\bm w) = 2 \; \forall \; i \; \in \; \singleblks{\pi}, \right.\\ &\hspace{6cm} \left.\degree_i(\bm w) = 0 \; \forall \; i \; \notin \; \{1,k+1\} \cup \singleblks{\pi} \right\}.
        \end{align*}
    By Mehler's formula, on the event $\mathcal{E}$, we have:
    \begin{align*}
        & \left| \E\left[ Q_\firstfnc(\widetilde{z}_{a_1}) \cdot  Q_\lastfnc(\widetilde{z}_{a_{k+1}}) \cdot \left( \prod_{i \in \singleblks{\pi}} q_{i-1}(\widetilde{z}_{a_i}) \right)   \bigg| \bm A \right] - \sum_{\bm w \in \weightedGnum{\pi}{1}} \hat{\coeff}(\bm w,\bm \Psi) \cdot \matmom{\bm \Psi}{\bm w}{\pi}{\bm a}  \right| \\
             &\hspace{8cm }  \leq C(\altprod) \cdot \left( \frac{\log^3(m)}{m \kappa^2} \right)^{\frac{2 + |\singleblks{\pi}|}{2}},
    \end{align*}
    where,
    \begin{align*}
        \hat{\coeff}(\bm w, \bm \Psi) & = \frac{1}{\bm w!} \cdot \left(\prod_{i=1}^{k+1} \frac{1}{\sigma_{a_i}^{\degree_i(\bm w)}} \right) \cdot \left(\hat{Q}_\firstfnc(1) \hat{Q}_\lastfnc(1) \prod_{i \in \singleblks{\pi}} \hat{q}_{i-1}(2)   \right) \frac{1}{\kappa^{\|\bm w\|}}, 
    \end{align*}
    and $\matmom{\bm \Psi}{\bm w}{\pi}{\bm a}$ are matrix moments as defined in Definition \ref{def: matrix moment}. Note that the coefficients $\hat{\coeff}(\bm w, \bm \Psi)$ depend on $\bm \Psi$ since,
    \begin{align*}
        \sigma_i^2 & = 1 + \frac{\bm \Psi_{ii}}{\kappa},
    \end{align*}
    but we can remove this dependence. On the event $\mathcal{E}$, note that,
    \begin{align*}
        \max_{i \in [m]} |\sigma_{ii}^2 - 1| & \leq C \sqrt{\frac{\log^3(m)}{m \kappa^2}}.
    \end{align*}
    Hence defining:
    \begin{align*}
        \hat{\coeff}(\bm w, \pi) & = \frac{1}{\bm w!} \cdot  \left(\hat{Q}_\firstfnc(1) \hat{Q}_\lastfnc(1) \prod_{i \in \singleblks{\pi}} \hat{q}_{i-1}(2)   \right) \frac{1}{\kappa^{\|\bm w\|}}, 
    \end{align*}
    we have, for $m$ large enough and on the event $\mathcal{E}$,
    \begin{align*}
        |\hat{\coeff}(\bm w, \pi) - \hat{\coeff}(\bm w, \bm \Psi)| & \leq  C_k \sqrt{\frac{\log^3(m)}{m \kappa^2}}.
    \end{align*}
    Furthermore, we have the estimate,
    \begin{align*}
        |\matmom{\bm \Psi}{\bm w}{\pi}{\bm a}| & \leq (\max_{i,j} |\Psi_{ij}| )^{\|\bm w\|_1} \\
     & \explain{(a)}{\leq} C\left( \frac{\log^3(m)}{m \kappa^2} \right)^{\frac{1 + |\singleblks{\pi}|}{2}},
    \end{align*}
    where in the step (a), we used the definition of the event $\mathcal{E}$ in \eqref{eq: good_event_qf_firstmom} and the fact that $\|\bm w\| = 1 + |\singleblks{\pi}|$ for any $\bm w \in \weightedGnum{\pi}{1}$. Hence we obtain, on the event $\mathcal{E}$,
    \begin{align*}
        &\left| \E\left[ Q_\firstfnc(\widetilde{z}_{a_1}) \cdot  Q_\lastfnc(\widetilde{z}_{a_{k+1}}) \cdot \left( \prod_{i \in \singleblks{\pi}} q_{i-1}(\widetilde{z}_{a_i}) \right)   \bigg| \bm A \right] - \sum_{\bm w \in \weightedGnum{\pi}{1}} \hat{\coeff}(\bm w, \pi) \cdot \matmom{\bm \Psi}{\bm w}{\pi}{\bm a}  \right| \nonumber \\
             &\hspace{8cm }  \leq C(\altprod) \cdot \left( \frac{\log^3(m)}{m \kappa^2} \right)^{\frac{2 + |\singleblks{\pi}|}{2}}.
    \end{align*}
    Combining this estimate with \eqref{eq: qf_lemma_firstmom_Vexpansion} and \eqref{eq: mehler_conclusion_notimp} gives us:
    \begin{subequations}
    \begin{align*}
        &\Indicator{\mathcal{E}} \cdot \left| \E[ \widetilde{z}_{a_1} q_1(\widetilde{z}_{a_2}) q_2(\widetilde{z}_{a_3}) \cdots q_{k-1}(\widetilde{z}_{a_k})  \widetilde{z}_{a_{k+1}} |  \bm A ] - \sum_{\bm w \in \weightedGnum{\pi}{1}} {\coeff}(\bm w, \pi) \cdot \matmom{\bm \Psi}{\bm w}{\pi}{\bm a}  \right| \nonumber \\
             &\hspace{8cm }  \leq C(\altprod) \cdot \left( \frac{\log^3(m)}{m \kappa^2} \right)^{\frac{2 + |\singleblks{\pi}|}{2}},
    \end{align*}
    where
    \begin{align*}
        {\coeff}(\bm w, \pi) & = \frac{1}{{\kappa^{\|\bm w\|}} \bm w!}  \cdot \left(\hat{Q}_\firstfnc(1) \hat{Q}_\lastfnc(1) \prod_{i \in \singleblks{\pi}} \hat{q}_{i-1}(2)   \right)  \cdot \left( \prod_{i \in [|\pi| - |\singleblks{\pi}| - 2]} \mu_{\blocks_i} \right)
    \end{align*}
    \begin{align*}
            \weightedGnum{\pi}{1} & \explain{def}{=} \left\{\bm w \in \weightedG{k+1}: \degree_1(\bm w) = 1, \; \degree_{k+1}(\bm w) = 1, \; \degree_i(\bm w) = 2 \; \forall \; i \; \in \; \singleblks{\pi}, \right. \nonumber\\ &\hspace{6cm} \left.\degree_i(\bm w) = 0 \; \forall \; i \; \notin \; \{1,k+1\} \cup \singleblks{\pi} \right\},
        \end{align*}
    %\label{eq: mehler_conclusion_imp}
    \end{subequations}
    and $C(\altprod)$ denotes a constant depending only on the functions appearing in $\altprod$ and $k$. This was precisely the claim of Lemma \ref{lemma: qf_mehler_conclusion}.
\end{proof}
%%%%%%%%%%%%%%%%%%%%%%%%%%%%%%%%%%%%%%%%%%%%%%%%%%%%%%%%%%%%%
%%%%%%%%%%%%%%%%%%%%%%%%%%%%%%%%%%%%%%%%%%%%%%%%%%%%%%%%%%%%%
%%%%%%%%%%%%%%%%%%%%%%%%%%%%%%%%%%%%%%%%%%%%%%%%%%%%%%%%%%%%%

%%%%%%%%%%%%%%%%%%%%%%%%%%%%%%%%%%%%%%%%%%%%%%%%%%%%%%%%%%%%%%%
%%%%%%%%%%%%%%%%%%% APPENDIX-MEHLER%%%%%%%%%%%%%%%%%%%%%%%%%%%%
%%%%%%%%%%%%%%%%%%%%%%%%%%%%%%%%%%%%%%%%%%%%%%%%%%%%%%%%%%%%%%
\section{Proof of Proposition \ref{proposition: mehler}}
\label{appendix: mehler}
\begin{proof}[Proof of Proposition \ref{proposition: mehler}]
Let $\gpdf{\bm z}{\bm \Sigma}$ denote the density of a $k$ dimensional zero mean Gaussian vector with positive definite covariance matrix $\bm \Sigma$ i.e. $\bm z \sim \gauss{\bm 0}{\bm \Sigma}$.  Suppose that $\Sigma_{ii} = 1 \; \forall  \; i \in [k]$. In this situation \citet{slepian1972symmetrized} has found an explicit expression for the Taylor series expansion of $\gpdf{\bm z}{\bm \Sigma}$ around $\bm \Sigma = \bm I_k$ given by:
\begin{align*}
    {\gpdf{\bm z}{\bm \Sigma}} & = \sum_{\bm w \in \weightedG{k}} \frac{D^{\bm w}_{\bm \Sigma} \;  \psi(\bm z; \bm I_k)}{\bm w !} \cdot \left( \prod_{i<j} \Sigma_{ij}^{w_{ij}} \right),
\end{align*}
where $D^{\bm w}_{\bm \Sigma} \; \psi(\bm z; \bm I_k)$ denotes the derivative:
\begin{align*}
    D^{\bm w}_{\bm \Sigma} \;  \psi(\bm z; \bm I_k) & \explain{def}{=} \frac{\partial^{\|\bm w\|}}{\partial \Sigma_{12}^{w_{12}} \; \partial \Sigma_{13}^{w_{13}}\cdots \partial \Sigma_{23}^{w_{23}}\; \partial \Sigma_{24}^{w_{24}} \cdots \partial \Sigma_{k-1,k}^{w_{k-1,k}}} \;  \psi(\bm z; \bm \Sigma) \bigg|_{\bm \Sigma = \bm I_k} \\&=  \left(\prod_{i=1}^k   H_{\degree_i(\bm w)}(z_i) \right) \cdot \gpdf{\bm z}{\bm  I_k}.
\end{align*}
We intend to integrate the Taylor series for $\gpdf{\bm z}{\bm \Sigma}$ to obtain the expansion for the expectation in Proposition \ref{proposition: mehler}. In order to do so we need to understand the truncation error in the Taylor Series. By Taylors Theorem, we know that:
\begin{align}
    \gpdf{\bm z}{\bm \Sigma} -  \sum_{\bm w \in \weightedG{k}: \|\bm w\| \leq t} \frac{D^{\bm w}_{\bm \Sigma} \; \psi(\bm z; \bm I_k)}{\bm w !} \cdot \left( \prod_{i<j} \Sigma_{ij}^{w_{ij}} \right)  & = \sum_{\bm w \in \weightedG{k}: \|\bm w\| = t+1} \frac{D^{\bm w}_{\bm \Sigma} \; \gpdf{\bm z}{\bm \Sigma_\gamma}}{\bm w!} \cdot \bm \Sigma^{\bm w}, \label{eq: taylors_thm}
\end{align}
where $\bm \Sigma_\gamma = \gamma \bm \Sigma + (1-\gamma) \bm I_k$ for some $\gamma \in (0,1)$. \citeauthor{slepian1972symmetrized} has further showed the following remarkable identity:
\begin{align*}
    D^{\bm w}_{\bm \Sigma} \; \gpdf{\bm z}{\bm \Sigma} & = \frac{\partial^{2 \|\bm w\|}}{\partial z_1^{\degree_1(\bm w)} \; \partial z_2^{\degree_2(\bm w)} \cdots \partial z_k^{\degree_k(\bm w)}} \; \gpdf{\bm z}{\bm \Sigma}.
\end{align*}
An inductive calculation shows that the ratio:
\begin{align*}
     \frac{1}{\gpdf{\bm z}{\bm \Sigma}}\; \frac{\partial^{2 \|\bm w\|}}{\partial z_1^{\degree_1(\bm w)} \; \partial z_2^{\degree_2(\bm w)} \cdots \partial z_k^{\degree_k(\bm w)}} \; \gpdf{\bm z}{\bm \Sigma},
\end{align*}
is a polynomial of degree $4 \|\bm w\|$ in the variables $z_1,z_2 \dots z_k, \{(\bm \Sigma^{-1})_{ij}\}_{i<j}$. Hence:
\begin{align*}
    &\left| \frac{1}{\gpdf{\bm z}{\bm \Sigma}}\; \frac{\partial^{2 \|\bm w\|}}{\partial z_1^{\degree_1(\bm w)} \; \partial z_2^{\degree_2(\bm w)} \cdots \partial z_k^{\degree_k(\bm w)}} \; \gpdf{\bm z}{\bm \Sigma}\right|  \leq \\ & \hspace{4cm}  C_{\|\bm w\|} \cdot ( 1 + \sum_{i < j} |(\bm \Sigma^{-1})_{ij}|^{4 \|\bm w\|} + \sum_{i=1}^k |z_i|^{4\|\bm w\|}),
\end{align*}
where $C_{\|\bm w\|}$ denotes a constant depending only on $\|\bm w\|$. Observing that:
\begin{align*}
    (\bm \Sigma^{-1})_{ij} & \leq \|\bm \Sigma^{-1}\|_{\op}  = \frac{1}{\lambda_{\min}(\bm \Sigma)} < \infty.
\end{align*}
This gives us:
\begin{align*}
      \left| \frac{1}{\gpdf{\bm z}{\bm \Sigma}}\; \frac{\partial^{2 \|\bm w\|}}{\partial z_1^{\degree_1(\bm w)} \; \partial z_2^{\degree_2(\bm w)} \cdots \partial z_k^{\degree_k(\bm w)}} \; \gpdf{\bm z}{\bm \Sigma}\right| & \leq C_{\|\bm w\|}  \left( 1 + \frac{k^2}{\lambda_{\min}^{4\|\bm w\|}(\bm \Sigma)} + \sum_{i=1}^k |z_i|^{4\|\bm w\|}\right).
\end{align*}
Substituting this estimate in \eqref{eq: taylors_thm} gives us:
\begin{align*}
    &\left|\gpdf{\bm z}{\bm \Sigma} -  \sum_{\bm w \in \weightedG{k}: \|\bm w\| \leq t} \frac{D^{\bm w}_{\bm \Sigma} \; \psi(\bm z; \bm I_k)}{\bm w !} \cdot \bm \Sigma^{\bm w}  \right|  \\& \hspace{2cm} \leq  C_{t,k} \cdot \left( 1 + \frac{k^2}{\lambda_{\min}^{4t+4}(\bm \Sigma_\gamma)} + \sum_{i=1}^k |z_i|^{4t+4}\right)\cdot \left( \max_{i \neq j} |\Sigma_{ij}| \right)^{t+1} \cdot \gpdf{\bm z}{\bm \Sigma_{\gamma}}.
\end{align*}
Note that $\lambda_{\min}(\bm \Sigma_\gamma) = \gamma + (1-\gamma) \lambda_{\min}(\bm \Sigma) \geq \min(1,\lambda_{\min}(\bm \Sigma))$. Hence,
\begin{align*}&\left|\gpdf{\bm z}{\bm \Sigma} -  \sum_{\bm w \in \weightedG{k}: \|\bm w\| \leq t} \frac{D^{\bm w}_{\bm \Sigma} \; \psi(\bm z; \bm I_k)}{\bm w !} \cdot \bm \Sigma^{\bm w}  \right|  \\& \hspace{1cm} \leq  C_{t,k} \cdot \left( 1 + \frac{k^2}{\min(\lambda_{\min}^{4t+4}(\bm \Sigma),1)} + \sum_{i=1}^k |z_i|^{4t+4}\right)\cdot \left( \max_{i \neq j} |\Sigma_{ij}| \right)^{t+1} \cdot \gpdf{\bm z}{\bm \Sigma_{\gamma}}.
\end{align*}
Using this expansion to compute the expectation of $\prod_{i=1}^k f_i(z_i)$ we obtain:
\begin{align*}
   \left| \E \left[ \prod_{i=1}^k f_i(z_i) \right] - \sum_{\substack{\bm w \in \weightedG{k}\\ \|\bm w\| \leq t}}   \left( \prod_{i=1}^k \hat{f}_i(\degree_i(\bm w)) \right) \cdot \frac{\bm \Sigma^{\bm w}}{\bm w!} \right| & \leq C \left(1 + \frac{1}{\lambda_{\min}^{4t+4}(\bm \Sigma)} \right)  \left( \max_{i \neq j} |\Sigma_{ij}| \right)^{t+1},
\end{align*}
where $C=C_{t,k,f_{1:k}}$ denotes a constant depending only on $t,k$ and the functions $f_{1:k}$.
In obtaining the above estimate we use the fact that since the functions $f_i$ have polynomial growth and marginally $z_i \sim \gauss{0}{1}$ under the measure $\gauss{\bm 0}{\bm \Sigma_{\gamma}}$ (since $(\Sigma_{\gamma})_{ii} = 1$) we have,
\begin{align*}
    \E_{\bm z \sim \gauss{\bm 0}{\bm \Sigma_\gamma}} \left[ |z_i|^{4t+4} \prod_{j=1}^k |f_j(z_j)| \right] & \leq \sum_{j=1}^k \E_{\bm z \sim \gauss{\bm 0}{\bm \Sigma_\gamma}} \left[ |z_i|^{4t+4} |f_j(z_j)|^k \right]  = C_{t,k,f_{1:k}} < \infty.
\end{align*}
\end{proof}

% We also version Proposition \ref{proposition: mehler} for complex Gaussians. The proof is analogous.

% \begin{prop} \label{proposition: mehler_complex} Let $f_1,f_2 \dots f_k: \C \rightarrow \R$ be $k$ arbitrary functions whose absolute value can be upper bounded by a polynomial. Consider a $k$ dimensional complex Gaussian vector $\bm z = \bm u + \i \bm v \in \C^k$ where $(\bm u, \bm v) \sim \gauss{\bm 0}{\bm \Sigma}$ is a 2k-dimensional real Gaussian vector  such that $\Sigma_{ii} = 1$ for all $i \in [2k]$.  
% Then,
% \begin{align*}
%   \left| \E \left[ \prod_{i=1}^k f_i(z_i) \right] - \sum_{\substack{\bm w \in \weightedG{2k}\\ \|\bm w\| \leq t}}   \left( \prod_{i=1}^k \hat{f}_i(\degree_i(\bm w),d_{i+k}(\bm w)) \right) \cdot \frac{\bm \Sigma^{\bm w}}{\bm w!} \right| & \leq C_{t,k, f_{1:k}} \left(1 + \frac{1}{\lambda_{\min}^{4t+4}(\bm \Sigma)} \right) \cdot \left( \max_{i \neq j} |\Sigma_{ij}| \right)^{t+1},
% \end{align*}
% where:
% \begin{align*}
%     \hat{f}_i(j,j^\prime) \explain{def}{=} \E f(Z + \i Z^\prime) H_j(Z) H_j^\prime(Z^\prime), \; Z, Z^\prime \explain{i.i.d.}{\sim} \gauss{0}{1}, 
% \end{align*}
% and the other notations are as in Proposition \ref{proposition: mehler}.
% \end{prop}
%%%%%%%%%%%%%%%%%%%%%%%%%%%%%%%%%%%%%%%%%%%%%%%%%%%%%%%%%%%%%
%%%%%%%%%%%%%%%%%%%%%%%%%%%%%%%%%%%%%%%%%%%%%%%%%%%%%%%%%%%%%
%%%%%%%%%%%%%%%%%%%%%%%%%%%%%%%%%%%%%%%%%%%%%%%%%%%%%%%%%%%%%

\section{Derivation of Proposition \ref{prop: SE_rotationally_invariant}} \label{sec:appendix-SE-derivation}
In this section, we sketch how Proposition \ref{prop: SE_rotationally_invariant} can be derived by instantiating \citet[Theorem 1]{venkataramanan2021estimation} to our setup. Recall that the linearized AMP iterations are given by:
\begin{align}\label{eq:LAMP-recall}
     \hat{\bm z}^{(t+1)} &:= \left( \frac{1}{\kappa} \bm A \bm A^\UT - \bm I \right) \cdot  q_t(\bm Z) \cdot \hat{\bm z}^{(t)},
\end{align}
where, 
\begin{align} \label{eq:q-func-def}
    q_t(z) = \eta_t(|z|) - \E[\eta_t(|z|)],
\end{align} and $q_t(\bm Z) = \diag{q(z_1), q(z_2), \dotsc, q(z_m)}$ where $\bm z$ is the vector of signed measurements. Since we assume that the function $\eta$ is bounded and Lipschitz, $q_t$ is also a bounded Lipschitz function. 

We will obtain a state evolution result for iteration \eqref{eq:LAMP-recall}, we will relate it to an instance of a much more general class of approximate message passing algorithms studied in the work of \citet{venkataramanan2021estimation}. However, a minor difficulty is that the sensing matrix in our setup is obtained by picking $n$ columns of a $m \times m$ Haar matrix uniformly at random:
\begin{align*}
    \bm A = \bm H \bm P \bm S.
\end{align*}
In particular, $\bm A$ is left rotationally invariant but not right rotationally invariant, whereas the result in \citep{venkataramanan2021estimation} requires both left and right rotational invariance of the sensing matrix. The reason why this doesn't pose any difficulties is that it is easy to check that for any $\bm V \in \mathbb{O}(n)$, the sequence of iterates ${\hat{\bm z}}^{(t)}$ generated when the signal is $\bm x$ and the sensing matrix is $\bm A$ is identical to the sequence of iterates generated when the signal is $\bm V \bm x$ and the sensing matrix is $\bm A \bm V^\UT$. Hence by taking $\bm V \sim \unif{\mathbb{O}(n)}$, one obtains the desired right rotational invariance of the sensing ensemble. Next, we relate the iteration \eqref{eq:LAMP-recall} to the following iteration covered by the results in \citep{venkataramanan2021estimation}. We will closely follow the notation in \citep{venkataramanan2021estimation} for the convenience of the reader. Consider an algorithm that maintains two iterates $\bm x^{(t)}$ and $\bm r^{(t)}$ which are updated as follows:
\begin{subequations}\label{eq:venkat-AMP}
\begin{align}
    \bm x^{(t+1)} &= \bm A^\UT h_{t}(\bm r^{t}, \bm z) - \bm \epsilon^{(t+1)} \\
    \bm r^{(t+1)} &= \bm A  \bm x^{(t+1)} - \kappa \cdot h_t(\bm r^{(t)}, \bm z) - \bm \delta^{(t+1)}.
\end{align}
\end{subequations}
In the above display, $\kappa = n/m$ and the function $h_t : \R^2 \rightarrow \R$ acts entry-wise on the vectors $\bm r^{t}, \bm z$. We set $h_t(r,z) = r q_t(z)/\kappa$ to caliberate the iteration \eqref{eq:venkat-AMP} with \eqref{eq:LAMP-recall}. The vectors $\bm \epsilon^{(t+1)}$ and $\bm \delta^{(t+1)}$ are given by:
\begin{align*}
    \bm \epsilon^{(t)} & = \sum_{i=1}^{t-1}  \beta_{t,i} \cdot  \bm x^{(i)}, \\
    \bm \delta^{(t)} & = \sum_{i=1}^{t-1} \alpha_{t,i} \cdot h_{i-1}(\bm r^{(i-1)}) + (\alpha_{t,t} - \kappa) \cdot  h_{t-1}(\bm r^{(t-1)},\bm z),
\end{align*}
where the de-biasing coefficients $\alpha_{t,i}, \beta_{t,i}$ are as given in \citep[Equations 3.6-3.7]{venkataramanan2021estimation}. In order to relate iteration \eqref{eq:LAMP-recall} to the iteration \eqref{eq:venkat-AMP} we can combine the two iterations in \eqref{eq:venkat-AMP} to obtain:
\begin{align*}
    \bm r^{(t+1)} & = (\bm A \bm A^\UT - \kappa \bm I_{m}) \cdot h_t(\bm r^{(t)}, \bm z) - \bm A \bm \epsilon^{(t+1)} - \bm \delta^{(t+1)} \\
    & =  \left( \frac{1}{\kappa} \bm A \bm A^\UT - \bm I \right) \cdot  q_t(\bm Z) \cdot {\bm r}^{(t)} -  \bm A \bm \epsilon^{(t+1)} - \bm \delta^{(t+1)}. 
\end{align*}
We can now recursively control the error between the iterates $\| \bm r^{(t+1)} -  \hat{\bm z}^{(t+1)}\|_2$:
\begin{align}\label{eq:error-recursion}
    \| \bm r^{(t+1)} -  \hat{\bm z}^{(t+1)}\|_2 & \leq \frac{\|q_t\|_\infty}{\kappa} \cdot  \| \bm r^{(t)} -  \hat{\bm z}^{(t)}\|_2 + \|\bm \epsilon^{(t+1)}\|_2 + \|\bm \delta^{(t+1)}\|_2.
\end{align}
Using the formula for the de-biasing coefficients $\alpha_{t,i}, \beta_{t,i}$ are as given in \citep[Equations 3.6-3.7]{venkataramanan2021estimation} and the fact that:
\begin{align*}
    \frac{1}{m} \sum_{i=1}^m \partial_r h_t(r_i^{(t)},z_i) & = \frac{1}{m\kappa} \sum_{i=1}^m q_t(z_i) \explain{P}{\rightarrow} \frac{\E[q_t(Z)]}{\kappa} \explain{\eqref{eq:q-func-def}}{=} 0,
\end{align*}
we obtain that:
\begin{align*}
    \beta_{t,i} &\explain{P}{\rightarrow} 0\; \forall \; i \; \leq \; t-1, \\
    \alpha_{t,i} &\explain{P}{\rightarrow} 0\; \forall \; i \; \leq \; t-1, \\
     \alpha_{t,t} &\explain{P}{\rightarrow} \kappa.
\end{align*}
Which immediately yields for any $t \in \N$,
\begin{align*}
     \|\bm \epsilon^{(t)}\|^2_2/m \explain{P}{\rightarrow} 0, \; \|\bm \delta^{(t)}\|^2_2/m \explain{P}{\rightarrow} 0. 
\end{align*}
Combining this with \eqref{eq:error-recursion} gives us:
\begin{align*}
     \frac{\| \bm r^{(t+1)} -  \hat{\bm z}^{(t+1)}\|^2_2}{m} &\explain{P}{\rightarrow} 0.
\end{align*}
Consequently, the state evolution for the iteration $\bm r^{(t+1)}$ given in \citep[Theorem 1]{venkataramanan2021estimation} also holds for $\hat{\bm z}^{(t+1)}$, which gives us the claim of Proposition \ref{prop: SE_rotationally_invariant}. 
%%%%%%%%%%%%%%%%%%%%%%%%%%%%%%%%%%%%%%%%%%%%%%%%%%%%%%%%%%%%%
%%%%%%%%%%%%%%%%%% APPENDIX-MISC%%%%%%%%%%%%%%%%%%%%%%%%%%%%%
%%%%%%%%%%%%%%%%%%%%%%%%%%%%%%%%%%%%%%%%%%%%%%%%%%%%%%%%%%%%%
\section{Some Miscellaneous Facts}

\begin{fact}[Hanson-Wright Inequality \citep{rudelson2013hanson}] \label{fact: hanson_wright} Let $\bm x = (x_1, x_2 \dots,x_n) \in \R^n$ be a random vector with independent 1-subgaussian, zero mean components. Let $\bm A$ be an $n \times n$ matrix. Then, for every $t \geq 0$,
\begin{align*}
    \P \left( | \bm x^\UT \bm A \bm x - \E \bm x^\UT \bm A \bm x| > t \right) & \leq 2 \exp \left( - c \min \left( \frac{t^2}{\|\bm A\|_{\mathsf{Fr}}^2}, \frac{t}{\|\bm A\|_\op} \right) \right).
\end{align*}
\end{fact}

\begin{fact}[Gaussian Poincare Inequality] \label{fact: gaussian_poincare}Let $\bm x \sim \gauss{0}{\bm I_n}$. Then, for any $L$-Lipchitz function $f: \R^n \rightarrow \R$ we have,
\begin{align*}
    \var(f(\bm x)) & \leq L^2.
\end{align*}
\end{fact}

\begin{fact}[Moments of a Random Unit vector, Lemma 2.22 \& Proposition 2.5 of \citep{meckes2019random}] \label{fact: unit_vector_moments} Let $\bm x \sim \unif{\mathbb{S}_{n-1}}$. Let $i,j,k,\ell$ be distinct indices. Then:
\begin{align*}
    \E x_i^4 & = \frac{3}{n(n+2)}, \; \E x_i^2 x_j^2  = \frac{n+1}{n(n-1)(n+2)} \; \E x_i^3 x_j = 0 \; \E x_i x_j x_k^2 = 0, \; \E x_i x_j x_k x_l = 0.
\end{align*}
Furthermore, there exists a universal constant $C$ such that, for any $t \in \N$:
\begin{align*}
    \E |x_i|^t & \leq \left( \frac{Ct}{m} \right)^{\frac{t}{2}}.
\end{align*}
\end{fact}

\begin{fact}[Concentration on the Sphere, \citet{ball1997elementary}] \label{fact: concentration_sphere} Let $\bm x \sim \unif{\mathbb{S}_{n-1}}$. Then
\begin{align*}
    \P \left( |x_1| \geq \epsilon \right) & \leq 2 e^{-n \epsilon^2/2}.
\end{align*}
\end{fact}

\begin{fact}[$\ell_\infty$ norm of a random unit vector] \label{fact: infty_norm_uniform_unit}  $\bm x \sim \unif{\mathbb{S}_{n-1}}$. Then
\begin{align*}
    \E \|\bm x\|_\infty^t & \leq   \left( \frac{C  \log(n) }{n} \right)^{\frac{t}{2}},
\end{align*}
for a universal constant $C$.
\end{fact}
\begin{proof}
For a random unit vector we can control $\E \|\bm x\|_\infty^t$ as follows. Let $q \in \N$ be a parameter to be set suitably. Then,
\begin{align*}
    \E \|\bm x\|_\infty^t & = \left( \E \|\bm x\|_\infty^{qt} \right)^{\frac{1}{q}} \\
    & \leq \left( \sum_{i=1}^n \E |x_i|^{qt} \right)^{\frac{1}{q}} \\
    & \explain{(a)}{=} \left( n \E |x_1|^{qt} \right)^{\frac{1}{q}} \\
    & \explain{(b)}{=} n^{\frac{1}{q}} \cdot q^{\frac{t}{2}} \cdot  \left( \frac{C t}{n} \right)^{\frac{t}{2}} \\
    & \explain{(c)}{\leq} e^t \cdot (2 \log(n))^{\frac{t}{2}} \cdot  \left( \frac{C  }{n} \right)^{\frac{t}{2}}.
\end{align*}
In the step marked (a) we used the fact that the coordinates of a random unit vector are exchangeable, in (b) we used the fact that $u_1$ is $C/m$-subgaussian (see Fact \ref{fact: concentration_sphere}) and in (c) we set $q = \lfloor \frac{2\log(n)}{t} \rfloor$.
\end{proof}
\begin{fact}[Poincare Inequality for Haar Measure, \citet{gromov1983topological}]  \label{fact: poincare_haar}Consider the following setups:
\begin{enumerate}
    \item Let $\bm O \sim \unif{\O(m)}$ and $f: \R^{m \times m} \rightarrow \R$ be a function such that:
    \begin{align}
        f(\bm O) & = f(\bm O \bm D), \; \bm D = \diag{1,1,1, \dots, 1, \mathsf{sign}(\det(\bm O))},  \label{eq: condition_poincare}
    \end{align}
    then,
    \begin{align*}
        \var(f(\bm O )) & \leq \frac{8}{m} \cdot \E \|\nabla f(\bm O)\|_\fr^2.
    \end{align*}
    for any $m \geq 4$.
    \item  Let $\bm O \sim \unif{\mathbb{U}(m)}$ and $f: \mathbb{C}^{m \times m} \rightarrow \R$. Then,
    \begin{align*}
        \var(f(\bm O )) & \leq \frac{8}{m} \cdot \E \|\nabla f(\bm O)\|_\fr^2.
    \end{align*}
\end{enumerate}
\end{fact}
\begin{proof} This result is due to \citet{gromov1983topological}. Our reference for these inequalities was the book of \citet{meckes2019random}. Theorem 5.16 of \citeauthor{meckes2019random} shows that Haar measures on $\mathbb{SO}(m), \mathbb{U}(m)$ satisfy Log-sobolev inequality with constant $8/m$. It is well known that Log-Sobolev Inequality implies the Poincare Inequality (see for e.g.  Lemma 8.12 in \citet{van2014probability}). Note that, in the real case we only obtain the Poincare inequality for the Haar measure on $\mathbb{SO}(m)$, condition \eqref{eq: condition_poincare} ensures the result still holds for  $\bm O \sim \unif{\O(m)}$.
\end{proof}

\begin{fact}[Continuity of Matrix Square Root {\citep[Lemma 2.2]{schmitt1992perturbation}}] \label{fact: matrix_sqrt} For any two symmetric positive semi-definite matrices $\bm M_1, \bm M_2$ we have,
\begin{align*}
    \|{\bm M}_1^{\frac{1}{2}} - \bm M_2^{\frac{1}{2}} \|_{\op} & \leq \frac{ \|{\bm M}_1 - \bm M_2 \|_{\op}}{\sqrt{\lambda_{\min}(\bm M_1)}}.
\end{align*}
\end{fact}%%%%%%%%%%%%%%%%%%%%%%%%%%%%%%%%%%%%%%%%%%%%%%%%%%%%%%%%%%%%%
%%%%%%%%%%%%%%%%%%%%%%%%%%%%%%%%%%%%%%%%%%%%%%%%%%%%%%%%%%%%%%
%%%%%%%%%%%%%%%%%%%%%%%%%%%%%%%%%%%%%%%%%%%%%%%%%%%%%%%%%%%%%

\end{document}